\documentclass[11pt]{article}
\usepackage{bbm}
\usepackage[top=1in, bottom=1.5in]{geometry}
\usepackage{color}
\usepackage{amsfonts}
\usepackage{latexsym, amssymb, amsmath, amscd, amsthm, amsxtra}
\usepackage{mathtools}
\usepackage{enumerate}
\usepackage[all]{xy}
\usepackage{mathrsfs}
\usepackage{fancyhdr}
\usepackage{listings}
\usepackage{hyperref}

\usepackage{algorithm}
\usepackage{algpseudocode}
\usepackage{framed}
\usepackage{lipsum}

\hypersetup{
	colorlinks   = true, 
	urlcolor     = blue, 
	linkcolor    = blue, 
	citecolor   = red 
}

\newtheorem{theorem}{Theorem}[section]
\newtheorem{lemma}[theorem]{Lemma}

\newtheorem{proposition}[theorem]{Proposition}

\theoremstyle{definition}
\newtheorem{definition}[theorem]{Definition}

\newtheorem{remark}[theorem]{Remark}

\usepackage{times}

\newcommand{\ER}{Erd\H{o}s-R\'enyi\ }

\newcommand{\dif}{\operatorname{d}\!}
\begin{document}

	\title{Optimal recovery of correlated \ER graphs}
	
	\author{Hang Du\footnote{Department of Mathematics, Massachusetts Institute of Technology.}}
	
	\date{\today}
	
	\maketitle
	
	\begin{abstract}
For two unlabeled graphs $G_1,G_2$ independently sub-sampled from an \ER graph $\mathbf G(n,p)$ by keeping each edge with probability $s$, we aim to recover \emph{as many as possible} of the corresponding vertex pairs. We establish a connection between the recoverability of vertex pairs and the balanced load allocation in the true intersection graph of $ G_1 $ and $ G_2 $. Using this connection, we analyze the partial recovery regime where $ p = n^{-\alpha + o(1)} $ for some $ \alpha \in (0, 1] $ and $ nps^2 = \lambda = O(1) $. We derive upper and lower bounds for the recoverable fraction in terms of $ \alpha $ and the limiting load distribution $ \mu_\lambda $ (as introduced in \cite{AS16}). These bounds coincide asymptotically whenever $ \alpha^{-1} $ is not an atom of $ \mu_\lambda $. Therefore, for each fixed $ \lambda $, our result characterizes the asymptotic optimal recovery fraction for all but countably many $ \alpha \in (0, 1] $.

	\end{abstract}
	
	\section{Introduction and main results}
	
Motivated by applications in diverse fields including social science~\cite{NS08, NS09}, computer vision~\cite{BBM05,CSS06}, computational biology~\cite{SXB08,VCP15}, and natural language processing~\cite{HNM05}, a growing body of research has focused on studying correlated networks from theoretical perspectives. Among the various models for such networks, the simplest and arguably the most canonical one---the \emph{correlated Erd\H{o}s--R\'enyi graph model} introduced in \cite{PG11}---has garnered the most attention and has been the most well-studied.

    Given a large integer $n$ with two parameters $p,s\in (0,1]$, the correlated \ER graph model is defined as follows. Let $G_0\sim \mathcal G(n,p)$ be an \ER graph on $[n]=\{1,\dots,n\}$ with edge density $p$, and let $G_1,G_2^*$ be two independent subgraphs of $G_0$ sampled by keeping each edge in $G_0$ with probability $s$ independently. Then, sample a uniform permutation $\pi^*\in \operatorname{S}_n$ and relabel the vertices of $G_2^*$ through $\pi^*$ to get another graph $G_2$ (see the more precise definition in Section~\ref{subsec-notation}). We write $\mathcal{P}^*=\mathcal P^*(n,p,s)$ for the joint distribution of $(\pi^*,G_1,G_2)$.
	The primary objective for this model is to estimate the matching $\pi^*$ from $(G_1,G_2)$ given $(\pi^*,G_1,G_2)\sim \mathcal{P}^*$. 
	
	In what follows, we focus on the case that there exist constants $0<\alpha\le 1$ and $\lambda>1$ such that $p=n^{-\alpha+o(1)}$ (where $o(1)$ is a term vanishes as $n\to\infty$) and $s$ is chosen to satisfy $nps^2=\lambda$. As will be discussed later, these assumptions naturally imply that it is impossible to recover a $1-o(1)$ fraction of the coordinates in $\pi^*$ \cite{WXY22}, while recovering a non-vanishing fraction of coordinates in $\pi^*$ remains feasible (at least for large $\lambda$ \cite{DD23b}). This raises a natural question: what is the optimal recovery fraction of $\pi^*$ across all estimators based on $(G_1,G_2)$? In this paper, we tackle this question and provide an almost tight characterization of the optimal recovery fraction.
	
	For two matchings $\pi_1,\pi_2\in \operatorname{S}_n$, we write $\operatorname{ov}(\pi_1,\pi_2)\equiv \#\{i\in [n]:\pi_1(i)=\pi_2(i)\}$ as the size of their overlap.
	Our main theorem gives upper and lower bounds for the asymptotic optimal value of $\operatorname{ov}(\widetilde{\pi},\pi^*)$ over all estimators $\widetilde{\pi}=\widetilde{\pi}(G_1,G_2)$. 
	
	\begin{theorem}\label{thm-main-ER}
		Fix two constants $\alpha,\lambda$ such that $0<\alpha\le 1$ and $\lambda>1$. Let $\mu_\lambda$ be the limiting load distribution introduced in \cite{AS16} (as formally defined in Proposition~\ref{prop-weak-convergence-of-empirical-measure}) and  $F_\lambda(x)\equiv\mu_\lambda\big((x,+\infty)\big)$. Let $p=p(n),s=s(n)$ satisfy that $p=n^{-\alpha+o(1)}$ as $n\to\infty$ and $nps^2=\lambda$ for any $n$. Then for any $\varepsilon>0$, the following hold for $(\pi^*,G_1,G_2)\sim \mathcal P^*=\mathcal P^*(n,p,s)$ as $n\to\infty$:
		\begin{itemize}
			\item [(i)] There exists an estimator $\widetilde{\pi}=\widetilde{\pi}(G_1,G_2)$, such that with $1-o(1)$ probability,
			\[
			\operatorname{ov}(\widetilde{\pi},\pi^*)\ge\big(F_\lambda(\alpha^{-1}+\varepsilon)-\varepsilon\big)n\,.
			\]
			\item [(ii)] There is no estimator $\widehat{\pi}=\widehat{\pi}(G_1,G_2)$ such that with non-vanishing probability,
			\[
			\operatorname{ov}(\widehat{\pi},\pi^*)\ge \big(F_\lambda(\alpha^{-1}-\varepsilon)+\varepsilon\big)n\,.
			\]
		\end{itemize}
	\end{theorem}
	
	Sending $\varepsilon$ down to $0$, the first item indicates that with high probability, one can correctly recover a $\big(F(\alpha^{-1})-o(1)\big)$-fraction of coordinates. Meanwhile, the second item implies it is almost impossible to recover a $\big(F_\lambda(\alpha^{-1}-)+o(1)\big)$-fraction of coordinates, where $F_\lambda(\cdot-)$ denotes the left-hand limit of $F_\lambda(\cdot)$. These two bounds  asymptotically align if and only if $F_\lambda$ is continuous at $\alpha^{-1}$ (i.e., $\mu_\lambda$ does not have an atom at $\alpha^{-1}$). Since $F_\lambda$ is a bounded monotone function, it can have at most countably many discontinuites. Consequently, for each fixed $\lambda$, our result is tight for all but countably many $\alpha$.
	
Let $ \varrho(\lambda) $ denote the rightmost point of $ \operatorname{supp}(\mu_\lambda) $ (see also Proposition~\ref{prop-maximal-density}). The theorem above implies that recovering a non-vanishing fraction of $ \pi^* $ is achievable when $ \varrho(\lambda) > \alpha^{-1} $, whereas the opposite holds when $ \varrho(\lambda) < \alpha^{-1} $. This is, in fact, the main result of \cite{DD23b}. Therefore, Theorem~\ref{thm-main-ER} extends the previous result and provides new insights for the case when $ \varrho(\lambda) > \alpha^{-1} $. Furthermore, we conjecture that for any $ \lambda > 1 $ and any point $ t $ with $ 1 \leq t < \varrho(\lambda) $, $ t $ is a discontinuity of $ F_\lambda $ if and only if $ t $ is a rational number. Assuming this conjecture holds, our result is tight for all irrational $ \alpha $ such that $ \varrho(\lambda)^{-1} < \alpha < 1 $.

	On the other hand, when $\alpha\in (\varrho(\lambda)^{-1},1]$ is a rational number, we have evidence suggesting that the optimal recovery fraction lies strictly between $F_\lambda(\alpha^{-1}-)$ and $F_\lambda(\alpha^{-1})$. However, obtaining an explicit characterization of the threshold appears challenging, as the optimal estimator likely requires a degree of ``random guessing''. To illustrate this, consider the simplest case when $s=1$ and $np=\lambda$ (so $\alpha=1$). In this setup, $G_1$ and $G_2$ are two isomorphic \ER graphs with edge density $\frac \lambda n$. Here, the best estimator of $\pi^*$ is nothing but a uniformly random isomorphism between $G_1$ and $G_2$. Therefore, the optimal recovery corresponds to the overlap of two automorphisms of $G_1\sim \mathbf G(n,\frac\lambda n)$, chosen independently and uniformly at random. 
 
 It can be shown that, with high probability, two random automorphisms of a sparse \ER graph agree on vertices within the giant $2$-core, while disagree almost entirely outside the giant component. This suggests the optimal recovery is bounded between the sizes of giant $2$-core and the giant component, aligning with results from Theorem~\ref{thm-main-ER} (see Lemma~\ref{lem-load-1-and->1} for details). However, for vertices located between the $2$-core and the giant, things become more complicated: there is always a random subset of these vertices lying in the overlap, while its size concentrates around a (positive) constant times $n$. Unfortunately, we could not find a simple and explicit expression for this constant. We leave determining the optimal value for general $\alpha$ (particularly the case  $\alpha=1$) as a direction for future work.

 \noindent\textbf{Acknowledgement}. The author would like to thank Brice Huang, Nike Sun and Jiaming Xu for many stimulating discussions on this problem and related topics, and thank Nike Sun for providing helpful comments on early drafts of this paper. 
	
	\subsection{Related works}
	
	\noindent\textbf{Correlated \ER graphs}. The statistical properties of the correlated \ER graph model have been extensively studied in literature \cite{PG11, CK16, CK17, CKMP20, GLM21, WXY22, WXY23, HM23, DD23a, DD23b}. Two fundamental questions regarding this model are correlation detection and matching recovery. Previously, the primary focus has been on identifying the thresholds at which detection and recovery become feasible. Two breakthrough works \cite{WXY22, WXY23} determined the informational threshold for detection and exact recovery, and they also characterized the partial recovery threshold up to multiplicative constants (here, partial recovery means recovering a non-vanishing fraction of coordinates of $\pi^*$). These constant factor gaps were later filled by \cite{DD23a,DD23b}. Hence, it is fair to say the community has now gained a comprehensive understanding of informational phase transitions in both detection and recovery. 
	
	Intriguingly, except for some very sparse cases, there is typically a sharp phase transition in detection, meaning that the total variational distance between the planted model and the null model (two independent \ER graphs with the same marginals) transitions abruptly from $0$ to $1$ \cite{WXY22, DD23a}. Similarly, in the dense regime where $p=n^{o(1)}$, the recovery problem exhibits an all-or-nothing phenomenon, in which the absence of partial recovery transitions sharply to the presence of almost-exact recovery (recovery of $1-o(1)$ fraction of coordinates) \cite{WXY23}. However, for the non-dense regime where $p=n^{-\alpha+o(1)}$ for some $\alpha\in (0,1]$, such an all-or-nothing behavior no longer exhibits. Specifically, in the regime (recalling $\varrho$ as defined in Proposition~\ref{prop-maximal-density})
	\[
	\varrho^{-1}(\alpha^{-1})<nps^2=O(1)\,,
	\]
	it is known from \cite{WXY23, DD23b} that only a fraction (strictly between $0$ and $1$) of coordinates of $\pi^*$ can be recovered. There have been not many results regarding the optimal recovery fraction in this regime. To our best knowledge, the only known result was from \cite{GLM21}, which established a similar version of Theorem~\ref{thm-main-ER}-(ii) for the special case $np=O(1)$ (correspondingly $\alpha=1$). While the proof therein relies on combinatorial arguments specific to $np=O(1)$, our approach is entirely different and extends their results to any $p$ in the non-dense regime. 
	
	Beyond information-theoretical results, progressively improved efficient algorithms have been proposed to tackle the detection and recovery tasks 
	\cite{FQMKJ16, SGE17, HKPRS17, HS17,  BCLS19, DCKG19, MX20, DMWX21, MRT23, GMS24, FMWX23a, BH22, GML24, MWXY21+, MWXY23, DL23}. Moreover, based on the low-degree hardness framework introduced in \cite{HS17,HKPRS17,Hop18,BHKKMA19}, \cite{DDL23+} provides an algorithmic upper bound for the detection and exact recovery problems, suggesting that known algorithms are nearing the potential algorithmic limit. Building on \cite{DDL23+} and a revised low-degree conjecture, a very recent work \cite{Li25} establishes the same algorithmic upper bound for partial recovery as well. We also remark that our result captures the information-theoretic limit, but the estimator $\widetilde{\pi}$ in Theorem~\ref{thm-main-ER}-(i) cannot be \emph{efficiently-computable}. 
	
Given the community’s relatively comprehensive understanding of the correlated \ER graph model, various extensions and generalizations have been explored, including models with inhomogeneous edge probabilities \cite{RS23, DFW23}, subgraph-specific correlation structures \cite{HSY24}, the correlated random geometric graph model \cite{WWXY22}, and the correlated stochastic block model \cite{CDGL24}. We believe that the ideas and techniques presented in this paper can be generalized to broader settings, such as the correlated stochastic block model.
	
	\noindent\textbf{Optimal partial recovery in statistical models}. Planted signal recovery problems have consistently attracted interest from researchers in this area. While much effort has focused on determining when an estimator can achieve either positive correlation or nearly full correlation with the ground truth (corresponding to partial recovery and almost exact recovery, respectively), understanding the optimal correlation among all estimators between these thresholds presents a more nuanced and often more challenging question.

    There are several successful examples for understanding the optimal partial recovery in high- dimensional statistical models, such as the low-rank matrix recovery problem \cite{LM19, MV21}. Most of these models have a background in spin glass theory, and as a result, the optimal partial recovery—often characterized via the minimal mean squared error (MMSE)—is governed by some order parameters of the corresponding spin glass model which can be calculated using physics ideas. However, the spin glass approaches are not universally applicable. For example, they do not apply (at least superficially) to models with essential combinatorial structures, such as the stochastic block model and the correlated \ER graph model, for which the pictures are much vaguer. To our best knowledge, the known results regarding optimal recovery in combinatorial statistical models are largely restricted to the stochastic block model \cite{MNS16,CS20}. The optimal partial recovery in the stochastic block model is captured by fixed points of the belief propagation iteration, which comes from statistical physics as well. Additionally, none of the estimators appeared in the aforementioned works is one-sided, meaning that one cannot do optimal partial recovery by only estimating a subset of coordinates but making a vanishing fraction of mistakes in total. Compared to previous results, our result appears as the first one that relies on purely combinatorial considerations and provides a one-sided estimator achieving optimal partial recovery (see Section~\ref{sec-positive} for details).

	\subsection{Intuitions and proof overview}
	
	Now we turn to discuss the intuition behind Theorem~\ref{thm-main-ER} as well as our proof strategy. We begin with some notations. For two graphs $G_1,G_2$ on $[n]$ and a permutation $\pi\in\operatorname{S}_n$, we define $\mathcal H_{\pi}=([n],\mathcal E_\pi)$ as the intersection graph of $G_1,G_2$ through $\pi$. That is, for any $(i,j)\in \operatorname{E}$, $(i,j)\in \mathcal E_\pi$ if and only if $(i,j)\in E(G_1)$ and $(\pi(i),\pi(j))\in E(G_2)$. Furthermore, for $U\subset [n]$, we denote $\mathcal H_\pi(U)=(U,\mathcal E_\pi(U))$ as the induced subgraph of $\mathcal H_\pi$ on $U$. For $(\pi^*,G_1,G_2)\sim \mathcal P^*$, we let $\mu=\mu(G_1,G_2)$ be the posterior of $\pi^*$ given $(G_1,G_2)$. A straightforward calculation shows that for any $\pi\in \operatorname{S}_n$, 
	\[
	\mu(\pi)=\frac{P^{|\mathcal E_\pi|}}{\mathcal Z}\,,\text{ where }P=\frac{1-2ps+ps^2}{p(1-s)^2}\,,\text{ and }
	\mathcal Z\equiv \sum_{\pi\in \operatorname{S}_n} P^{|\mathcal E_\pi|}\,.
	\]
	
	We first give a heuristic derivation of Theorem~\ref{thm-main-ER}, explaining why we expect the optimal recovery fraction to approximate $\mu_\lambda\big((\alpha^{-1},\infty)\big)$, particularly when $\alpha^{-1}$ is not an atom of $\mu$. For now, assume that the optimal estimator for $\pi^*$ is a typical sample from the posterior $\mu$. Towards understanding $\operatorname{ov}(\pi,\pi^*)$ for $\pi\sim \mu$, we analyze the primary contribution to $\mathcal Z$. For any subset $U$ of $[n]$, consider the following lower bound on $\mathcal Z$:
	\[
	\mathcal Z \ge \sum_{\pi\in \operatorname{S}_n:\pi\mid_U=\pi^*\mid_U}P^{|\mathcal E_\pi|}=P^{|\mathcal E_{\pi^*}(U)|}\sum_{\pi\in \operatorname{S}_n:\pi\mid_U=\pi^*\mid_U}P^{|\mathcal E_{\pi}\setminus \mathcal E_\pi(U)|}\ge P^{|\mathcal E_{\pi^*}(U)|}\cdot (n-|U|)!\,.
	\]
	Since this holds for any $U\subset [n]$, we conclude (noticing $P=n^{\alpha+o(1)}$)
	\begin{equation}\label{eq-sup-U}
		\log_n \mathcal Z\ge\max_{U\subset [n]}\{\alpha|\mathcal E_{\pi^*}(U)|-|U|\}+n+o(n)\,.
	\end{equation}
	Now, if we assume this approximates an identity, then the primary contribution for $\mathcal Z$ originates from those $\pi\in \operatorname{S}_n$ such that $\pi\mid_{U^*}=\pi^*\mid_{U^*}$, where $U^*$ maximizes the right hand side of \eqref{eq-sup-U}. Therefore, heuristically, we expect the posterior measure $\mu$ to resemble the uniform distribution on the set $\{\pi\in \operatorname{S}_n:\pi\mid_{U^*}=\pi^*\mid_{U^*}\}$. Consequently, we expect $\operatorname{ov}(\pi,\pi^*)\approx |U^*|$ for $\pi\sim \mu$, suggesting that $U^*$ corresponds the optimal recoverable set of vertices. 
	
	We now elaborate the origin of $\mu_\lambda$ in terms of $U^*$. The above discussion highlights the importance of the following subgraph optimization problem:
	\[
	\text{maximizing }f_\alpha(U)\text{ over }U\subset [n],\quad\text{where }f_\alpha(U)\equiv \alpha|\mathcal E_{\pi^*}(U)|-|U|\,.
	\]
	It turns out the solution to this problem is closely related to a concept called \emph{balanced load allocation}. Specifically, there exists a function $\partial\theta:[n]\to [0,\infty)$ dictated by $\mathcal H_{\pi^*}$ called the balanced load, such that $U^*=\{i\in [n]:\partial\theta(i)\ge \alpha^{-1}\}$ (see Section~\ref{subsec-balanced-load} for details). Furthermore, since $\mathcal H_{\pi^*}$ behaves like a sparse \ER graph with average degree $\lambda$, the empirical measure of $\partial\theta$ vaguely converges to a limiting distribution $\mu_\lambda$ in probability \cite{AS16}. As a result, $|U^*|/n$ converges to $\mu\big((\alpha^{-1},\infty)\big)$ in probability as $n\to\infty$, provided that $\alpha^{-1}$ is not an atom of $\mu_\lambda$. This completes our heuristic justification. Additionally, the intuition here hints why the problem becomes subtler when $\alpha^{-1}$ is an atom of $\mu_\lambda$: there are linearly many vertices with balanced loads close to $\alpha^{-1}$, and thus the maximizer $U^*$ of $f_\alpha$ is essentially non-stable. Precisely, the size of sets asymptotically maximizing $f_\alpha(U)$ can fluctuate linearly in $n$, which leads to a more complicated structure of the posterior measure $\mu$.
	
	We next turn to discuss our actual proof strategies. While the derivation above suggests a unified approach to both the positive and negative results in Theorem~\ref{thm-main-ER}, we adopt entirely different arguments to prove the upper and lower bounds separately. We first introduce some notations: fix $\varepsilon>0$, recall $(\pi^*,G_1,G_2)\sim \mathcal P^*$ and $\mathcal H_{\pi^*}$ is the intersection graph of $G_1,G_2$ though $\pi^*$.
	We call a vertex in $[n]$ heavy (resp. light), if it has balanced load in $\mathcal H_{\pi^*}$ no less than $\alpha^{-1}+\varepsilon$ (resp. no more than $\alpha^{-1}-\varepsilon$). The core of the argument establishes that the $\pi^*$ values of heavy vertices are recoverable, whereas light vertices are not. Our arguments are in large part inspired by \cite{DD23b}, but we believe we have also developed methods and arguments that are conceptually new and of independent interest. In what follows, we discuss in more detail about our approaches, highlighting the connections with \cite{DD23b} as well as the new ingredients we add. 
	
    For the positive result, we show one can recover the $\pi^*$ values for almost all heavy vertices with high probability. While writing the paper, we noticed a similarity to \cite{CKMP20}, where the $k$-core is generalized here to an ``$\alpha^{-1}$-core''. However, our arguments are more delicate, drawing largely from a truncated first moment calculation in \cite{DD23b}. Specifically, for $U\in [n]$ and $\pi\in \operatorname{S}_n$, if the induced subgraph $\mathcal H_\pi(U)$ has a large edge-to-ratio density (exceeding $\alpha^{-1}$), we expect that $\pi\mid_U$ should generally align with $\pi^*\mid_U$. However, a direct first moment method fails to prove this heuristic, as the predominant contribution to the first moment comes from the rare event that $\mathcal H_\pi(U)$ has an unusually dense subgraph. \cite{DD23b} overcame this issue by applying a truncated first moment method with the additional constraint that $\mathcal H_\pi(U)$ is \emph{well-balanced}. In this context, $\mathcal H_\pi(U)$ being well-balanced means its maximal subgraph density is close to its overall density. Using informal notations, it can be shown as in \cite{DD23b} that for any $\rho>\alpha^{-1}$ and sufficiently small $\varepsilon>0$ (we use $\pi\mid_U\approx \pi^*\mid_U$ to indicate $\pi\mid_U,\pi^*\mid_U$ agree except on $o(n)$ vertices), 
    \begin{align*}
    &\quad\ \  \mathbb{P}[\exists U,\pi\text{ s.t. } |\mathcal E_\pi(U)|\ge (\rho-\varepsilon)|U|; |\mathcal E_\pi(W)|\le (\rho+\varepsilon)|W|,\forall W\subset U;\pi\mid_U\not\approx\pi^*\mid_U]
    \\\le&\ \sum_U\sum_{\sigma=\pi\mid_U}\mathbf{1}\{\sigma\not\approx\pi^*\mid_U\}\mathbb{P}[|\mathcal E_\sigma(U)|\ge(\rho-\varepsilon)|U|;|\mathcal E_\sigma(W)|\le (\rho+\varepsilon)|W|,\forall W\subset U]=o(1)\,.
    \end{align*}
    In the proof therein, it is necessary to leverage the constraint $|\mathcal E(W)|\le (\rho+\varepsilon)|W|,\forall W\subset U$ to bound the probability terms effectively. 
    We note an equivalent interpretation to this crucial constraint is that all vertices in $\mathcal H_\pi(U)$ have roughly the same balanced loads. Inspired by this, we consider to recover all the heavy vertices iteratively, starting from those with highest loads and progressing to those with loads close to $\alpha^{-1}+\varepsilon$. At each step, we focus on recovering vertices of roughly the same loads, thereby creating an analogue of the crucial constraint. This enables a similar first moment method to work effectively and ensures very few mismatches are made in each step with high probability. By this means, we conclude that the iterative algorithm recovers all but very few heavy vertices with high probability.
	
 Several conceptual challenges remain before we can make the above idea work.  The first issue arises from the iterative nature of our algorithm: each step is affected by previous iterations, requiring the introduction of certain conditioning. This conditioning substantially complicates the distribution and prevents an explicit calculation of the first moment. Fortunately, we can largely bypass this issue using a monotone property of balanced load (Lemma~\ref{lem-load-monotonicity}) and the FKG inequality. With these tools, we can drop the most annoying conditioning terms in strategic places by relaxing towards a favorable direction. Notably, a similar approach was used in \cite{DDG24}.
Another issue is that knowing the algorithm produces few mismatches does not directly imply the algorithm ends at an optimal position; for instance, it might result in only $o(n)$ matched pairs. To address this, we use a variational characterization of the balanced load (Lemma~\ref{lem-load-variational-characterization}) together with an associated stability result (Lemma~\ref{lem-stability}). We will argue that the set of vertices matched by the algorithm is a near maximizer of $f_{\alpha}$ (defined as in Lemma~\ref{lem-load-variational-characterization}), and thus stability guarantees the set to contain most of the heavy vertices. With these two main challenges addressed, we prove the positive result using several additional technical refinements from \cite{DD23b}.  
	
For the negative side, we argue it is almost impossible to recover the $\pi^*$ value of any non-vanishing fraction of light vertices. To this end, we reduce to showing that even knowing the set and $\pi^*$ values of all the non-light vertices, recovering the $\pi^*$ values for a positive fraction of light vertices remains unlikely. This reduction place us in a position similar to the negative result in \cite{DD23b}, allowing us to adopt analogous arguments from that work. \cite{DD23b} proves the impossibility result via reducing it to a property of the posterior distribution of $\pi^*$ given $(G_1,G_2)$. Specifically, since the distribution of $\pi^*$ given $(G_1,G_2)$ is captured by the posterior $\mu_n=\mu_n(G_1,G_2)$, the impossibility of partial recovery follows upon showing that for any $\varepsilon>0$, with high probability over $(G_1,G_2)$ sampled from the correlated model, 
    \begin{equation}\label{eq-pos-1}
    \max_{\widehat{\pi}}\mu_n[\pi:\operatorname{ov}(\widehat{\pi},\pi)\ge \varepsilon n]=o(1)\,.
    \end{equation}
  
    We follow a similar strategy in this paper for our reduced model. However, an important technical modification is required. In \cite{DD23b}, one could shift to show that \eqref{eq-pos-1} happens with high probability over $(G_1,G_2)$ sampled from a simpler correlation-free null model, thanks to an indistinguishable property established in \cite{DD23a}. Here, we lack such a property. Nevertheless, we can prove an $\exp(O(n))$-contiguity between our modified correlated and null models, implying that in order to show a certain event is typical under the correlated model, it suffices to prove it happens with probability $1-\exp(-\omega(n))$ under the null model. Applying this strategy to the above claim, it suffices to prove that an analogue of \eqref{eq-pos-1} happens with $1-\exp(-\omega(n))$ in the null model. At a high level, this seems promising because in the absence of correlation, correctly guessing $\varepsilon n$ coordinates of $\pi^*$ is $\exp(-\Theta(n\log n))$ unlikely. A careful refinement of the arguments in \cite{DD23b} makes this heuristic rigorous, thereby completing the proof of the negative result.

	The paper is organized as follows: Section~\ref{sec-balanced-load} presents preliminaries on key concepts and tools for the proof, including balanced load, the limiting load measure, and the cycle decomposition scheme. Section~\ref{sec-positive} is dedicated to proving Theorem~\ref{thm-main-ER}-(i), while Section~\ref{sec-negative} outlines the conceptual framework for the proof of Theorem~\ref{thm-main-ER}-(ii). Proofs of several technical lemmas and propositions, which are primarily variants or refinements of existing arguments, are deferred to the appendix.
	
	\subsection{Notations}\label{subsec-notation}
	
We conclude this section by introducing some notation conventions that will be used throughout the paper. We use normal font, such as $ U $, to denote a set that may be random or variable, while we use the \texttt{mathtt} font, such as $ \mathtt{U} $, to denote a deterministic, fixed set. Typically, $ \mathtt{U} $ represents a specific realization of the random set $ U $. We use both $ |A| $ and $ \#A $ to represent the cardinality of $A$. 

Below, we provide a list of specific notation that will be referenced frequently in the paper:

	\begin{itemize}
		\item \emph{$\operatorname{E}=\operatorname{E}(n)$ and $\operatorname{E}(U)$ for $U\subset [n]$}. $\operatorname{E}$ denotes for the set of unordered distinct pairs in $[n]$, and for $U\subset [n]$, $\operatorname{E}(U)\equiv\{(i,j)=(j,i)\in \operatorname{E}:i,j\in U\}$. 
            \item \emph{The correlated graphs $G_1,G_2$}. Let $\pi^*$ be a uniform permutation in $\operatorname{S}_n$. Sample independent Bernoulli variables $I_{i,j}\sim \mathbf{B}(1,p), J_{i,j}^1\sim \mathbf B(1,s),J_{i,j}^2\sim \mathbf B(1,s)$ for all $(i,j)\in \operatorname{E}$. We define two graphs $G_1,G_2$ on $[n]$, such that
	\begin{equation*}\label{eq-IJJ}
	(i,j)\in E(G_1)\iff I_{i,j}J_{i,j}^1=1\,,\ \ (i,j)\in E(G_2)\iff I_{(\pi^*)^{-1}(i),(\pi^*)^{-1}(j)}J_{i,j}^2=1\,.
	\end{equation*}
		\item \emph{The measures $\mathcal P^*,\mathcal P$ and $\mathcal P_{\pi^*}$}. We write $\mathcal P^*=\mathcal P^*(n,p,s)$, where $p=p(n),s=s(n)$ satisfies the assumption in Theorem~\ref{thm-main-ER}, and write $\mathcal P$ for its marginal on $(G_1,G_2)$. For $\pi^*\in \operatorname{S}_n$, we abbreviate $\mathcal P_{\pi^*}\equiv \mathcal P^*[\cdot\mid \pi^*]$. We should also think $\mathcal P_{\pi^*}$ as a product measure on the indicators $I_{i,j},J^1_{i,j},J^2_{i,j},(i,j)\in \operatorname{E}$ as defined above.
		\item  \emph{$\pi$-intersection graphs for $\pi\in \operatorname{S}_n$}. For two graphs $G_1,G_2$ on $[n]$ and any matching $\pi\in \operatorname{S}_n$, we define the $\pi$-intersection graph $\mathcal H_\pi=\mathcal H_{\pi}(G_1,G_2)=([n],\mathcal E_\pi)$ as the intersection of $G_1,G_2$ though $\pi$. Precisely, we have for any $(i,j)\in \operatorname{E}$, $(i,j)\in\mathcal E_\pi$ if and only if $(i,j)\in E(G_1)$ and $(\pi^{-1}(i),\pi^{-1}(j))\in E(G_2)$. 
		\item \emph{$\mathcal H_\pi(U)$ as subgraphs of $\mathcal H_\pi$}. For $U\subset [n]$, we write $\mathcal H_\pi(U)$ for the induced subgraph of $\mathcal H_\pi$ on $U$.
		\item \emph{partial matching}. For a subset $V$ of $[n]$, a partial matching $\pi$ on $V$ is an injection $\pi:V\to [n]$. Furthermore, the set of partial matchings on $V$ with image $V'$ (i.e., the set of bijections from $V$ to $V'$) is denoted as $\operatorname{B}(V,V')$. 
		\item \emph{$\pi$-intersection graph for partial matching $\pi$}. For a partial matching $\pi$ on $V$,
		we can also define $\mathcal H_\pi(U)=(U,\mathcal E_\pi(U))$ for any $U\subset V$. To do this, we pick any extension of $\pi$ which is a real matching $\pi'\in \operatorname{S}_n$, and define $\mathcal H_\pi(U)$ as the induced subgraph on $U$ of $\mathcal H_{\pi'}$ (note this is well-defined for $U\subset V$). 
		\item \emph{The balanced load function}. For a graph $G=(V,E)$, we denote $\partial\theta=\partial\theta(G):V\to \mathbb{R}_{\ge 0}$ as the balanced load function of $G$ (defined in Section~\ref{subsec-balanced-load} below). For two graphs $G_1,G_2$, and any matching $\pi\in \operatorname{S}_n$, we denote $\partial\theta_\pi$ in short for the function $\partial\theta(\mathcal H_\pi)$. When $\pi$ is a partial matching on $U$, we also denote $\partial\theta_\pi$ for the function $\partial\theta(\mathcal H_\pi(U))$ on $U$.  
		\item \emph{Weighted edge-vertex function}. For a graph $G=(V,E)$ and any $t\ge 0$, we write $f_t=f_t(G)$ for the function defined on subsets of $V$ given by $f_t(U)\equiv t|E(U)|-|U|$ ($E(U)$ is the edge set of the induced subgraph of $G$ on $U$). For two graphs $G_1,G_2$ on $[n]$ and a matching $\pi\in \operatorname{S}_n$, we denote $f_t^\pi$ for the function $f_t(\mathcal H_\pi)$. When $\pi$ is a partial matching on $U\subset [n]$, we also write $f_t^\pi$ for the corresponding function $f_t(\mathcal H_\pi)$, restricting on subsets of $U$.
	\end{itemize}
	
	\section{Preliminaries}\label{sec-balanced-load}
	
	In this section, we introduce several important concepts and ingredients that will become useful in later proof.
	
	\subsection{Balanced load}\label{subsec-balanced-load}
	
The concept of balanced load was first introduced in \cite{Haj90, Haj96}, emerging naturally from the study of convex optimization problems related to graphs. For a simple graph $ G $, let $ \vec{E}(G) $ denote its directed edge set $\{(x\to y),(y\to x):(x,y)\in E(G)\}$. We use $ x \sim y $ to indicate that vertices $ x $ and $ y $ are adjacent. A load allocation of $ G $ is a function $ \theta : \vec{E}(G) \to [0, 1] $ such that $ \theta(x\to y) + \theta(y\to x) = 1 $ for every directed edge $ (x\to y) \in \vec{E}(G) $. The load of a vertex $ x $ is given by $ \partial \theta(x) \equiv \sum_{y \sim x} \theta(y\to x) $. A load allocation $ \theta $ is called balanced if it satisfies the following condition:
\begin{equation}
\label{eq-balanced-load}
\forall x \sim y, \quad \partial \theta(x) < \partial \theta(y) \implies \theta(x\to y) = 0.
\end{equation}
Such a balanced condition can be viewed as a local optimality constraint: it is impossible to reduce the imbalance by making local adjustments. However, it was shown in \cite{Haj90} that condition \eqref{eq-balanced-load} actually guarantees strong global optimality. Specifically, the following statements are equivalent:
\begin{itemize}
    \item $ \theta $ is a balanced load allocation;
    \item $ \theta $ minimizes $ \sum_{x \in V} f(\partial \theta(x)) $ for some strictly convex function $ f : \mathbb{R}_{\ge 0} \to \mathbb{R} $;
    \item $ \theta $ minimizes $ \sum_{x \in V} f(\partial \theta(x)) $ for every convex function $ f : \mathbb{R}_{\ge 0} \to \mathbb{R} $.
\end{itemize}
As a result, balanced load allocations exist for every graph and induce a unique load function $ \partial \theta $. We refer to this common function as the balanced load function. Since we will only work with balanced load allocations, we denote the balanced load function simply as $ \partial \theta $ (with a slight abuse of notation) and refer to it as the load function.

There is a useful variational characterization of the load function in terms of optimization problems over induced subgraphs. The most well-known case is that the set of vertices with the largest loads solves the densest subgraph problem, where the maximum load corresponds to the maximum subgraph density. More generally, we have the following lemma.
	
	\begin{lemma}\label{lem-load-variational-characterization}
		For a graph $G=(V,E)$ and any $t>0$, the set $U=\{x\in V:\partial\theta(i)\ge t^{-1}\}$ is the maximizer of $f_t(W)\equiv t |E(W)|-|W|$ over $W\subseteq V$, and if there are multiple maximizers, $U$ is the one that further maximizes $|W|$. Furthermore, for any $W\subseteq V$, it holds that $f(U\cup W) \ge f(W)$.
	\end{lemma}
	
	\begin{proof} 
		Let $W$ be a subset of $V$, and let $X=U\setminus W$, $Y=W\setminus U$, and $Z=U\cap W$. Since $X\subseteq U$, we have
		\begin{align*}
			\frac{|X|}{t}
			&\le \sum_{x\in X} \partial\theta(x)
			= \sum_{x\in X}\sum_{y\in V} \theta(y,x)
			\stackrel{\star}{=} 
			\sum_{x\in X}\bigg[\sum_{y\in X} +
			\sum_{y\in Z}\bigg] \theta(y,x)\\
			&= |E(X)|
			+ \sum_{x\in X} \sum_{y\in Z}\theta(y,x)
			\le |E(X)| + |E(X,Z)|\,,
		\end{align*}
		In the above, the equality marked ``$\star$'' uses that if $x\in U$ and $y\notin U$, then $\partial \theta(x)\ge \alpha^{-1} > \partial\theta(y)$, so the balanced condition \eqref{eq-balanced-load} implies $\theta(y,x)=0$. Rearranging the above gives $f_t(U) \ge f_t(Z)$.
		Similarly, since $U\cap Y=\emptyset$, we have
		\begin{align*}
			\frac{|Y|}{t}
			&\ge \sum_{y\in Y}\partial\theta(y)
			= \sum_{y\in Y}\sum_{x\in V}\theta(x,y)
			\ge \sum_{y\in Y}
			\bigg[\sum_{x\in Y}+\sum_{x\in Z}\bigg]
			\theta(x,y) \\
			&=|E(Y)|
			+ \sum_{y\in Y}\sum_{x\in Z} \theta(x,y)
			\stackrel{\star}{=} |E(Y)| + |E(Y,Z)|\,,
		\end{align*}
		where the first inequality is strict whenever $Y\ne\emptyset$. In the above, the equality marked ``$\star$'' again uses the balanced condition which tells us that $\theta(x,y)=1$ for $x\in U$ and $y\notin U$. Rearranging the above gives $f_t(Z)\ge f_t(W)$, with strict inequality unless $Z=W$.
		Lastly, we note that since $U\cup W$ is the disjoint union of $W$ with $X$, we have
		\[f_t(U\cup W)-f_t(W)
		= t\Big( |E(X)|+|E(X,W)|\Big) - |X|
		\ge 0\]
		by the first inequality proved above, since
		$E(X,Z)\subseteq E(X,W)$. This proves the lemma.
	\end{proof}
	
	As a quick application, we show the load function $\partial\theta$ is coordinatewise non-decreasing with respect to the graph. Such a monotone property will allow us to apply the FKG inequality in the analysis of the iterative matching recovery algorithm (see Section~\ref{sec-positive} below). 
	
	\begin{lemma}\label{lem-load-monotonicity}
		For any $G_1=(V,E_1)$ and $G_2=(V,E_2)$ with $E_1\subset E_2$, we have $\partial\Theta_1(x)\le \partial\Theta_2(x)$ for any $x\in V$, where $\partial\Theta_i(x)$ is the load of $x$ in $G_i$.
	\end{lemma}
	\begin{proof}
		It suffices to show that for any $t>0$, $\partial\Theta_1(v)\ge t^{-1}$ implies $\partial\Theta_2(v)\ge t^{-1}$ for any $v\in V$. Assume the contrary, let $V_i$ be the set of vertices with loads no less than $t^{-1}$ in $G_i$, $i=1,2$, then $V_1\cup V_2\neq V_1$. By Lemma~\ref{lem-load-variational-characterization} we see the following two inequalities hold:
		\[
		t|E_1(V_1)|-|V_1|\ge t|E_1(V_1\cap V_2)|-|V_1\cap V_2|\iff t|E_1(V_1\setminus V_1\cap V_2)|-|V_1\setminus V_2|\ge 0\,,
		\]
		\[
		t|E_2(V_1\cup V_2)|-|V_1\cup V_2|<t|E_2(V_2)|-|V_2|\iff t|E_2(V_1\cup V_2\setminus V_2|-|V_1\setminus V_2|<0\,.
		\]
		However, since $E_1\subset E_2$, we have $E_2(V_1\cup V_2\setminus V_2)\supset E_1(V_1\setminus V_1\cap V_2)$. This leads to a contradiction and completes the proof.
	\end{proof}
	
The following lemma can be seen as a stability result for the variational representation given in Lemma~\ref{lem-load-variational-characterization}, which plays a key role in analyzing the algorithm in Section~\ref{sec-positive}. For any $ t > 0 $, it is intuitive to expect that if $ W $ is a near-maximizer of $ f_t $, then $ W $ should be close to $ U_t = \{x \in V : \partial \theta(x) \ge t^{-1}\} $, the true maximizer. However, we need to be more cautious since it is plausible that many vertices have balanced loads close to $ t^{-1} $. The next lemma provides a similar result by relaxing $ U_t $ to $ U_r = \{x \in V : \partial \theta(x) \ge r^{-1}\} $ for some $ r $ slightly smaller than $ t $.

	\begin{lemma}\label{lem-stability}
		Let $G=G(V,E)$ be an arbitrary graph on $n$ vertices. For any $0<r<t$, let $U_t=\{x\in V:\partial\theta(x)\ge t^{-1}\}$ and $U_{r}=\{x\in V:\partial\theta(x)\ge r^{-1}\}$. Then for any $\varepsilon>0$, there exists $\delta=\delta(t,r,\varepsilon)>0$ that does not depend on $G$, such that the following holds for any $W\subset V$:
		\begin{equation}\label{eq-stability}
			f_t(W\cup U_t)\le f_t(W)+\delta n\quad\Rightarrow\quad |U_{r}\setminus W|\le \varepsilon n\,.
		\end{equation}
	\end{lemma}
	\begin{proof}
		We will prove the contrapositive direction. Assuming that $W$ satisfies $|U_{r}\setminus W|>\varepsilon n$, we show that there exists $\delta$ depending only on $r,t$ and $\varepsilon$ such that $f_t(W\cup U_t)>f_t(W)+\delta n$.
		Applying Lemma~\ref{lem-load-variational-characterization} with $t$ replaced by $r$, we have
		\[
		f_{r}(W\cup U_r)\ge f_{r}(W)\,,
		\]
		which implies that
		\[
		|E(W\cup U_r)|-|E(W)|\ge r^{-1}(|W\cup U_r|-|W|)\ge r^{-1}|U_{r}\setminus W|\ge r^{-1}\varepsilon n\,.
		\]
		Since $r<t$, we have
		\[
		f_t(W\cup U_r)-f_t(W)=(f_{r}(W\cup U_r)-f_{r}(W))+(t-r)(|E(W\cup U_r)|-|W|)> \delta n
		\]
		for $\delta=\frac12(t-r)r^{-1}\varepsilon>0$. Given this, we have 
		\[
		f_t(W\cup U_t)-f_t(W)=f_t(W\cup U)-f_t(W\cup U_r)+f_t(W\cup U_r)-f_t(W)>0+\delta n=\delta n\,,
		\]
		where the $0$ term comes from the fact that (since $U_r\subset U$)
		\[
		f_t(W\cup U_t)-f_t(W\cup U_s)=f_t(W\cup U_s\cup U_t)-f_t(W\cup U_s)\ge 0\,,
		\]
		using Lemma~\ref{lem-load-variational-characterization} again. This completes the proof.
	\end{proof}

	We end our discussion of the load function for general graphs with a graph-theoretic description of vertices that have load less than $ 1 $ or greater than $ 1 $. This characterization will be useful in analyzing the case when $ \alpha = 1 $. We say a connected graph is \emph{non-simple} if its number of edges is strictly greater than the number of vertices.

	\begin{lemma}\label{lem-load-1-and->1}
		For any graph $G$, the set of vertices with loads less than $1$ is the union of the tree components of $G$. In addition, the set of vertices with loads greater than $1$ is the union of maximal $2$-cores inside non-simple components of $G$.  
	\end{lemma}
	
	\begin{proof}
		For the first statement, we take $t=1$ in Lemma~\ref{lem-load-variational-characterization} and note that the maximizer of $f_1(\cdot)$ (the one that also maximizes the set size) is the union of all non-tree components. 
		
		For the second claim, without loss of generality, we may assume $G$ is connected. We also assume that $G$ is non-simple, since otherwise there is no vertex with balanced load larger than $1$. Let $U=\{x\in V:\partial\theta(x)>1\}$. Since $V$ is finite, we have for any small enough $\eta>0$, it holds that $U=\{x\in V:\partial\theta(x)\ge 1+\eta\}$ and thus $U$ is the maximizer of $f_{1+\eta}$. This implies that $U$ is a $2$-core, because removing any vertex in $U$ will not increase $f_{1+\eta}(\cdot)$. 
		
		We further claim that $U$ is the maximal $2$-core in $G$. Assuming the contrary, let $W$ be another $2$-core that does not contained in $U$. Choose $x\in W\setminus U$, and let $(x,u),(x,v)$ be two edges within $W$. Since $G$ is connected and $W$ is a $2$-core, both $G-\{(x,u)\}$ and $G-\{(x,v)\}$ are connected. Therefore, there are two distinct paths that connecting $x$ to $U$ (though they might share common edges). Then it is straightforward to check that adding all the vertices on these two paths (including $x$) to $U$ will increase the function $f_{1+\eta}$ provided that $\eta$ is small enough (say $\eta<|V|^{-1}$). This leads to a contradiction and thus proves the second statement. 
	\end{proof}
	
	\subsection{The limiting load distribution}\label{subsec-limiting-load-measure}
Motivated by the connection between the balanced load function and the densest subgraph problem, \cite{Haj96} studied the empirical distribution of the balanced load function in sparse \ER graphs and infinite trees. The author conjectured that as the size of the \ER graph tends to infinity, the empirical distribution converges in a suitable sense to a deterministic measure dictated by random trees. Additionally, he proposed a variational characterization of the limiting measure based on non-rigorous arguments.

The conjecture was resolved in \cite{AS16} for a broader class of sparse random graphs, using the unified framework of local weak convergence. The authors showed that if a sequence of graphs $ \{G_n\} $ converges to a probability measure $ \mathscr{P} $ on rooted graphs in the local weak convergence sense, then the empirical balanced load distribution of $ \{G_n\} $ weakly converges in probability to a measure $ \mu $ on $ \mathbb{R}_{\ge 0} $, determined by $ \mathscr{P} $. Since sparse \ER graphs locally weakly converge to Poisson Galton-Watson trees, the results in \cite{AS16} apply, yielding a family of limiting distributions $ \{\mu_\lambda\}_{\lambda > 0} $ arising from Poisson Galton-Watson trees $ \{\operatorname{PGWT}(\lambda)\}_{\lambda > 0} $. The authors also established the conjectured variational characterization of $ \mu_\lambda $. Their proof relies heavily on the objective method; while we do not delve into the details here, we refer interested readers to \cite{AS04} for an excellent survey on this approach. We will state only the corollaries of results in \cite{AS16} that are particularly relevant for our analysis.

	\begin{proposition}\label{prop-weak-convergence-of-empirical-measure}\cite[Theorem 1]{AS16}
		There exists a family of deterministic measure $\{\mu_\lambda\}_{\lambda>0}$ with the following property: For $G_n\sim \mathbf G(n,\frac\lambda n)$ and any $a$ not an atom of $\mu_\lambda$, it holds that
		\begin{equation}\label{eq-ewak-convergence}
			\frac{\#\{x\in V(G_n):\partial\theta(x)>a\}}{n}\stackrel{\text{in probability}}{\longrightarrow} \mu_\lambda\big((a,\infty)\big)\text{ as }n\to\infty\,.
		\end{equation}
	\end{proposition}
	
Beyond the ``bulk measure" convergence result, \cite{AS16} also proved the convergence of the maximum balanced load, assuming the degrees of $ \{G_n\} $ have exponential tails. In particular, this result applies to sparse \ER graphs, leading to the following proposition.

	\begin{proposition}\label{prop-maximal-density}\cite[Theorem 3]{AS16}
		Let $\rho(\lambda)=\sup\{t\in \mathbb{R}:\mu_\lambda((t,\infty))>0\}$ be the rightmost point of $\operatorname{supp}(\mu_\lambda)$. Then for $G_n\sim \mathbf{G}(n,\frac\lambda n)$, it holds that the maximal load in $G_n$ (i.e., the maximal subgraph density of $G_n$) converges in probability to $\rho(\lambda)$ as $n\to\infty$.
	\end{proposition}

	\subsection{The orbit decomposition}\label{subsec-orbit-decomposition}
	
We now return to the correlated \ER graph model. In Section~\ref{sec-positive}, we will need to handle certain large deviation events concerning the number of edges in (subgraphs of) the $ \pi $-intersection graph $ \mathcal{H}_\pi $. The first step in this analysis involves an orbit decomposition of $ \sigma = \pi^* \circ \pi^{-1} $, where $ \pi^* $ is the true matching. In this subsection, we present the simplest orbit decomposition scheme and give its first application. We will revisit this topic in Section~\ref{sec-positive} under a more complex setting, employing a more refined analysis. Thus, we intend for this discussion to serve as a warm-up for the more intricate arguments to follow.

	Recall that $\mathcal E_\pi$ denotes for the set of edges in the $\pi$-intersection graph $\mathcal H_\pi$. Our main goal is to prove the following lemma. The last statement therein will serve as a technical input in both the design and the analysis of the iterative matching algorithm in Section~\ref{sec-positive}. 
	\begin{lemma}\label{lem-max-E-pi}
		There exists $A=A(\alpha,\lambda)>0$, such that with high probability over $(\pi^*,G_1,G_2)\sim \mathcal P^*$, we have $|\mathcal E_\pi|\le An$ for any $\pi\in \operatorname{S}_n$. In particular, with high probability for any $\eta>0$ and $\pi\in \operatorname{S}_n$, the number or vertices with load at least $A/\eta$ in $\mathcal H_\pi$ is at most $\eta n$. 
	\end{lemma}
	
	
	We will prove this lemma by a (truncated) union bound. To this end, we need to understand the probability that $|\mathcal E_\pi|$ gets large, which motivates the orbit decomposition below. In what follows, we conditioned on the true matching $\pi^*\in \operatorname{S}_n$ and so $(G_1,G_2)\sim \mathcal P_{\pi^*}$. 
	We also fix $\pi\in \operatorname{S}_n$. Let $\sigma=\pi^*\circ \pi^{-1}\in \operatorname{S}_n$. Recalling that $\operatorname{E}$ is the set of unordered pairs $(i,j)$ with $1\le i\neq j\le n$, it is clear $\sigma$ induces a permutation $\Sigma$ on $\operatorname{E}$ give by $\Sigma(i,j)=(\sigma(i),\sigma(j))$, and $\operatorname{E}$ decomposes into cycles of the form $(e_1,\dots,e_k)$ such that $\Sigma(e_i)=e_{i+1},1\le i\le k-1$ and $\Sigma(e_k)=e_1$. We shall refer to these cycles as the orbits induced by $\pi$, and denote the set of cycles by $\mathcal O_\pi$ (be aware that this implicitly depends on $\pi^*$). 
	
	For any $x\in [n]$, let $C_\sigma(x)$ be the cycle in $\sigma$ containing $x$, and let $n_\sigma(x)$ be its length. Similar notations $C_\Sigma(e),n_\Sigma(e)$ apply for $e\in \operatorname{E}$. Some cycles in $\Sigma$ are of a special type: for $x\in [n]$ with $n_\sigma(x)$ being even, let $y=\sigma^{(n_\sigma(x)/2)}$, then $C_\Sigma(x,y)$ is a cycle of length $n_\sigma(x)/2$. Besides the special cycles, there is a simple relation between $n_\Sigma(x,y)$ and $n_\sigma(x),n_\sigma(y)$ given by $n_\Sigma(x,y)=\operatorname{LCM}(n_\sigma(x),n_\sigma(y))$. In particular, we have $n_\Sigma(x,y)\ge n_\sigma(x)\vee n_\sigma(y)$ for any $(x,y)$ that does not belong to a special cycle (one should think that special cycles are rare).

	
	The point of defining the orbit set $\mathcal O_\pi$ is regarding the number of edges in $\mathcal E_\pi$, the contributions from different orbits in $\mathcal O_\pi$ are independent of each other under $\mathcal P_{\pi^*}$. Moreover, it turns out that edges from different types of orbits have different large deviation rates. Inspired by this, we group edges in $\mathcal E_\pi$ according to the type of orbits they come from, and define respectively $E_s$ and $E_k,k=1,2,\dots,n^2$ to be the number of edges in $\mathcal E_\pi$ that belongs to special cycles, and non-special cycles with length $k$, $k=1,2,\dots,n^2$ in $\mathcal O_\pi$. It is clear that 
	\[
	|\mathcal E_\pi|=E_s+E_1+E_2+\cdots+E_{n^2}\,,
	\]
	and the random variables $E_s,E_k,k=1,2,\dots,n^2$ are independent (under $\mathcal P_{\pi^*}$). The key fact we need is that for any $k\in \{\operatorname{s},2,3,\dots,n^2\}$, the tail probability $\mathcal P_{\pi^*}[E_k\ge x]$ decays like $\exp(-\Theta(x\log n))$, as shown in the next lemma. 	
	
	\begin{lemma}\label{lem-tail-crude}
		There exists $M=o(n\log n)$ such that the following holds for any $x\ge 0$:
		\begin{align}\label{eq-tail-special-cycle}
			&\mathcal P_{\pi^*}[E_s\ge x]\le \exp(M-x\log n)\,,
		\end{align}
		and for $E_{>1}=E_2+\cdots+E_{n^2}$, $\gamma=\alpha\wedge 1/2$,  
		\begin{equation}\label{eq-tail-k}
			\mathcal P_{\pi^*}[E_{>1}\ge x]\le  \exp(M-\gamma  x\log n)\,.
		\end{equation}
	\end{lemma}
	
	Lemma~\ref{lem-tail-crude} is just a special case of \cite[Lemma 2.4]{DD23b}, and we omit the proof here. In the next section, we will give precise tail estimates for each of the variables $E_k$ (under an appropriate conditioning therein) in order for a much more delicate estimation.

	\begin{proof}[Proof of Lemma~\ref{lem-max-E-pi}]
		We focus on the first statement as the second one follows readily. 
		We still work under the conditioning of $\pi^*$. Recalling $\gamma=\alpha\wedge 1/2>0$, we pick $C$ such that
		\begin{equation}
			\label{eq-choice-C}
			1-\gamma (A-\varrho(\lambda)-1)<0\,.
		\end{equation}
	We define $\mathcal U$ as the event that $|\mathcal E_\pi|\ge An$ for some $\pi\in \operatorname{S}_n$, and $\mathcal V$ as the event that the maximal subgraph density of $\mathcal H_{\pi^*}$ is at most $\varrho(\lambda)+1$. Since $\mathcal H_{\pi^*}\sim \mathbf G(n,\frac\lambda n)$, from Proposition~\ref{prop-maximal-density} we conclude that $\mathcal P_{\pi^*}[\mathcal V]=1-o(1)$. Therefore, it suffices to show that $\mathcal P_{\pi^*}[\mathcal U\cap \mathcal V]=o(1)$. 
		
		For each fixed $\pi$, recall the decomposition 
		\[
		|\mathcal E_\pi|=E_s+E_1+E_2+\cdots+E_n=E_s+E_1+E_{>1}\,.
		\]
		Since $E_1$ is the number of edges in $\mathcal H_{\pi}$ within fixed points of $\sigma=\pi^*\circ \pi^{-1}$, and all these edges also lie in $\mathcal H_{\pi^*}$, we conclude that $E_1\le |\mathcal E_{\pi^*}([n])|$. This is further bounded by $(\rho(\lambda)+1)n$ under $\mathcal V$, and thus $\mathcal U\cap \mathcal V$ implies $E_s+E_{>1}\ge (A-\varrho(\lambda)-1)n]\equiv \xi n$ for some $\pi$. Therefore, by taking union bounds twice we have
		\begin{align*}
			\mathcal P_{\pi^*}[\mathcal U\cap \mathcal V]
			\le &\ n!\cdot \mathcal P_{\pi^*}[E_s+E_{>1}\ge \xi n]\\
			\le&\ \exp (n\log n)\cdot \sum_{\substack{x_1,x_2\in \mathbb{N}\\x_1+x_2\ge \xi n}}\mathcal P_{\pi^*}[E_s\ge x_1, E_{>1}\ge x_2]\\
			\le&\ \exp((1+o(1))n\log n)\sup_{x_1+x_2\ge \xi n}\mathcal P_{\pi^*}[E_s\ge x_1]\mathcal P_{\pi^*}[E_{>1}\ge x_2]\,,
		\end{align*}
		where the last inequality follows from independence. Using Lemma~\ref{lem-tail-crude}, we conclude the above expression is bounded by
		\begin{align*}
			\exp((1+o(1))&n\log n) \cdot\sup_{x_1+x_2 \ge \xi n}\exp(2M-(x_1+\gamma x_2)\log n)\\
			=&\ \exp((1-\gamma \xi+o(1))n\log n)\,,
		\end{align*}
		which is $o(1)$ by our choice of $A$ in \eqref{eq-choice-C},  as desired.  
	\end{proof}

	\section{Proof of the positive result}\label{sec-positive}
This section is devoted to proving Theorem~\ref{thm-main-ER}-(i). Fix $\varepsilon>0$ and assume that $ (\pi^*, G_1, G_2) \sim \mathcal{P}^* $. Recall that for any $ \pi \in \operatorname{S}_n $, we denote by $ \partial \theta_\pi $ the load function of $ \mathcal{H}_\pi $. Our goal is to find an estimator $ \widetilde{\pi} = \widetilde{\pi}(G_1, G_2) $ such that $ \widetilde{\pi} $ agrees with $ \pi^* $ on most of the heavy vertices, i.e., those vertices $ i \in [n] $ for which $ \partial \theta_{\pi^*}(i) \ge \alpha^{-1} + \varepsilon $. As mentioned earlier, we will construct an iterative algorithm that recovers almost all of the heavy vertices in batches, starting from the heaviest and progressing to the lightest.

We now present our algorithm. It takes two graphs $ G_1 $ and $ G_2 $ along with two parameters $ \varepsilon, \eta > 0 $ as input and outputs a partial matching $ \widetilde{\pi} $. The algorithm proceeds greedily, aiming to maximize the size of a specific set $ U_t $ at each step. To handle ties, we fix an arbitrary total ordering $ \prec $ on $ 2^{[n]} $.

	\begin{algorithm}
		\caption{Iterative Matching Algorithm}
		\begin{algorithmic}
			\State \textbf{Input:} $G_1$, $G_2$, $\varepsilon,\eta>0$. 
			\State Let $N=\lceil(A/\eta-\alpha^{-1}-\varepsilon)/\eta\rceil$, where $A=A(\alpha,\lambda)$ is defined as in Lemma~\ref{lem-max-E-pi}. 
			\State Define $I_0=[A/\eta,\infty)$, $I_k= [A/\eta-k\eta,A/\eta-(k-1)\eta)$, $k=1,\dots,N-1$, and $I_N=[\alpha^{-1}+\varepsilon,A/\eta-(N-1)\eta)$.
			\State Set $U_{-1}=\emptyset$, $\widetilde\pi:U_{-1}\to [n]$. 
			\For{$ 0\le k\le N$}
			\State Enumerate over all intersection graphs $\mathcal H_\pi$ where $\pi\mid_{U_{k-1}}=\widetilde{\pi}$.
			\State Select $\pi_k$ that maximizes the size of $V_k^{\pi_k}=\{i\in [n]:\partial\theta_\pi(i)\in I_k\}$. If there is a tie, we choose $\pi_k$ minimizing $V_k^{\pi_k}$ under $\prec$.
			\State Set $U_{k}=U_{k-1}\cup V_k^{\pi_k}$, and
			update $\widetilde\pi$ to $\widetilde{\pi}:U_{k}\to [n]$ defined by $\widetilde{\pi}=\pi_k\mid_{U_{k}}$. 
			\EndFor\\
			\textbf{Output:} $U_N,\widetilde\pi:U_N\to[n]$. 
		\end{algorithmic}
	\end{algorithm}
	
Given the output $ (U_N, \widetilde{\pi}) $ of the algorithm, we denote by $ U^{\operatorname{cor}} $ the set of vertices $ i \in U_N $ such that $ \widetilde{\pi}(i) = \pi^*(i) $. We also define $ U^{\operatorname{heavy}} $ as the set of heavy vertices. Our goal is to show that, provided $ \eta $ is sufficiently small, with high probability $ U^{\operatorname{cor}} $ contains almost the entirety of $ U^{\operatorname{heavy}} $. To establish this, we will use the stability result for the variational characterization of balanced loads (Lemma~\ref{lem-stability}). With this tool in hand, the following proposition ensures the performance of our algorithm. We write $ \alpha_\varepsilon = (\alpha^{-1} + \varepsilon)^{-1} $ and recall that $ f_{\alpha_\varepsilon}^{\pi^*}(U) \equiv \alpha_\varepsilon |\mathcal{E}_{\pi^*}(U)| - |U| $.

	\begin{proposition}\label{prop-algo-performance}
		For $(\pi^*,G_1,G_2)\sim \mathcal P^*$ and any $\varepsilon,\delta>0$, there exists $\eta_0=\eta_0(\alpha,\lambda,\varepsilon,\delta)$, such that whenever $\eta<\eta_0$, with high probability the set $U^{\operatorname{cor}}$ satisfies
		\begin{equation}\label{eq-near-maximizer}
			f^{\pi^*}_{\alpha_\varepsilon}(U^{\operatorname{cor}})\ge f^{\pi^*}_{\alpha_\varepsilon}(U^{\operatorname{cor}}\cup U^{\operatorname{heavy}})-\delta n\,,
		\end{equation} 
	\end{proposition}
	
	
	\begin{proof}[Proof of Theorem~\ref{thm-main-ER}-(i) assuming Proposition~\ref{prop-algo-performance}]
		Fix an arbitrary $\varepsilon>0$. Let $\alpha_\varepsilon=(\alpha^{-1}+\varepsilon)^{-1}$ as above, and $\alpha_{2\varepsilon}=(\alpha^{-1}+2\varepsilon)^{-1}$. Pick $\delta=\delta(\alpha_\varepsilon,\alpha_{2\varepsilon})$ as in Lemma~\ref{lem-stability}, and pick $\eta_0=\eta_0(\alpha,\lambda,\varepsilon,\delta)$ as in Proposition~\ref{prop-algo-performance}. For any $\eta<\eta_0$, we have from Proposition~\ref{prop-algo-performance} that with high probability, $U^{\operatorname{cor}}$ satisfies \eqref{eq-near-maximizer}. From Lemma~\ref{lem-stability}, \eqref{eq-near-maximizer} indicates that $|U_{\alpha_{2\varepsilon}}\setminus U^{\operatorname{cor}}|\le \varepsilon n$, where $U_{\alpha_{2\varepsilon}}$ is the set of vertices with balanced load no less than $\alpha_{2\varepsilon}^{-1}=\alpha^{-1}+2\varepsilon$ in $\mathcal H_{\pi^*}$. consequently, $\widetilde{\pi}(i)=\pi^*(i)$ holds for all but at most $\varepsilon n$ many vertices in $U_{\alpha_{2\varepsilon}}$. 
		
		On the other hand, when $\alpha^{-1}+2\varepsilon$ is not an atom of $\mu_\lambda$, Proposition~\ref{prop-weak-convergence-of-empirical-measure} implies that with high probability, $|U_{\alpha_{2\varepsilon}}|\ge (\mu_\lambda\big((\alpha^{-1}+2\varepsilon,\infty)\big)-\varepsilon)n$. Therefore, under the intersection of all these typical events, we have $\widetilde{\pi}(i)=\pi^*(i)$ holds for at least $(\mu_\lambda\big((\alpha^{-1}+2\varepsilon,\infty)\big)-2\varepsilon)n$
		many indices $i\in [n]$. Since $\mu_\lambda$ has at most countably many atomic points, the above holds for all but countably many $\varepsilon>0$ and $\eta$ small enough in terms of $\varepsilon$. Thus, Theorem~\ref{thm-main-ER} -(i) follows by taking appropriate $\varepsilon,\eta>0$ and arbitrarily extending the output $\widetilde{\pi}$ of the iterative matching algorithm to a real matching in $\operatorname{S}_n$. 
	\end{proof}
	
	The remaining of this section serves for the proof of Proposition~\ref{prop-algo-performance}. Towards this end, we will perform several layers of reduction, as divided into subsections below. Section~\ref{subsec-reduce-to-show-wrong-pairs-are-rare} provides the most conceptual part of our reduction: in order to show \eqref{eq-near-maximizer} happens with high probability, it suffices to prove the partial matching $\widetilde\pi$ produces very few mismatches. This might seem surprising at the first glance, as the condition $|U^{\operatorname{cor}}|\approx |U_N|$ does not provide an a-priori nontrivial lower bound on $|U^{\operatorname{cor}}|$. The key here lies in the greedy nature of our algorithm. For example, in the final iteration, the algorithm aims to maximize the number of vertices with load at least $\alpha^{-1}+\varepsilon$ in $\mathcal H_{\pi_N}$. Thus, if $\widetilde\pi$ closely algins with $\pi^*$ on $U_{N-1}$, then $|U_N|$ should be roughly lower-bounded by $|U^{\operatorname{heavy}}|$, as it is possible to take $\pi_N\approx \pi^*$ with $\pi_N=\widetilde\pi$ restricted to $U_{N-1}$. 
	
	In light of the above reduction, it suffices to show that the algorithm produces very few spurious matchings in each step. Sections~\ref{subsec-filtration}, \ref{subsec-orbit-decomposition-conditioning} and \ref{subsec-proof-or-reduction} are dedicating to proof this claim. Due to the iterative nature of our algorithm, it is necessary to handle conditioning throughout the analysis. In Section~\ref{subsec-filtration}, we define a suitable filtration and further reduce the claim to showing that a certain event $\mathcal U$ happens with vanishing conditional probability (see Proposition~\ref{prop-reduced-conditional-probability}). In Section~\ref{subsec-orbit-decomposition-conditioning} and Section~\ref{subsec-proof-or-reduction}, we extend arguments in \cite[Section 2]{DD23b} to a conditional setting to prove the reduced proposition, thereby concluding the proof of positive result.
	
	\subsection{Reduction to few mismatches}\label{subsec-reduce-to-show-wrong-pairs-are-rare}
	
	Our first step is to argue that Proposition~\ref{prop-algo-performance} follows if we can show $\widetilde\pi(i)=\pi^*(i)$ for all but very few indices $i\in U_N$ (see Proposition~\ref{prop-wrong-pairs-are-rare} below for a precise statement). 
	We start with a technical lemma, whose proof we leave into the appendix.
	\begin{lemma}\label{lem-D}
		For any constant $\delta>0$, there exists a constant $D=D(\lambda,\delta)>0$ such the following event $\mathcal D$ happens with high probability:
		\begin{equation}\label{eq-event-D}
			\mathcal D=\{\forall U\subset [n],\#\{\text{edges in }\mathcal H_{\pi^*}\text{ with at least one end point in }U\}\}\le D|U|+\frac{\delta n}{2}\,.
		\end{equation}
	\end{lemma}
	
	We claim it suffices to prove the following proposition.
	
	\begin{proposition}\label{prop-wrong-pairs-are-rare}
		With the notations in Proposition~\ref{prop-algo-performance}, let $D=D(\lambda,\delta)$ be defined as in Lemma~\ref{lem-D}, then there exists $\eta_0=\eta_0(\alpha,\lambda,\varepsilon,\delta)>0$ such that for any $\eta<\eta_0$, it holds with high probability that
		\begin{equation}\label{eq-wrong-pairs-are-rare}
			|U_N\setminus U^{\operatorname{cor}}|\le \frac{\delta n}{4D}\,.
		\end{equation}
	\end{proposition}
	\begin{proof}[Proof of Proposition~\ref{prop-algo-performance} assuming Proposition~\ref{prop-wrong-pairs-are-rare}]
		We claim that \eqref{eq-near-maximizer} holds under $\mathcal D\cap \{\eqref{eq-wrong-pairs-are-rare}\text{ holds}\}$. From Lemma~\ref{lem-D} and Proposition~\ref{prop-wrong-pairs-are-rare}, clearly Proposition~\ref{prop-algo-performance} follows from this claim.
		
		To verify the claim, we argue by contradiction. Assuming that both $\mathcal D$ and \eqref{eq-wrong-pairs-are-rare} hold, while \eqref{eq-near-maximizer} fails. Recalling  $\alpha_\varepsilon=(\alpha^{-1}+\varepsilon)^{-1}$, by our choice of $U_N$ and $\widetilde\pi:U_N\to [n]$, $U_N$ is the maximizer of $f_{\alpha_\varepsilon}^\pi(U)=\alpha_\varepsilon|U|-|\mathcal E_\pi(U)|$ for any matching $\pi$ such that $\pi\mid_{U_N}=\widetilde{\pi}$. Define
		\[
		U^{\operatorname{miss}}=\{v\in [n]: v\in U^{\operatorname{heavy}},v\notin U_N,\pi^*(v)\notin \widetilde\pi(U_N)\}\,.
		\]
		We will argue that for any matching $\pi\in \operatorname{S}_n$ with $\pi\mid_{U_N}=\widetilde\pi$ and $\pi\mid_{U^{\operatorname{miss}}}=\pi^*$, 
		\[
		f_{\alpha_\varepsilon}^\pi(U^{\operatorname{miss}}\cup U_N)>f_{\alpha_\varepsilon}^{\pi}(U_N)\,,
		\]
		leading to a contradiction.

		Since \eqref{eq-near-maximizer} is not true, we have for $U^{\operatorname{cor}}=\{i\in U_N:\widetilde\pi(i)=\pi^*(i)\}$ and $U^{\operatorname{heavy}}=\{i\in [n]:\partial\theta_{\pi^*}(i)\ge \alpha^{-1}+\varepsilon\}$, it holds
		\begin{align*}
			f_{\alpha_\varepsilon}^{\pi^*}(U^{\operatorname{cor}}\cup U^{\operatorname{heavy}})&=\alpha_\varepsilon|\mathcal E_{\pi^*}(U^{\operatorname{cor}}\cup U^{\operatorname{heavy}})|-|U^{\operatorname{cor}}\cup U^{\operatorname{heavy}}|\\\ge f_{\alpha_\varepsilon}^{\pi^*}(U^{\operatorname{cor}})+\delta n&=\alpha_\varepsilon|\mathcal E_\pi(U^{\operatorname{cor}})|-|U^{\operatorname{cor}}|+\delta n\,,
		\end{align*}
		or equivalently,
		\[
		\alpha_\varepsilon(|\mathcal E_{\pi^*}(U^{\operatorname{cor}}\cup U^{\operatorname{heavy}})|-|\mathcal E_{\pi^*}( U^{\operatorname{cor}})|)\ge |U^{\operatorname{heavy}}\setminus U^{\operatorname{cor}}|+\delta n\,.
		\]
		We note $U^{\operatorname{heavy}}\setminus U^{\operatorname{cor}}$ is the set of vertices $i\in U^{\operatorname{heavy}}$ that are not correctly matched by $\widetilde\pi$, and there are three types of these vertices: (i) neither $i$ is in $U_N$ nor $\pi^*(i)$ is matched by some vertex in $U_N$ by $\widetilde\pi$; (ii) $i\in U_N$ but $\widetilde\pi(i)\neq \pi^*(i)$; (iii) $\pi^*(i)=\widetilde\pi(j)$ for some $i\neq j\in U_N$. This yields the relation
		\[
		U^{\operatorname{miss}}\subset U^{\operatorname{heavy}}\setminus U^{\operatorname{cor}}\subset U^{\operatorname{miss}}\cup (U_N\setminus U^{\operatorname{cor}})\cup ((\pi^*)^{-1}\circ\widetilde\pi(U_N)\setminus U^{\operatorname{cor}})\,.
		\]
		Thus, we have $|U^{\operatorname{heavy}}\setminus U^{\operatorname{cor}}|\ge |U^{\operatorname{miss}}|$ and by \eqref{eq-wrong-pairs-are-rare},
		\[
		|(U^{\operatorname{heavy}}\setminus U^{\operatorname{cor}})\setminus U^{\operatorname{miss}}|=|U^{\operatorname{heavy}}\setminus(U^{\operatorname{miss}}\cup U^{\operatorname{cor}})|\le \frac{\delta n}{2D}\,.
		\]
		Meanwhile, it is clear that any edge in $\mathcal E_{\pi^*}(U^{\operatorname{heavy}} \cup U^{\operatorname{cor}})\setminus \mathcal E_{\pi^*}(U^{\operatorname{miss}}\cup U^{\operatorname{cor}})$ has at least one endpoint in $U^{\operatorname{heavy}}\setminus (U^{\operatorname{miss}}\cup U^{\operatorname{cor}})$, so under the event $\mathcal D$ we have
		\begin{align*}
			&\ |\mathcal E_{\pi^*}(U^{\operatorname{heavy}} \cup U^{\operatorname{cor}})|- |\mathcal E_{\pi^*}(U^{\operatorname{miss}}\cup U^{\operatorname{cor}})|\\
			\le&\  |\mathcal E_{\pi^*}(U^{\operatorname{heavy}} \cup U^{\operatorname{cor}})\setminus \mathcal E_{\pi^*}(U^{\operatorname{miss}}\cup U^{\operatorname{cor}})|\\
			\le &\ D|U^{\operatorname{heavy}}\setminus (U^{\operatorname{miss}}\cup U^{\operatorname{cor}})|+\frac{\delta n}{2}
			\le \delta n\,.
		\end{align*}
		Altogether we conclude
		\begin{align*}
			&\ \alpha_\varepsilon(|\mathcal E_{\pi^*}(U^{\operatorname{miss}}\cup U^{\operatorname{cor}})|-|\mathcal E_{\pi^*}(U^{\operatorname{cor}})|)\\>&\  \alpha_\varepsilon(|\mathcal E_{\pi^*}(U^{\operatorname{heavy}}\cup U^{\operatorname{cor}})|-|\mathcal E_{\pi^*}(U^{\operatorname{cor}})|)-\alpha_\varepsilon\cdot \delta n\\
			\ge&\ |U^{\operatorname{heavy}}\setminus U^{\operatorname{cor}}|+\delta n- \delta n
			\ge  |U^{\operatorname{miss}}|\,.
		\end{align*}
		This can be rewritten as (noticing that $U^{\operatorname{miss}}\cap U^{\operatorname{cor}}=\emptyset$ and thus $|U^{\operatorname{miss}}|=|U^{\operatorname{miss}}\cup U^{\operatorname{cor}}|-|U^{\operatorname{cor}}|$)
		\[
		f_{\alpha_\varepsilon}^{\pi^*}(U^{\operatorname{miss}}\cup U^{\operatorname{cor}})>f_{\alpha_\varepsilon}^{\pi^*}(U^{\operatorname{cor}})\,.
		\]
		Now for any matching $\pi\in \operatorname{S}_n$ with $\pi\mid_{U_N}=\widetilde\pi$ and $\pi\mid_{U^{\operatorname{miss}}}=\pi^*$, because $\pi=\pi^*$ on the set $U^{\operatorname{miss}}\cup U^{\operatorname{cor}}$, we also have
		\[
		f_{\alpha_\varepsilon}^{\pi}(U^{\operatorname{miss}}\cup U^{\operatorname{cor}})>f_{\alpha_\varepsilon}^{\pi}(U^{\operatorname{cor}})\,.
		\]
		But $U^{\operatorname{cor}}\subset U_N$, an easy algebraic manipulation yields the desired contradiction
		\[
		f_{\alpha_\varepsilon}^\pi(U^{\operatorname{miss}}\cup U_N)>f_{\alpha_\varepsilon}^{\pi}(U_N)\,.
		\]
		This verifies the claim and thus proves Proposition~\ref{prop-algo-performance}.
	\end{proof}

	\subsection{The conditioning scheme}\label{subsec-filtration}
	
	We now proceed to prove Proposition~\ref{prop-wrong-pairs-are-rare}. For simplicity, we write $\zeta=\delta/(4D)$. Fix a small $\eta>0$. Recall that $N=\lceil( C/\eta-\alpha^{-1}-\varepsilon)/\eta\rceil$, and there are $N+1$ rounds of iteration in the algorithm. First, Lemma~\ref{lem-max-E-pi} implies that very few vertices are matched in Round-$0$ (at most $\eta n$ of them), so our main focus will be on Rounds $1$ through $N$. Our strategy is to show that for each $1\le k\le N$, in Round-$k$, either $|U_k\setminus U_{k-1}|$ is small (i.e., the algorithm performs very few matchings), or the fraction of mismatched vertices $\#\{i\in U_k\setminus U_{k-1}:\widetilde{\pi}(i)\neq \pi^*(i)\}/|U_k\setminus U_{k-1}|$ is small (i.e., the mismatches account for only a small portion). More precisely, for $1\le k\le N$, we call Round-$k$ negligible, if $|U_k\setminus U_{k-1}|\le N^{-1}\eta n$, and non-negligible otherwise. We will prove the following crucial proposition which implies Proposition~\ref{prop-wrong-pairs-are-rare}.
	
	\begin{proposition}\label{prop-each-step-is-good}
		There exists $\eta_1=\eta_1(\alpha,\lambda,\varepsilon,\delta)$, such that for any $\eta<\eta_1$ and $1\le k\le N$,
		\begin{equation}\label{eq-each-step-is-nice}
			\mathcal P^*\left[\mathbf{1}(\text{Round-$k$ is non-negligible})\cdot \frac{\#\{i\in U_k\setminus U_{k-1}:\widetilde{\pi}(i)\neq \pi^*(i)\}}{|U_k\setminus U_{k-1}|}\le \frac{\zeta}{2}\right]=1-o(1)\,.
		\end{equation}
	\end{proposition}
	
	\begin{proof}[Proof of Proposition~\ref{prop-wrong-pairs-are-rare} assuming Proposition~\ref{prop-each-step-is-good}]
		Choose $\eta_0=\eta_0(\alpha,\lambda,\varepsilon,\delta)$ such that $\eta_0\le \zeta/4$ and $\eta_0<\eta_1(\alpha,\lambda,\varepsilon,\delta)$ as in Proposition~\ref{prop-each-step-is-good}. Consider the output $(U_N,\widetilde\pi)$ for the iterative matching algorithm with $\eta<\eta_0$. It is clear that the number of mismatched vertices in $U_N$ produced by negligible rounds is at most $N\cdot N^{-1}\eta n=\eta n$. Additionally, by Lemma~\ref{lem-max-E-pi}, with high probability $|U_0|\le \eta n$. By Proposition~\ref{prop-each-step-is-good} together with a union bound (since $N=O(1)$), with high probability, the mismatched vertices in any non-negligible round occupy at most a $\zeta/2$ fraction. Combining these results, we see that with high probability for $\eta<\eta_0$, the number of mismatched vertices in $U_N$ is at most $\eta n+\eta n+\frac{\zeta}{2}n<\zeta n$, as desired.
	\end{proof}
	
	Now we focus on proving Proposition~\ref{prop-each-step-is-good}. To analyze the probability in \eqref{eq-each-step-is-nice}, we need to introduce an appropriate filtration $\{\mathcal F_{k}\}_{1\le k\le N}$ and work with the conditional probability under typical realizations of $\mathcal F_{k},1\le k\le N$.  We define $\{\mathcal F_{k}\}_{1\le k\le N}$ as follows: for each $1\le k\le N$, $\mathcal F_k$ is the $\sigma$-field generated by the information of the true matching $\pi^*$, the sets $U_0\subset \dots \subset U_{k-1}$ identified by the algorithm in Round-$0$ to Round-$(k-1)$, the partial matching $\widetilde{\pi}$ on $U_{k-1}$, and the graph $\mathcal H_{\widetilde{\pi}}(U_{k-1})$ (recall that this is the intersection graph of $G_1,G_2$ through $\pi_{k-1}$ on $U_{k-1}$). We will argue that for $1\le k\le N$, under typical realizations of $\mathcal F_{k-1}$, the event in \eqref{eq-each-step-is-nice} holds with conditional probability $1-o(1)$. Given this, \eqref{eq-each-step-is-nice} follows by the iterative expectation theorem. 
	
To analyze the conditioning on $ \mathcal{F}_{t-1} $, we need to examine more carefully what this conditioning entails. In what follows, we will characterize this conditioning in terms of explicit events. The key observation is that these events are decreasing (in an appropriate sense). Using this fact, we can strategically apply the FKG inequality to eliminate the conditioning.

Now fix $ 1 \le k \le N $ together with a realization of $ \mathcal{F}_{k-1} $. Let $ \pi^* $ be the true matching, and let $ \mathtt{U}_0, \dots, \mathtt{U}_{k-1}, \tilde{\pi} $ be the realizations of $ U_0, \dots, U_{k-1}, \widetilde{\pi} $ identified by the algorithm. Also, let $ \mathtt{H}_{k-1} = (\mathtt{U}_{k-1}, \mathtt{E}_{k-1}) $ be the realization of $ \mathcal{H}_{\widetilde{\pi}}(U_{k-1}) = (U_{k-1}, \mathcal{E}_{\widetilde{\pi}}(U_{k-1})) $. We say that a realization of $ \mathcal{F}_{k-1} $ is \emph{good} if $ |\mathtt{E}_{k-1}| \le An $ (where $ A $ is defined as in Lemma~\ref{lem-max-E-pi}). By Lemma~\ref{lem-max-E-pi}, it follows immediately that a realization of $ \mathcal{F}_{k-1} $ is good with high probability. Therefore, to prove \eqref{eq-each-step-is-nice}, it suffices to show the same result holds when conditioning on any good realization (see the end of this subsection for deteailed arguments).

	Henceforth we assume the realization is good. Recalling the notation $\mathcal P_{\pi^*}[\cdot]=\mathcal P^*[\cdot\mid \pi^*]$, we have by definition that
	\[
	\mathcal P^*[\cdot\mid \mathcal F_{k-1}]=\mathcal P_{\pi^*}[\cdot\mid U_0=\mathtt{U}_0,\dots,U_{k-1}=\mathtt{U}_{k-1},\widetilde\pi=\tilde\pi,\mathcal H_{\tilde\pi}(\mathtt{U}_{k-1})=\mathtt{H}_{k-1}]\,.
	\]
	We first address the conditions $U_0=\mathtt{U}_0,\dots,U_{k-1}=\mathtt{U}_{k-1}$ and $\widetilde\pi=\tilde\pi$. By the rule of our algorithm, for any $0\le t\le k-1$, $U_t$ is the set of vertices with loads at least $A/\eta-t\eta$ in $\mathcal H_{\pi_t}$. In particular, the load function $\partial\theta_{\widetilde\pi}$ of the induced subgraph $\mathcal H_{\widetilde\pi}(U_t)$ is the same as $\partial\theta_{\pi_t}$ restricting on $U_t$, since in $\mathcal H_{\pi_t}$ no vertex in $U_t$ receives load from outside due to the balanced condition \eqref{eq-balanced-load}. Therefore, for any $\mathtt{U}_0,\dots,\mathtt{U}_{k-1}$ serving as realizations of $U_0,\dots,U_{k-1}$, they satisfy
	\begin{equation}\label{eq-loads-in-Ut-are-large}
		\partial\theta_{\widetilde\pi}(i)\ge A/\eta-t\eta\,,\ \forall\ 0\le t\le k-1,i\in \mathtt{U}_t\,.
	\end{equation}
	Additionally, since $U_t$ maximizes the size of $V_{\pi_t}^t$ over any matching $\pi_t$ satisfying $\pi_t\mid_{U_{t-1}}=\widetilde\pi$ (and further minimize $V_{\pi_t}^t$ under $\prec$ when there is a tie), we conclude that for $U_0=\mathtt{U}_0,\dots, U_{k-1}=\mathtt{U}_{k-1}$ to occur, the following event must hold:
	\begin{equation*}\label{eq-no-larger-Ut}
		\begin{aligned}
			\mathcal B=&\ \mathcal B(\mathtt{U}_0,\dots,\mathtt{U}_{k-1},\tilde\pi)\\
			\equiv&\ \bigcap_{0\le t\le k-1}\{\not\exists\ \pi_t'\in \operatorname{S}_n \text{ s.t. }\pi_t'\mid_{U_{t-1}}=\widetilde\pi,V^{\pi_t'}_t\prec \mathtt{V}_t,|V_t^{\pi_t'}|\ge  |\mathtt{V}_t|\text{ or }V_t^{\pi_t'}\succ \mathtt{V}_t,|V_{t}^{\pi_t'}|>|\mathtt{V}_t|\}\,.
		\end{aligned}
	\end{equation*}
	Here, $\mathtt{V}^t\equiv \mathtt{U}^t\setminus \mathtt{U}^{t-1},0\le t\le k-1$. 
    
Conversely, given that \eqref{eq-loads-in-Ut-are-large} holds and $ \mathcal{B}(\mathtt{U}_0, \dots, \mathtt{U}_{k-1}, \tilde{\pi}) $ is true, a simple induction argument shows that the algorithm indeed identifies $ U_0 = \mathtt{U}_0, \dots, U_{k-1} = \mathtt{U}_{k-1} $ and $ \widetilde{\pi} = \tilde{\pi} $ before Round-$ k $. Therefore, conditioning on $ U_0 = \mathtt{U}_0, \dots, U_{k-1} = \mathtt{U}_{k-1} $ and $ \widetilde{\pi} = \tilde{\pi} $ is equivalent to conditioning on \eqref{eq-loads-in-Ut-are-large} and $ \mathcal{B}(\mathtt{U}_0, \dots, \mathtt{U}_{k-1}, \tilde{\pi}) $. Moreover, note that \eqref{eq-loads-in-Ut-are-large} follows from $ \{\mathcal{H}_{\tilde{\pi}}(\mathtt{U}_{k-1}) = \mathtt{H}_{k-1}\} $, provided that $ \mathtt{H}_{k-1} $ is a valid realization of $ \mathcal{H}_{\widetilde{\pi}}(U_{k-1}) $. As a result, we obtain:
\begin{equation}\label{eq-F_k-1}
\mathcal{P}^*[\cdot \mid \mathcal{F}_{k-1}] = \mathcal{P}_{\pi^*}[\cdot \mid \mathcal{B}, \mathcal{H}_{\tilde{\pi}}(\mathtt{U}_{k-1}) = \mathtt{H}_{k-1}].
\end{equation}

	Next we turn to the event $\{\mathcal H_{\tilde\pi}(\mathtt{U}_{k-1})=\mathtt{H}_{k-1}\}$. Recall the Bernoulli indicators $I_{i,j},J^1_{i,j},J^2_{i,j}$ for $(i,j)\in \operatorname{E}$ in the definition of $\mathcal P^*$ on Page \pageref{eq-IJJ}. Specifically, the edge sets of $G_1$ and $G_2$ are given by that $$(i,j)\in E(G_1)\iff I_{i,j}J_{i,j}^1=1,\quad (i,j)\in E(G_2)\iff I_{(\pi^*)^{-1}(i),(\pi^*)^{-1}(j)}J_{i,j}^2=1\,.$$
	Therefore, for any graph $\mathtt{H}_{k-1}$ on $\mathtt{U}_{k-1}$ with edge set $\mathtt{E}_{k-1}$, letting $\sigma=(\pi^*)^{-1}\circ\widetilde\pi:\mathtt{U}_{k-1}\to [n]$, we have $\{\mathcal H_{\widetilde\pi}(U_{k-1})=\mathtt{H}_{k-1}\}$ is equivalent to $\mathcal E_1\cap \mathcal E_2$, where
	\begin{align*}
		\mathcal E_1=\mathcal E_1(\mathtt{E}_{k-1})&\equiv \{I_{i,j}J_{i,j}^1I_{\sigma(i),\sigma(j)}J^2_{\sigma(i),\sigma(j)}=1,\forall (i,j)\in E_0\}\\&=\{I_{i,j}=J^1_{i,j}=I_{\sigma(i),\sigma(j)}=J^2_{\sigma(i),\sigma(j)}=1,\forall (i,j)\in E_0\}\,,
	\end{align*}
	and
	\[
	\mathcal E_2=\mathcal E_2(\mathtt{U}_{k-1},\mathtt{E}_{k-1})= \{I_{i,j}J_{i,j}^1I_{\sigma(i),\sigma(j)}J^2_{\widetilde{\pi}(i),\widetilde{\pi}(j)}=0,\forall (i,j)\in \operatorname{E}(\mathtt{U}_{k-1})\setminus \mathtt{E}_{k-1}\}\,.
	\]
	Thus, denoting $\mathcal P_{\pi^*,\mathtt{E}_{k-1}}[\cdot]=\mathcal P_{\pi^*}[\cdot\mid \mathcal E_1]$, we can rewrite \eqref{eq-F_k-1} as
	\begin{align}
		\mathcal P^*[\cdot\mid \mathcal F_{k-1}]=\mathcal P_{\pi^*}[\cdot\mid \mathcal B, \mathcal E_1, \mathcal E_2]=\mathcal P_{\pi^*,\mathtt{E}_{k-1}}[\cdot\mid \mathcal B\cap \mathcal E_2]\,.\label{eq-conditioning-description}
	\end{align}
	
	The advantage of expressing $\mathcal P^*[\cdot\mid \mathcal F_{k-1}]$ in this form is that if we
	view $\mathcal P_{\pi^*}$ as a product measure over the indicators $I_{i,j},J_{i,j}^1,J_{i,j}^2$, then $\mathcal P_{\pi^*,\mathtt{E}_{k-1}}$ remains a product measure on a reduced sample space (i.e., the space of indicators not involved in $\mathcal E_1$). Moreover, both the events $\mathcal B$ and $\mathcal E_1$ are decreasing with respect to these indicators (for $\mathcal B$ this follows from Lemma~\ref{lem-load-monotonicity}, and for $\mathcal E_1$ this is obvious). Hence from the FKG inequality, when we aim to upper-bound $\mathcal P^*[\mathcal I\mid \mathcal F_{k-1}]$ for some increasing event $\mathcal I$, we can omit the conditioning on $\mathcal B\cap \mathcal E_1$, resulting in significant simplifications.
	
	We end this subsection with yet another layer of reduction. In order to upper bound the conditional probability of the event in \eqref{eq-each-step-is-nice}, we consider the following relaxation. Denote by $[l_k,u_k)$ the interval $I_k$ as defined in the algorithm, then $0<u_k-l_k\le \eta$. Define $\mathcal U=\mathcal U(\mathtt{U}_{k-1},\tilde\pi)$ as the event that there exists $U\supset \mathtt{U}_{k-1}$ and a partial matching $\check\pi$ on $U$ extending $\tilde\pi$ with the following three properties:
	\begin{itemize}
		\item [(i)] $|U\setminus \mathtt{U}_{k-1}|\ge N^{-1}\eta n$, and $\#\{i\in U\setminus \mathtt{U}_{k-1}:\check\pi(i)\neq \pi^*(i)\}\ge \frac\zeta 2\cdot |U\setminus \mathtt{U}_{k-1}|$. 
		\item [(ii)] $|\mathcal E_{\check\pi}(U)\setminus \mathcal E_{\check\pi}(\mathtt{U}_{k-1})|\ge l_k|U\setminus \mathtt{U}_{k-1}|$;
		\item [(iii)] For any $W$ such that $\mathtt{U}_{k-1}\subset W\subset U$, $|\mathcal E_{\check\pi}(W)\setminus \mathcal E_{\check\pi}(\mathtt{U}_{k-1})|\le u_k|W\setminus \mathtt{U}_{k-1}|$.
	\end{itemize}
	In light of the following reduction arguments, we may shift our attention to show that $\mathcal P^*[\mathcal U\mid \mathcal F_{k-1}]=o(1)$ for good realizations (but be aware that $\mathcal U$ is not an increasing event).
	
	\begin{proposition}\label{prop-reduced-conditional-probability}
		There exists $\eta_1=\eta_1(\alpha,\lambda,\varepsilon)>0$ such that the following is true whenever $\eta<\eta_1$: For any $1\le k\le N$ and any good realization of $\mathcal F_{k-1}$, let $\mathcal U$ be defined as above, then $$\mathcal P^*[\mathcal U\mid \mathcal F_{k-1}]=\mathcal P_{\pi^*,\mathtt{E}_{k-1}}[\mathcal U\mid \mathcal B\cap \mathcal E_2]=o(1)\,.$$	
	\end{proposition} 
	
	\begin{proof}[Proof of Proposition~\ref{prop-each-step-is-good} assuming Proposition~\ref{prop-reduced-conditional-probability}]
    Demote $\mathcal W$ as the event in \eqref{eq-each-step-is-nice}. We claim that $\mathcal W^c$ is contained in $\mathcal U$. To verify this, assume that $\mathcal W$ fails, then we choose $U=U_k$ and $\check\pi=\widetilde\pi$ as identified by the algorithm in Round-$k$. The failure of $\mathcal W$ implies $(U,\check\pi)$ satisfies Item-(i). Additionally, by definition $U_k$ is the set of vertices with loads at least $l_k$ in $\mathcal H_{\pi_k}$, and thus the maximizer of $f_{l_k}^{\widetilde\pi}$. In particular, 
		\[
		f_{l_k}^{\widetilde\pi}(U_k)=l_k|\mathcal E_{\widetilde\pi}(U_k)|-|U_k|\ge f_{l_k}^{\widetilde\pi}(\mathtt{U}_{k-1})=l_k|\mathcal E_{\widetilde\pi}(\mathtt{U}_{k-1})|-|\mathtt{U}_{k-1}|\,,
		\] 
	rearranging which implies Item-(ii). Item-(iii) follows from a similar reasoning: since $\mathtt{U}_{k-1}$ is the set of vertices in $\mathcal H_{\widehat\pi}(U)$ that have loads at least $u_k$, and thus the maximizer of $f_{u_k}^{\widetilde\pi}$, we have
		$
		f_{u_k}^{\widetilde\pi}(\mathtt{U}_{k-1})\le f_{u_k}^{\widetilde\pi}(W)
		$ for any $\mathtt{U}_{k-1}\subset W\subset U_k$, indicating the desired inequality. This verifies the claim, and thus from Proposition~\ref{prop-reduced-conditional-probability}, we have for any good realization of $\mathcal F_{k-1}$, the conditional probability of $\mathcal W$ under $\mathcal F_{k-1}$ is $1-o(1)$.
		
		Recalling that a realization of $\mathcal F_{k-1}$ is good, if the induced subgraph $\mathcal H_{\tilde\pi}(\mathtt{U}_{k-1})$ has no more than $An$ edges, which happens with high probability by Lemma~\ref{lem-max-E-pi}. Therefore, by the iterative expectation formula, we have
        \begin{align*}
\mathcal P^*[\mathcal W]=\mathbb{E}\big[\mathcal P_{\pi^*,\mathtt E_{k-1}}[\mathcal W\mid \mathcal F_{k-1}]\big]\ge\mathbb{E}\big[\mathcal P_{\pi^*,\mathtt E_{k-1}}[\mathcal W\mid \mathcal F_{k-1}]\mathbf{1}\{\mathcal F_{k-1}\text{ is good}\}\big]\ge 1-o(1)\,,
        \end{align*}
        as desired.
	\end{proof}
	

	\subsection{The orbit decomposition under conditioning}\label{subsec-orbit-decomposition-conditioning}
	In this subsection, we prepare for the proof of Proposition~\ref{prop-reduced-conditional-probability}. Again, we fix $1\le k\le N$ together with a good realization 
	\[
	(\mathtt{U}_0,\dots,\mathtt{U}_{k-1},\tilde\pi,\mathtt{H}_{k-1}=(\mathtt{U}_{k-1},\mathtt{E}_{k-1}))
	\]
	of $\mathcal F_{k-1}$. Our strategy is to apply a union bound to control $\mathcal P[\mathcal U\mid \mathcal F_{k-1}]$. Towards this end, we will need to understand the conditional probability that a particular pair $(U,\check\pi)$ satisfies conditions (i)-(iii) in the event $\mathcal U$. The main probability loss comes from condition (ii), and thus we focus on the conditional distribution of $|\mathcal E_{\check\pi}(U)\setminus \mathcal E_{\tilde\pi}(\mathtt{U}_{k-1})|$ under $\mathcal F_{k-1}$. This turns out to be closely related with the orbit decomposition in Section~\ref{subsec-orbit-decomposition}. The difference here is that instead of looking at edge cycles in the whole set $\operatorname{E}$, we only restrict to $\operatorname{E}(U)\setminus \operatorname{E}(\mathtt{U}_{k-1})$, the set of edges that are within $U$ but not within $\mathtt U_{k-1}$. Such a restriction creates additional chain structures that did not appear in Section~\ref{subsec-orbit-decomposition}, as we elaborate in more detail below.
	
Fix a subset $ U \subset [n] $ with $ U \supset \mathtt{U}_{k-1} $, and a partial matching $ \check{\pi} $ on $ U $ such that $ \check{\pi} \mid_{\mathtt{U}_{k-1}} = \tilde{\pi} $. Choose $ \pi $ as an arbitrary extension of $ \check{\pi} $ in $ \operatorname{S}_n $, and consider the orbits in $ \mathcal{O}_{\pi} $ induced by $ \sigma = \pi^* \circ \pi^{-1} $, as defined in Section~\ref{subsec-orbit-decomposition}. Recall that these orbits are edge cycles in $ \operatorname{E} $ induced by the permutation $ \Sigma $. If we restrict to the set $ \operatorname{E}(U) \setminus \operatorname{E}(\mathtt{U}_{k-1}) $, the cycles entirely contained within this set remain unchanged, while other cycles break into chains of the form $ (e_1, \dots, e_k) $, where $ \Sigma(e_i) = e_{i+1} $ for $ 1 \le i \le k-1 $, and $ \Sigma(e_k), \Sigma^{-1}(e_1) \notin \operatorname{E}(U) \setminus \operatorname{E}(\mathtt{U}_{k-1}) $. We denote the new set of orbits (including both cycles and chains) in $ \operatorname{E}(U) \setminus \operatorname{E}(\mathtt{U}_{k-1}) $ as $ \mathcal{O}_{\check{\pi}}(\mathtt{U}_{k-1}, U) $. Furthermore, we call a chain $ (e_1, \dots, e_k) $ in $ \mathcal{O}_{\check{\pi}}(\mathtt{U}_{k-1}, U) $ \textbf{free} if neither $ \Sigma(e_k) $ nor $ \Sigma^{-1}(e_1) $ belongs to $ \mathtt{E}_{k-1} $; otherwise, we call it \textbf{confined}. All these definitions are well-defined in the sense that they do not depend on the specific choice of the extension $ \pi $.

The advantage of defining these orbits is that, under $ \mathcal{P}_{\pi^*, \mathtt{E}_{k-1}} $, the contributions to $ |\mathcal{E}_{\check{\pi}}(U) \setminus \mathcal{E}_{\check{\pi}}(\mathtt{U}_{k-1})| $ from different orbits in $ \mathcal{O}_{\check{\pi}}(\mathtt{U}_{k-1}, U) $ are independent under $\mathcal P_{\pi^*,\mathtt{E}_{k-1}}$. This will become useful when analyzing the conditional probability $ \mathcal{P}^*[\cdot \mid \mathcal{F}_{k-1}] = \mathcal{P}_{\pi^*, \mathtt{E}_{k-1}}[\cdot \mid \mathcal{B} \cap \mathcal{E}_2] $ after dropping the conditioning on $ \mathcal{B} \cap \mathcal{E}_2 $ using the FKG inequality.

At this point, we need to introduce a truncation. Define $ L = L(\alpha, \varepsilon) $ as follows:
\begin{equation}\label{eq-L}
L = 
\begin{cases}
\lfloor (1 - \alpha)^{-1} \rfloor, & \text{if } \alpha < 1\,; \\
\lceil \frac{2 + 2\varepsilon}{\varepsilon} \rceil, & \text{if } \alpha = 1\,.
\end{cases}
\end{equation}
As before, let $ E_s $ and $ E_k $ for $ k = 1, 2, \dots, n^2 $ denote the number of edges in $ \mathcal{E}_{\check{\pi}}(U) \setminus \mathcal{E}_{\check{\pi}}(\mathtt{U}_{k-1}) $ that belong to special cycles and non-special cycles of length $ k = 1, 2, \dots, n^2$ in $ \mathcal{O}_{\check{\pi}}(\mathtt{U}_{k-1}, U) $, respectively. Let $ E_{>L} = E_{L+1} + \cdots + E_{n^2} $ represent the number of edges in cycles of length greater than $ L $.
Additionally, let $ E_c $, $ E_c^f $, and $ E_c^c $ denote the number of edges in $ \mathcal{E}_{\check{\pi}}(U) \setminus \mathcal{E}_{\check{\pi}}(\mathtt{U}_{k-1}) $ that belong to chains, free chains, and confined chains in $ \mathcal{O}_{\check{\pi}}(\mathtt{U}_{k-1}, U) $, respectively. Then, we have $ E_c = E_c^f + E_c^c $, and
\[
|\mathcal{E}_{\check{\pi}}(U) \setminus \mathcal{E}_{\check{\pi}}(\mathtt{U}_{k-1})| = E_s + E_c + E_1 + \cdots + E_L + E_{>L}.
\]
The following lemma is an analogue of \cite[Lemma 2.4]{DD23b} under the measure $ \mathcal{P}_{\pi^*, \operatorname{E}_{k-1}} $.
	
	\begin{lemma}\label{lem-tail-prob-conditioning}
		Write $\alpha_k=\frac{k-1}{k}$ for $1\le k\le L$ and $\alpha_{L+1}=\alpha\wedge\frac{L}{L+1}$. Assume that $\mathcal F_{k-1}$ is a good realization (i.e. $|\mathtt{E}_{k-1}|\le An$), then the following hold for some $M=o(n\log n)$ and any $0\le x\le O(n)$:
		\begin{align}
			\label{eq-tail-Es}\mathcal P_{\pi^*,\mathtt{E}_{k-1}}[E_s\ge x]&\le \exp(M-x\log n)\,,\\
			\label{eq-tail-Ek}\mathcal P_{\pi^*,\mathtt{E}_{k-1}}[E_k\ge x]&\le \exp(M-\alpha_k x\log n)\,,\forall 1\le k\le L\,,\\
			\label{eq-tail-E>L}\mathcal P_{\pi^*,\mathtt{E}_{k-1}}[E_{>L}\ge x]&\le \exp(M-\alpha_{L+1}x\log n)\,,\\
			\label{eq-tail-Ec}\mathcal P_{\pi^*,\mathtt{E}_{k-1}}[E_c\ge x]&\le \exp(M-\alpha x\log n)\,.
		\end{align}
	\end{lemma}
	
	\begin{remark}
The statement of Lemma~\ref{lem-tail-prob-conditioning} is essentially the same as \cite[Lemma 2.4]{DD23b}, and the proof follows a similar approach: we first compute the exponential moments of these random variables and then apply Chebyshev's inequality. The key difference here lies in \eqref{eq-tail-Ec}, where we need to handle the confined chains, which necessitates the condition $ |\mathtt{E}_{k-1}| = O(n) $. The full details are provided in the appendix.

	\end{remark}
	
Lemma~\ref{lem-tail-prob-conditioning} will be used to control the probability in Item-(ii). At first glance, the main contribution to the probability that $ |\mathcal{E}_{\check{\pi}}(U) \setminus \mathcal{E}_{\check{\pi}}(\mathtt{U}_{k-1})| $ becomes large seems to come from the case where $ E_1 $ itself is large, since its tail decays the slowest. However, Item-(iii) imposes additional constraints on the random variables $ E_1, E_2, \dots, E_L $, preventing this from happening, as we discuss below.

For a partial matching $ \check{\pi} $ on $ U \supset \mathtt{U}_{k-1} $ with $ \check{\pi} \mid_{\mathtt{U}_{k-1}} = \tilde{\pi} $, we choose an arbitrary extension $ \pi $ of $ \check{\pi} $ in $ \operatorname{S}_n $. We call a cycle of $ \sigma = (\pi^*)^{-1} \circ \pi $ \textbf{complete} if it is entirely contained within $ U $. Note that the complete cycles do not depend on the specific choice of $ \pi $. Furthermore, we define $ N_k = N_k(\mathtt{U}_{k-1}, U, \check{\pi}) $ for $ k = 1, 2, \dots, L $ as the set of vertices $ i \in U \setminus \mathtt{U}_{k-1} $ such that $ i $ belongs to a complete cycle of length $ k $. We make the following crucial observation.

	\begin{lemma}\label{lem-constraints-on-Ek}
		With the above notations, given that Item-(iii) holds, we have 
		\begin{equation}\label{eq-constraints-on-Ek}
			\sum_{t=1}^kE_t\le u_k\sum_{t=1}^k |N_t|,\ \forall\ 1\le k\le L\,.
		\end{equation}
	\end{lemma}
	\begin{proof}
Fix an arbitrary $ 1 \le k \le L $. For any $ (x, y) \in \operatorname{E}(U) \setminus \operatorname{E}(\mathtt{U}_{k-1}) $ that belongs to a non-special cycle in $ \mathcal{O}_{\check{\pi}}(\mathtt{U}_{k-1}, U) $ with length no greater than $ k $, we have both $ x $ and $ y $ belonging to a complete cycle of length at most $ k $. Consequently, if either $ x $ or $ y $ is in $ U \setminus \mathtt{U}_{k-1} $, then it must lie in $ N_1 \cup \cdots \cup N_k $. Therefore, we have
\[
\sum_{t=1}^k E_t \le |\mathcal{E}_{\check{\pi}}(\mathtt{U}_{k-1} \cup N_1 \cup \cdots \cup N_k) \setminus \mathcal{E}_{\check{\pi}}(\mathtt{U}_{k-1})|,
\]
and by Item-(iii), this is bounded by the right-hand side of \eqref{eq-constraints-on-Ek}. This completes the proof.
	\end{proof}
	
We end this subsection with an enumerative upper bound on the set of partial matchings $ \check{\pi} $ with a given profile of $ |N_k| $ for $ k = 1, \dots, L $. Consider a triple $ (\mathtt{U}_{k-1}, U, \tilde{\pi}) $ with $ \mathtt{U}_{k-1} \subset U $, and let $ T = |U \setminus \mathtt{U}_{k-1}| $. For $ n_1, \dots, n_L \in \mathbb{N} $, we define $ \operatorname{S}_{\mathtt{U}_{k-1}, U, \tilde{\pi}}(n_1, \dots, n_L) $ as the set of partial matchings $ \check{\pi} $ on $ U $ with $ \check{\pi} \mid_{\mathtt{U}_{k-1}} = \tilde{\pi} $, and such that $ |N_k| = n_k $ for $ 1 \le k \le L $. The following lemma is an analogue of \cite[Lemma 2.5]{DD23b}, and its proof is provided in the appendix.

	\begin{lemma}\label{lem-enumeration}
		For any triple $(\mathtt{U}_{k-1},U,\tilde\pi)$, with the above notations, we have
		\begin{equation}\label{eq-enumeration-S(n1,nL)}
			|\operatorname{S}_{\mathtt{U}_{k-1},U,\tilde\pi}(n_1,\dots,n_L)|\le\exp\Big(\Big(T-n_1-\frac{n_2}{2}-\cdots-\frac{n_L}{L}\Big)\log n+O(n)\Big)\,.		
		\end{equation}
	\end{lemma}


	\subsection{Proof of Proposition~\ref{prop-reduced-conditional-probability}}\label{subsec-proof-or-reduction}
We are now ready to prove Proposition~\ref{prop-reduced-conditional-probability}. Building on the previous discussions, the proof proceeds in conceptually the same way as in \cite{DD23b}, except that we need to apply the FKG inequality at an appropriate point and deal with the calculations with more precision.

	\begin{proof}[Proof of Propostition~\ref{prop-reduced-conditional-probability}]
Fix $ 1 \le k \le N $ and a good realization of $ \mathcal{F}_{k-1} $. We will bound the conditional probability of $ \mathcal{U} $ under $ \mathcal{F}_{k-1} $ using a union bound. We have
\begin{align}
    \mathcal{P}^*[\mathcal{U} \mid \mathcal{F}_{k-1}] \le \sum_{U} \sum_{\check{\pi}} \mathcal{P}^*[\text{(ii) and (iii) hold for } (U, \check{\pi}) \mid \mathcal{F}_{k-1}],
\end{align}
where the sum is taken over $ U \supset \mathtt{U}_{k-1} $ with $ |U \setminus \mathtt{U}_{k-1}| \ge N^{-1} \eta n $, and $ \check{\pi} : U \to [n] $ is an extension of $ \tilde{\pi} $ that satisfies
\begin{equation}\label{eq-some-wrong-pair}
    \frac{\#\{i \in U \setminus \mathtt{U}_{k-1} : \check{\pi}(i) \neq \pi^*(i)\}}{|U \setminus \mathtt{U}_{k-1}|} \ge \frac{\zeta}{2}.
\end{equation}
For any fixed choice of $ U $, we will show that the sum over $ \check{\pi} $ is upper-bounded by $ \exp(-(\chi+o(1)) n \log n) $ for some positive constant $ \chi = \chi(\alpha, \lambda, \eta, \varepsilon, \delta) $. Assuming this is true, the desired result follows because we are only summing over $ \exp(O(n)) $ many sets $ U $.

		Now we fix a set $U\supset \mathtt{U}_{k-1}$ with $|U\setminus \mathtt{U}_{k-1}|=T\ge N^{-1}\eta n$. We may upper-bound the sum over $\check\pi$ as
		\begin{align}
			&\ \sum_{\check\pi}\mathcal P^*[\text{(ii), (iii) hold for }(U,\check\pi)\mid \mathcal F_{k-1}]\nonumber\\
			\le&\ \sum_{n_1,\dots,n_L}\sum_{\check\pi\in \operatorname{S}_{\mathtt{U}_{k-1},U,\tilde\pi}(n_1,\dots,n_L)}\mathcal P^*[\text{(ii), (iii) hold for }(U,\check\pi)\mid \mathcal F_{k-1}]\nonumber\\
			\le&\ \sum_{n_1,\dots,n_L}|\operatorname{S}_{\mathtt{U}_{k-1},U,\tilde\pi}(n_1,\dots,n_L)|\cdot \max_{\check\pi\in \operatorname{S}_{U_{k-1},U,\tilde\pi}(n_1,\dots,n_L)}\mathcal P^*[\text{(ii), (iii) hold for }(U,\check\pi)\mid \mathcal F_{k-1}]\,.\label{eq-enumeration-times-max-probability}
		\end{align}
		where the sum of $n_i$'s is over the tuples $(n_1,\dots,n_L)\in \{0,1,\dots,T\}^L$ with $n_1+\cdots+n_L\le T$ and $n_1\le (1-\zeta/2)T$ (the latter constraint comes from \eqref{eq-some-wrong-pair}). In order to control the maximal probability term in \eqref{eq-enumeration-times-max-probability}, we define $\Delta(n_1,\dots,n_L)$ as the set 
		\[
		\Big\{\mathbf{x}=(x_t)_{-1\le t\le L+1}\in [0,N^2]^{L+1}, \sum_{t=-1}^{L+1}x_t\ge l_kT,\sum_{t=1}^kx_t\le u_k\sum_{t=1}^kn_t,1\le k\le L\Big\}\,.
		\]
		From Lemma~\ref{lem-constraints-on-Ek} we conclude for any $\check\pi\in\operatorname{S}_{\mathtt{U}_{k-1},U,\tilde\pi}(n_1,\dots,n_L)$, it holds
		\begin{align}
			&\nonumber\ \mathcal P^*[\text{(ii), (iii) hold for }(U,\check\pi)\mid \mathcal F_{k-1}] \\\le&\ \sum_{\mathbf{x}\in \Delta(n_1,\dots,n_L)\cap \mathbb{N}^{L+3}}\mathcal P^*[E_s\ge x_{-1},E_s\ge x_0,E_t\ge x_t,\forall 1\le t\le L,E_{>L}\ge x_{L+1}\mid \mathcal F_{k-1}]\,.\label{eq-union-bound-over-x}
		\end{align}
		It is clear that the event in the right-hand side of \eqref{eq-union-bound-over-x} is increasing, we conclude from \eqref{eq-conditioning-description} and the FKG inequality that 
		\begin{align*}
			&\ \mathcal P^*[E_s\ge x_{-1},E_c\ge x_0,E_t\ge x_t,\forall 1\le t\le L, E_{>L}\ge x_{L+1}\mid \mathcal F_{k-1}]
			\\=&\ \mathcal P_{\pi^*,\mathtt{E}_{k-1}}[E_s\ge x_{-1},E_c\ge x_0,E_t\ge x_t,\forall -1\le t\le L,E_{>L}\ge x_{L+1}\mid \mathcal B\cap \mathcal E_2]\\
			\le&\  \mathcal P_{\pi^*,\mathtt{E}_{k-1}}[E_s\ge x_{-1},E_c\ge x_0,E_t\ge x_t,1\le t\le L,E_{>L}\ge x_{L+1}]\,.
		\end{align*}
		Since $E_s,E_c,E_1,\dots,E_L,E_{>L}$ are independent under $\mathcal P_{\pi^*,\mathtt{E}_{k-1}}$, Lemma~\ref{lem-tail-prob-conditioning} yields that the above probability is further bounded by
		\[
		\exp\big(o(n\log n)-(x_{-1}+\alpha x_0+\alpha_1x_1+\cdots+\alpha_Lx_L+\alpha_{L+1}x_{L+1})\log n\big)\,.
		\]
		Combining this with \eqref{eq-enumeration-times-max-probability}, \eqref{eq-union-bound-over-x} and Lemma~\ref{lem-enumeration}, we obtain that 
		\begin{align}
			\nonumber&\ \sum_{\check\pi}\mathcal P^*[\text{(ii), (iii) hold for }(U,\check\pi)\mid \mathcal F_{k-1}]\\
			\le&\ \exp(o(n\log n))\cdot \sup_{n_1,\dots,n_L, \mathbf{x}}\exp\left(\Big(T-\sum_{t=1}^L\frac{n_t}{t}-x_{-1}-\alpha x_0-\sum_{t=1}^L\alpha_tx_t-\alpha_{L+1}x_{L+1}\Big)\log n\right)\,,\label{eq-algebra-bound}
		\end{align}
		where the supremum is taken over $(n_1,\dots,n_L)\in [0,T]^L$ with $n_1+\dots+n_L\le T$, $n_1\le (1-\zeta/2)T$, and $\mathbf{x}=(x_t)_{-1\le t\le L+1}\in \Delta(n_1,\dots,n_L)$. Now, we pick $\eta_1=\alpha\varepsilon\zeta/8$. For $\eta<\eta_1$, the next lemma implies \eqref{eq-algebra-bound} is upper-bounded by $\exp(-(\chi+o(1)) n\log n)$ for some positive constant $\chi=\chi(\alpha,\lambda,\varepsilon,\eta)$. This completes the proof.
	\end{proof}
	\begin{lemma}\label{lem-algebra}
		For any $\eta<\alpha\varepsilon\zeta/8$, there exists $\chi=\chi(\alpha,\lambda,\varepsilon,\eta,\delta)>0$ such that for any $T\ge N^{-1}\eta n$, $n_1,\dots,n_L\ge 0$ with $n_1+\cdots+n_L\le T$ and $n_1\le (1-\zeta/2)T$, and any $\mathbf{x}=(x_{-1},x_0,x_1,\dots,x_L,x_{L+1})\in \Delta(n_1,\dots,n_L)$, it holds that
		\begin{equation}\label{eq-algebra}
			T-\sum_{t=1}^L\frac{n_t}{t}-x_{-1}-\alpha x_0-\sum_{t=1}^L\alpha_tx_t-\alpha_{L+1}x_{L+1}\le -\chi n\,.
		\end{equation}
	\end{lemma}
	\noindent
	This lemma is an analogue of \cite[Lemma 2.5]{DD23b} and its proof is purely elementary, so we leave it in the appendix as well. So far we have proved Theorem~\ref{thm-main-ER}-(i).

	\section{Proof of Negative result}\label{sec-negative}
	
In this section, we prove Theorem~\ref{thm-main-ER}-(ii). Fix an arbitrarily small $ \varepsilon > 0 $. Let $ (\pi^*, G^1, G^2) \sim \mathcal{P}^* $, and denote by $ U^{\operatorname{light}} $ the set of light vertices, i.e., vertices $ i \in [n] $ with $ \partial\theta_{\pi^*}(i) \le \alpha^{-1} - \varepsilon $. Our goal is to show that it is almost impossible to correctly recover the values of $ \pi^* $ for a positive fraction of vertices in $ U^{\operatorname{light}} $. To this end, we will prove a slightly stronger result: even if the information of $ W \equiv [n] \setminus U^{\operatorname{light}} $ and $ \pi^* \mid_W $ are known, it is still impossible to recover $ \pi^* $ on $\varepsilon n$ vertices of $ U^{\operatorname{light}} $.

	\begin{proposition}\label{prop-reduced-negative-result}
		For any $\varepsilon>0$ and $(\pi^*,G_1,G_2)\sim \mathcal P^*$, there does not exist an estimator $\hat{\pi}=\hat{\pi}(G_1,G_2,W,\pi^*\mid_{W})\in \operatorname{B}(U^{\operatorname{light}}, [n]\setminus \pi^*(W))$ such that $\hat\pi(i)=\pi^*(i)$ holds for at least $\varepsilon n$ many of $i\in U^{\operatorname{light}}$ with non-vanishing probability. 
	\end{proposition}
	
	Theorem~\ref{thm-main-ER}-(ii) follows readily from Proposition~\ref{prop-reduced-negative-result}. For a permutation $\pi$ on $[n]$, we define the shrinking permutation $\sigma$ of $\pi$ on a subset $U$ of $[n]$ by letting $\sigma(i)=\pi^{(k)}(i)$ for any $i\in U$, where $k$ is the minimal positive integer such that $\pi^{(k)}(i)\in U$. 
	
	\begin{proof}[Proof of Theorem~\ref{thm-main-ER}-(ii) assuming Proposition~\ref{prop-reduced-negative-result}]
		It suffices to show the claim for sufficiently small $\varepsilon$, hence we may fix a small $\varepsilon>0$ such that $\alpha^{-1}-\varepsilon$ is not an atom of $\mu_\lambda$. Assuming the contrary, then there exists an estimator $\widehat\pi=\widehat\pi(G_1,G_2)$ such that $\mathcal G_1\equiv \{\operatorname{ov}(\widehat\pi,\pi^*)\ge( F_{\lambda}(\alpha^{-1}-\varepsilon)+\varepsilon)n\}$ happens with non-vanishing probability under $\mathcal P^*$. Let $\sigma$ be the shrinking permutation of $(\pi^*)^{-1}\circ\widehat{\pi}$ on $U^{\operatorname{light}}$, and define the bijection $\hat\pi$ from $U^{\operatorname{light}}$ to $[n]\setminus\pi^*(W)$ as $\hat\pi=\pi^*\circ\sigma$ (note that $\hat\pi$ only depends on $\widehat\pi$, $W$ and $\pi^*\mid_W$). It is straightforward to check that $\hat\pi(i)=\pi^*(i)$  for any $i\in U^{\operatorname{light}}$ such that $\widehat\pi(i)=\pi^*(i)$. Therefore, for $\mathcal G_2=\{|V|\le (F_\lambda(\alpha^{-1}-\varepsilon)+\varepsilon/2)n\}$, we have $\#\{i\in U^{\operatorname{light}}:\hat\pi(i)=\pi^*(i)\}\ge \varepsilon n/2$ under $\mathcal G_1\cap \mathcal G_2$. However,  Proposition~\ref{prop-weak-convergence-of-empirical-measure} implies $\mathcal P^*[\mathcal G_2]=1-o(1)$, and thus $\mathcal P^*[\mathcal G_1\cap \mathcal G_2]$ is non-vanishing. This contradicts to the result of Proposition~\ref{prop-reduced-negative-result}. Thus, the desired results follows.  
	\end{proof}
	
We take a further step in our reduction by conditioning on the realization of $ U^{\operatorname{light}} $ and $ \pi^* \mid_{W} $. To simplify notation, for two sets $ U, W \subset [n] $ with $ U \cap W = \emptyset $, we write $ \operatorname{E}(U; W) = \operatorname{E}(U \cup W) \setminus \operatorname{E}(W) $ and $ \mathcal{E}_{\pi^*}(U; W) = \mathcal{E}_{\pi^*}(U \cup W) \setminus \mathcal{E}_{\pi^*}(W) $. Fix two sets $ \mathtt{U}_1, \mathtt{U}_2 \subset [n] $ with $ |\mathtt{U}_1| = |\mathtt{U}_2| $, and denote $ \mathtt{W}_i = [n] \setminus \mathtt{U}_i $ for $ i = 1, 2 $. For $ (\pi^*, G_1, G_2) \sim \mathcal{P}^* $, let $ G_1^{\mathtt{U}_1} $ (resp. $ G_2^{\mathtt{U}_2} $) be the subgraph of $ G_1 $ (resp. $ G_2 $) with edge set $ E(G_1) \cap \operatorname{E}(\mathtt{U}_1; \mathtt{W}_1) $ (resp. $ E(G_2) \cap \operatorname{E}(\mathtt{U}_2; \mathtt{W}_2) $). We denote the joint law of $ (\pi^* \mid_{\mathtt{U}_1}, G_1^{\mathtt{U}_1}, G_2^{\mathtt{U}_2}) $ as $ \mathcal{P}_{\mathtt{U}_1, \mathtt{U}_2}^* $. Additionally, we write $\mathcal H_\pi(\mathtt U_1;\mathtt W_1)$ for the $\pi^*$-intersection graph of $G_1^{\mathtt U_1}$ and $G_2^{\mathtt U_2}$. 

We further condition on $ (\pi^*, G_1, G_2) $ such that $ \pi^*(\mathtt{U}_1) = \mathtt{U}_2 $ and $ U \equiv \{i \in [n] : \partial\theta_{\pi^*}(i) \le \alpha^{-1} - \varepsilon\} = \mathtt{U}_1 $. We denote by $ \mathcal{P}_{\mathtt{U}_1, \mathtt{U}_2}^{\dagger, *} $ the conditional version of $ \mathcal{P}_{\mathtt{U}_1, \mathtt{U}_2}^* $. Under this conditioning, the following hold in $ \mathcal{H}_{\pi^*} $: (i) any vertex in $ \mathtt{W}_1 $ has load greater than $ \alpha^{-1} - \varepsilon $ in $ \mathcal{H}_{\pi^*}(\mathtt{W}_1) $, and (ii) any vertex in $ \mathtt{U}_1 $ has load at most $ \alpha^{-1} - \varepsilon $.
We also note that Item-(i) is measurable with respect to the subgraph $ \mathcal{H}_{\pi^*}(\mathtt{W}_1) $, and thus is independent of $ (G_1^{\mathtt{U}_1}, G_2^{\mathtt{U}_2}) $ under $\mathcal P^{\dagger,*}_{\mathtt U_1,\mathtt U_2}$ (since we restrict that $\pi^*(\mathtt U_1)=\mathtt U_2$). Moreover, given that Item-(i) is true, Item-(ii) is equivalent to
\begin{equation}\label{eq-density-upper-bound}
    |\mathcal{E}_{\pi^*}(U; \mathtt{W}_1)| \le (\alpha^{-1} - \varepsilon) |U|, \quad \forall U \subset \mathtt{U}_1.
\end{equation}
Hence, we conclude that
\[
\mathcal{P}_{\mathtt{U}_1, \mathtt{U}_2}^{\dagger, *}[\cdot] = \mathcal{P}_{\mathtt{U}_1, \mathtt{U}_2}^*[\cdot \mid \eqref{eq-density-upper-bound} \text{ holds}].
\]

	For $\sigma_1,\sigma_2\in \operatorname{B}(\mathtt U_1,\mathtt U_2)$, we write $\operatorname{ov}(\sigma_1,\sigma_2)$ for the number of indices $i\in \mathtt U_1$ such that $\sigma_1(i)=\sigma_2(i)$. 
	To prove Proposition~\ref{prop-reduced-negative-result}, it suffices to prove the following stronger result.
	\begin{proposition}\label{prop-reduced-negative-result-stronger}
		For any subsets $\mathtt U_1,\mathtt U_2$ with $|\mathtt U_1|=|\mathtt U_2|$, given $(\pi^*\mid_{\mathtt U_1},G_1^{\mathtt U_1},G_2^{\mathtt U_2})\sim \mathcal P_{\mathtt U_1,\mathtt U_2}^{\dagger,*}$, then there is no estimator $\hat\sigma=\hat\sigma(G_1^{\mathtt U_1},G_2^{\mathtt U_2})\in \operatorname{B}(\mathtt U_1,\mathtt U_2)$ such that $\operatorname{ov}(\hat\sigma,\pi^*\mid_{\mathtt U_1})\ge \varepsilon n$ happens with non-vanishing probability.
	\end{proposition}
	
In what follows, we focus on the proof of Proposition~\ref{prop-reduced-negative-result-stronger}. This appears very similar to the impossibility result for partial recovery in \cite{DD23b}. While the key ingredient in \cite{DD23b} is the fact that the maximal subgraph density of $ \mathcal{H}_{\pi^*} $, i.e., the maximal load in $ \mathcal{H}_{\pi^*} $, does not exceed $ \alpha^{-1} - \varepsilon $ with high probability, here we have \eqref{eq-density-upper-bound} serving as an analogue. Indeed, we will follow conceptually similar arguments to those in \cite{DD23b}, with a few technical modifications to prove Proposition~\ref{prop-reduced-negative-result-stronger}. In this section, we only outline the main approach and highlight the necessary adjustments. For the technical details that involve minor adaptations from \cite{DD23b}, we defer them to the appendix.

To begin, we perform an appropriate truncation of $ \mathcal{P}_{\mathtt{U}_1, \mathtt{U}_2}^{\dagger, *} $. We define a suitable event $ \mathcal{A}_1 $ that is measurable with respect to $ \mathcal{H}_{\pi^*}(\mathtt{U}_1; \mathtt{W}_1) $, and let
\[
\mathcal{P}_{\mathtt{U}_1, \mathtt{U}_2}^{\ddagger, *} \equiv \mathcal{P}_{\mathtt{U}_1, \mathtt{U}_2}^{\dagger, *}[\cdot \mid \mathcal{A}_1] = \mathcal{P}_{\mathtt{U}_1, \mathtt{U}_2}^*[\cdot \mid \mathcal{A}],
\]
where $ \mathcal{A} \equiv \mathcal{A}_1 \cap \{\eqref{eq-density-upper-bound} \text{ holds}\} $. Roughly speaking, $ \mathcal{A}_1 $ includes several desirable properties of the true intersection graph $ \mathcal{H}_{\pi^*}(\mathtt{U}_1; \mathtt{W}_1) $ that facilitate the later proof. Since this truncation is only for technical reasons, we defer the precise definition of $ \mathcal{A}_1 $ to the appendix. Additionally, we will show in the appendix that $ \mathcal{A}_1 $ is typical under $ \mathcal{P}_{\mathtt{U}_1, \mathtt{U}_2}^{\dagger, *} $, so that
\[
\operatorname{TV}(\mathcal{P}_{\mathtt{U}_1, \mathtt{U}_2}^{\dagger, *}, \mathcal{P}_{\mathtt{U}_1, \mathtt{U}_2}^{\ddagger, *}) = o(1).
\]
Therefore, we may switch from $ \mathcal{P}_{\mathtt{U}_1, \mathtt{U}_2}^{\dagger, *} $ to $ \mathcal{P}_{\mathtt{U}_1, \mathtt{U}_2}^{\ddagger, *} $ in the proof of Proposition~\ref{prop-reduced-negative-result-stronger}.

The next step is to reduce Proposition~\ref{prop-reduced-negative-result-stronger} to a property of the posterior measure on $ \pi^* \mid_{\mathtt{U}_1} $ given $ (G_1^{\mathtt{U}_1}, G_2^{\mathtt{U}_2}) $. Denote by $ \mathcal{P}_{\mathtt{U}_1, \mathtt{U}_2}^{\ddagger} $ the marginal of $ \mathcal{P}_{\mathtt{U}_1, \mathtt{U}_2}^{\ddagger, *} $ on $ (G_1^{\mathtt{U}_1}, G_2^{\mathtt{U}_2}) $. For $ (G_1^{\mathtt{U}_1}, G_2^{\mathtt{U}_2}) \sim \mathcal{P}_{\mathtt{U}_1, \mathtt{U}_2}^{\ddagger} $, let $ \mu = \mu(\mathtt{U}_1, \mathtt{U}_2, G_1^{\mathtt{U}_1}, G_2^{\mathtt{U}_2}) $ be the posterior measure of $ \pi^* \mid_{\mathtt{U}_1} \in \operatorname{B}(\mathtt{U}_1, \mathtt{U}_2) $ under $ \mathcal{P}_{\mathtt{U}_1, \mathtt{U}_2}^{\ddagger, *} $ given $ (G_1^{\mathtt{U}_1}, G_2^{\mathtt{U}_2}) $. By definition, for any $ \sigma \in \operatorname{B}(\mathtt{U}_1, \mathtt{U}_2) $, we have:
\begin{align}\label{eq-posterior}
    \mu(\sigma) = \mathcal{P}_{\mathtt{U}_1, \mathtt{U}_2}^{\ddagger, *}[\pi^* \mid_{\mathtt{U}_1} = \sigma \mid G_1^{\mathtt{U}_1}, G_2^{\mathtt{U}_2}] = \frac{\mathcal{P}_{\mathtt{U}_1, \mathtt{U}_2}^{\ddagger, *}[\pi^* \mid_{\mathtt{U}_1} = \sigma, G_1^{\mathtt{U}_1}, G_2^{\mathtt{U}_2}]}{\mathcal{P}_{\mathtt{U}_1, \mathtt{U}_2}^{\ddagger}[G_1^{\mathtt{U}_1}, G_2^{\mathtt{U}_2}]}.
\end{align}
We claim that it suffices to show the following anti-concentration property of $ \mu $.

	\begin{proposition}\label{prop-posteior-anti-concentration}
		For any $\mathtt U_1,\mathtt U_2\subset [n]$ with $|\mathtt U_1|=|\mathtt U_2|$, with probability $1-o(1)$ over $(G_1^{\mathtt U_1},\mathtt G_2^{\mathtt U_2})\sim \mathcal P_{\mathtt U_1,\mathtt U_2}^{\ddagger}$, it holds for $\mu=\mu(\mathtt U_1,\mathtt U_2,G_1^{\mathtt U_1},G_2^{\mathtt U_2})$ defined as in \eqref{eq-posterior} that
		\begin{equation*}\label{eq-anti-concentration}
			\begin{aligned}
				\sup_{\sigma_0\in \operatorname{B}(\mathtt U_1,\mathtt U_2)}\mu[\sigma:\operatorname{ov}(\sigma_0,\sigma)\ge \varepsilon n]=\sup_{\sigma_0\in \operatorname{B}(\mathtt U_1,\mathtt U_2)}\sum_{\substack{\sigma\in \operatorname{B}(\mathtt U_1,\mathtt U_2)\\\operatorname{ov}(\sigma_0,\sigma)\ge \varepsilon n}}\frac{\mathcal P_{\mathtt U_1,\mathtt U_2}^{\ddagger,*}[\pi^*\mid_{\mathtt U_1}=\sigma,G_1^{\mathtt U_1},G_2^{\mathtt U_2}]}{\mathcal P_{\mathtt U_1,\mathtt U_2}^{\ddagger}[G_1^{\mathtt U_1},G_2^{\mathtt  U_2}]}=o(1)\,.
			\end{aligned}
		\end{equation*}
	\end{proposition}
	\begin{proof}[Proof of Proposition~\ref{prop-reduced-negative-result-stronger} assuming Proposition~\ref{prop-posteior-anti-concentration}]
		Fix any estimator $\hat{\sigma}=\hat\sigma(G_1^{\mathtt U_1},G_2^{\mathtt U_2})\in \operatorname{B}(\mathtt U_1,\mathtt U_2)$. Viewing $(\pi^*\mid_{\mathtt U_1},G_1^{\mathtt U_1},G_2^{\mathtt U_2})\sim \mathcal P_{\mathtt U_1,\mathtt U_2}^{\ddagger,*}$ as first sample $(G_1^{\mathtt U_1},G_2^{\mathtt U_2})\sim \mathcal P_{\mathtt U_1,\mathtt U_2}^{\ddagger}$, then sample $\pi^*\mid_{\mathtt U_1}\sim \mu=\mu(\mathtt U_1,\mathtt U_2,G_1^{\mathtt U_1},G_2^{\mathtt U_2})$, we have
		\begin{align*}
			&\ \mathcal P_{\mathtt U_1,\mathtt U_2}^{\ddagger,*}[\operatorname{ov}(\hat\sigma,\pi^*\mid_{\mathtt U_1})\ge \varepsilon n]\\=&\ \mathbb{E}_{(G_1^{\mathtt U_1},G_2^{\mathtt U_2})\sim \mathcal P_{\mathtt U_1,\mathtt U_2}^{\ddagger}}\mathbb{E}_{\pi^*\mid_{\mathtt U_1}\sim \mu}[\mathbf{1}\{\operatorname{ov}(\hat\sigma,\pi^*\mid_{\mathtt U_1})\ge \varepsilon n\}]\\
			\le&\ \mathbb{E}_{(G_1^{\mathtt U_1},G_2^{\mathtt U_2})\sim \mathcal P_{\mathtt U_1,\mathtt U_2}^{\ddagger}}\sup_{\sigma_0\in \operatorname{B}(\mathtt U_1,\mathtt U_2)}\mu[\sigma:\operatorname{ov}(\sigma,\sigma_0)\ge \varepsilon n]\,,
		\end{align*}
		which is $o(1)$ according to Proposition~\ref{prop-posteior-anti-concentration}. Since $\operatorname{TV}(\mathcal P_{\mathtt U_1,\mathtt U_2}^{\dagger,*},\mathcal P_{\mathtt U_1,\mathtt U_2}^{\ddagger,*})=o(1)$, we also have $\mathcal P_{\mathtt U_1,\mathtt U_2}^{\dagger,*}[\operatorname{ov}(\hat\sigma,\pi^*)\ge \varepsilon n]=o(1)$,  completing the proof.
	\end{proof}
	
We now proceed to bound the posterior weight. Define $ \mathcal{Q}_{\mathtt{U}_1, \mathtt{U}_2} $ as the law of the graph pair $ (G_1^{\mathtt{U}_1}, G_2^{\mathtt{U}_2}) $, where $ G_1^{\mathtt{U}_1} $ and $ G_2^{\mathtt{U}_2} $ are independent, and each edge in $ \operatorname{E}(\mathtt{U}_1; \mathtt{W}_1) $ (resp. $ \operatorname{E}(\mathtt{U}_2; \mathtt{W}_2) $) appears in $ E(G_1^{\mathtt{U}_1}) $ (resp. $ E(G_2^{\mathtt{U}_2}) $) with probability $ ps $, independently. The crux of the argument consists of two main components: a quantitative ``transfer theorem" between the original law $ \mathcal{P}_{\mathtt{U}_1, \mathtt{U}_2}^{\ddagger} $ and the simpler law $ \mathcal{Q}_{\mathtt{U}_1, \mathtt{U}_2} $, and a detailed analysis of the posterior weight under $ \mathcal{Q}_{\mathtt{U}_1, \mathtt{U}_2} $.

More precisely, our arguements rely on the following two main technical inputs.

	\begin{proposition}\label{prop-dp/dq-bound}
		There exists $C_1=C_1(\alpha,\lambda,\varepsilon)$ such that for large enough $n$,
		\begin{equation}\label{eq-dp/dq-expO(n)}
            \mathbb{E}_{(G_1^{\mathtt U_1},G_2^{\mathtt U_2})\sim\mathcal P_{\mathtt U_1,\mathtt U_2}^{\ddagger}}\frac{\mathcal P_{\mathtt U_1,\mathtt U_2}^{\ddagger}[G_1^{\mathtt U_1},G_2^{\mathtt U_2}]}{\mathcal Q_{\mathtt U_1,\mathtt U_2}[G_1^{\mathtt U_1},G_2^{\mathtt U_2}]}\le \exp(C_1n)\,.
		\end{equation}
	\end{proposition}
	
	\begin{proposition}\label{prop-posterior-bound-under-Q}
		Let $(G_1^{\mathtt U_1},G_2^{\mathtt U_2})\sim \mathcal Q_{\mathtt U_1,\mathtt U_2}$, then for large enough $n$, it holds with probability at most $\exp\big(-n\sqrt{\log n}\big)$ that
		\begin{equation}\label{eq-posterior-under-Q}	\sup_{\sigma_0\in \operatorname{B}(\mathtt U_1,\mathtt U_2)}\sum_{\substack{\sigma\in \operatorname{B}(\mathtt U_1,\mathtt U_2)\\\operatorname{ov}(\sigma_0,\sigma)\ge \varepsilon n}}\frac{\mathcal P_{\mathtt U_1,\mathtt U_2}^{\ddagger,*}[\pi^*\mid_{\mathtt U_1}=\sigma,G_1^{\mathtt U_1},G_2^{\mathtt U_2}]}{{\mathcal Q}_{\mathtt U_1,\mathtt U_2}[G_1^{\mathtt U_1},G_2^{\mathtt  U_2}]}\ge \exp(-2n)\,.
		\end{equation}
	\end{proposition}
	
These two propositions can be derived using refinements of arguments in \cite{DD23a, DD23b}, so we defer their proofs to the appendix. Assuming these results, we can then provide the proof of Proposition~\ref{prop-posteior-anti-concentration}, thereby completing the proof of Theorem~\ref{thm-main-ER}-(ii).

	\begin{proof}[Proof of Proposition~\ref{prop-posteior-anti-concentration} assuming the technical inputs]
		Define the event 
		\[
		\mathcal G_{\operatorname{LR}}=\left\{(G_1^{\mathtt U_1},G_2^{\mathtt U_2}):\exp(-n)\le \frac{\mathcal P_{\mathtt U_1,\mathtt U_2}^{\ddagger}(G_1^{\mathtt U_1},G_2^{\mathtt U_2})}{\mathcal Q_{\mathtt U_1,\mathtt U_2}(G_1^{\mathtt U_1},G_2^{\mathtt U_2})} \le\exp((C_1+1)n)\right\}\,.
		\]
		Clearly $\mathbb{E}_{\mathcal P_{\mathtt U_1,\mathtt U_2}^{\ddagger}}\frac{\dif \mathcal Q_{\mathtt U_1,\mathtt U_2}}{\dif \mathcal P_{\mathtt U_1,\mathtt U_2}^{\ddagger}}\le 1$, which in combination with \eqref{eq-dp/dq-expO(n)} and Markov inequality yields that
		\[
		\mathcal P_{\mathtt U_1,\mathtt U_2}^{\ddagger}[\mathcal G_{\operatorname{LR}}]\ge 1-2\exp(-n)=1-o(1)\,.
		\]
		In addition,  Proposition~\ref{prop-posterior-bound-under-Q} yields that
		\[
		\mathcal P_{\mathtt U_1,\mathtt U_2}^{\ddagger}[(G_1^{\mathtt U_1},G_2^{\mathtt U_2})\in \mathcal G_{\operatorname{LR}}, \text{ and }\eqref{eq-posterior-under-Q}\text{ holds}]\le \exp\big((C_1+1)n-n\sqrt{\log n}\big)=o(1)\,.
		\]
		Therefore, with probability $1-o(1)$ over $(G_1^{\mathtt U_1},G_2^{\mathtt U_2})\sim \mathcal P_{\mathtt U_1,\mathtt U_2}^{\ddagger}$, $\mathcal G_{\operatorname{LR}}$ happens while \eqref{eq-posterior-under-Q} fails. This implies that with high probability under $\mathcal P_{\mathtt U_1,\mathtt U_2}^{\ddagger}$, 
		\begin{align*}
			&\ \sup_{\sigma_0\in \operatorname{B}(\mathtt U_1,\mathtt U_2)}\sum_{\substack{\sigma\in \operatorname{B}(\mathtt U_1,\mathtt U_2)\\\operatorname{ov}(\sigma_0,\sigma)\ge \varepsilon n}}\frac{\mathcal P_{\mathtt U_1,\mathtt U_2}^{\ddagger,*}[\pi^*\mid_{\mathtt U_1}=\sigma,G_1^{\mathtt U_1},G_2^{\mathtt U_2}]}{\mathcal P_{\mathtt U_1,\mathtt U_2}^{\ddagger}[G_1^{\mathtt U_1},G_2^{\mathtt  U_2}]}\\
			=&\ \frac{\mathcal Q_{\mathtt U_1,\mathtt U_2}[G_1^{\mathtt U_1},G_2^{\mathtt U_2}]}{\mathcal P_{\mathtt U_1,\mathtt U_2}^{\ddagger}[G_1^{\mathtt U_1},G_2^{\mathtt U_2}]}\cdot  \sup_{\sigma_0\in \operatorname{B}(\mathtt U_1,\mathtt U_2)}\sum_{\substack{\sigma\in \operatorname{B}(\mathtt U_1,\mathtt U_2)\\\operatorname{ov}(\sigma_0,\sigma)\ge \varepsilon n}}\frac{\mathcal P_{\mathtt U_1,\mathtt U_2}^{\ddagger,*}[\pi^*\mid_{\mathtt U_1}=\sigma,G_1^{\mathtt U_1},G_2^{\mathtt U_2}]}{{\mathcal Q}_{\mathtt U_1,\mathtt U_2}[G_1^{\mathtt U_1},G_2^{\mathtt  U_2}]}\\
			\le&\ \exp(n)\cdot  \exp(-2n)=\exp(-n)=o(1)\,.\qedhere
		\end{align*}
	\end{proof}

	\appendix 
	\section{Deferred proofs in Section~\ref{sec-positive}}
	
	\subsection{Proof of Lemma~\ref{lem-D}}
 \begin{proof}[Proof of Lemma~\ref{lem-D}]
 Since $\mathcal H_{\pi^*}\sim \mathbf{G}(n,\frac\lambda n)$, we have for any $U\subset [n]$, the number of edges in $\mathcal H_{\pi^*}$ with at least one vertex in $U$ is stochastically dominated by $\mathbf{B}(n|U|,\frac\lambda n)$. Based on this relation, we take $D=\max\{e^2\lambda, 1+\delta^{-1}+\log(e\lambda\delta^{-2})\}$, and we prove \eqref{eq-event-D} via taking union bounds. 

On the one hand, for any $U\subset [n]$ satisfies $2e\lambda|U|\le \delta e^{-2/\delta}$, it follows by the Chernoff bound the probability that there are more than $\delta n/2$ edges connecting $U$ is bounded by
 \[
\exp\left(-\frac{\delta n}{2}\log\Big(\frac{\delta n}{2\lambda|U|}\Big)+\delta n\right)=\exp\left(-\frac{\delta n}{2}\log\Big(\frac{\delta n}{2e\lambda |U|}\Big)\right)\le \exp(-n)\,.
 \]
Hence, a union bound yields that with high probability this scenario does not happen for any $U$ with $2e\lambda |U|\le \delta e^{-2/\delta}n$.
 
 On the other hand, for any $U\subset [n]$ with $|U|=k$ and $2e\lambda k\ge \delta e^{-2/\delta}n$, by the Chernoff bound again the probability $\mathcal H_{\pi^*}$ has more than $D|U|$ edges connecting $U$ is bounded by
 \[
\exp\left(-Dk\log\Big(\frac{Dk}{e\lambda k}\Big)\right)=\exp\left(-D\log\Big(\frac{D}{e\lambda}\Big)k\right)\,.
 \]
 Since there are $\binom{n}{k}\le \exp(k\log(n/k)+k)\le \exp((1+\delta^{-1}+\log(e\lambda\delta^{-1})k)$ many $U$ with size $k$, a union bound yields with high probability there is no $U$ with $2e\lambda |U|\ge \delta e^{-2/\delta}n$ such that there are at least $D|U|$ edges connecting $U$. 
 Altogether we conclude that \eqref{eq-event-D} holds with high probability.
 \end{proof}
	
	\subsection{Proof of Lemma~\ref{lem-tail-prob-conditioning}}
	In this subsection we prove Lemma~\ref{lem-tail-prob-conditioning}. 
	We start with the definition of the types of random variables that serve as basic components of $E_{\operatorname{s}},E_k,E_{\operatorname{c}}^{\operatorname{c}}$ and $E_{\operatorname{c}}^{\operatorname{f}}$. Fix an integer $k\ge 1$, let $I_0,\dots,I_k, J_0^1,\\\dots,J_{k-1}^1,J_1^2,\dots,J_{k}^2$ be indicators, and write
	\[
	S_k=I_0J_0^1I_1J_1^2+I_1J_1^1I_2J_2^2+\cdots+I_{k-1}J_{k-1}^1I_kJ_k^2=\sum_{t=1}^{k}I_{t-1}J_{t-1}^1I_tJ_t^2\,.
	\]
	Define random variables $X_k,Y_k,Z_k$ as follows: 
	\begin{itemize}
 		\item ($k$-cycles) $X_k$ has the same distribution of $S_k$, where $I_0=I_k$, and $I_0,\dots,I_{k-1}\sim \mathbf B(1,p)$, $J_0^1,\dots,J_{k-1}^1,J_1^2,\dots,J_k^2\sim \mathbf B(1,s)$ are independent Bernoulli indicators.
		\item (Free $k$-chain) $Y_k$ has the same distribution of $S_k$, where $I_0,\dots,I_k\sim \mathbf B(1,p)$ and $J_0^1,\dots,J_{k-1}^1$,\\$J_1^2,\dots,J_{k}^2\sim \mathbf B(1,s)$ are independent Bernoulli indicators.
		\item (Confined $k$-chain) $Z_k$ has the same distribution as $S_k$, where $I_0=I_k=1$, $I_1,\dots,I_{k-1}\sim \mathbf B(1,p),J_0^1,\dots,J_{k-1}^1,J_1^2,\dots,J_k^2\sim \mathbf B(1,s)$ are independent Bernoulli indicators. 
	\end{itemize}
It is straightforward to check that each of $E_{\operatorname{s}},E_k,E_{\operatorname{c}}^{\operatorname{c}}$ and $E_{\operatorname{c}}^{\operatorname{f}}$ is an independent sum of random variables that are either distributed like or stochastically dominated by one of the above three types of variables. The following lemma provides an upper bound for their exponential moments.

	\begin{lemma}\label{lem-exponential-moments}
		For any $\theta>0$ such that $e^\theta\ll n$, we write $\nu=e^\eta-1$, and let $\mu_1>\mu_2>0$ be the two roots of the polynomial $x^2-(1+ps^2\nu)x+p(1-p)s^2$. Then for any $k\ge 1$, it holds that
  		\begin{align}
			\label{eq-exp-X}
			\mathbb{E}\exp(\theta X_k)\le&\ \mu_1^k+\mu_2^k\,,\\
			\label{eq-exp-Y}
			\mathbb{E}\exp(\theta Y_k)\le&\ \mu_1^k+e^\theta n^{-1-2\alpha+o(1)}\mu_2^k\,,\\
			\label{eq-exp-Z}
			\mathbb{E}\exp(\theta Z_k)\le&\ (1+3ps^2\nu)\mu_1^k\,.
		\end{align}
	\end{lemma}
	\begin{proof}
		As argued in the proof of \cite[Proposition 2.3]{DD23b}, the $\theta$-exponential moment of each of these variables can be expressed as a linear combinations of $\mu_1^k$ and $\mu_2^k$ with coefficients does not depend on $k$, so it remains to estimate the coefficients.  \eqref{eq-exp-X} and \eqref{eq-exp-Y} have already been proved in \cite[Proposition 2.3]{DD23b} (and \eqref{eq-exp-X} is actually an identity). For \eqref{eq-exp-Z}, we assume that for two constants $c_1,c_2$, it holds
		\[
		\mathbb{E}\exp(\theta Z_k)=c_1\mu_1^k+c_2\mu_2^k,\ k=1,2,\dots.
		\]
		By calculating the left hand side for $k=1,2$ directly, we conclude that
		\[
		c_1\mu_1+c_2\mu_2=A_1\equiv 1+ps^2\nu\,,\quad c_1\mu_1^2+c_2\mu_2^2=1-2ps^2\nu+p^2s^4\nu^2\,.
		\]
		Solving the linear system and using the asymptotics that $\mu_1=1+p^2s^2\nu+O(p^3s^4\nu)$ and $\mu_2=ps^2\nu+O(p^2s^2\nu)$ (see \cite[Section A.1]{DD23b} for details), we get (noticing that $e^\theta\ll n$ implies $ps^2\nu=o(1)$)
		\begin{align*}
			c_1&=\frac{A_2-\mu_2A_1}{\mu_1(\mu_1-\mu_2)}=\frac{1+2ps^2\nu+p^2s^4\nu^2-(1+ps^2\nu)(ps^2\nu+O(p^2s^2\nu))}{(1+p^2s^2\nu+O(p^3s^4\nu))(1-ps^2\nu+O(p^2s^2\nu))}\\
			&=\frac{1+ps^2\nu+p^2s^4\nu^2+O(p^2s^2\nu)}{1-ps^2\nu+O(p^2s^2\nu)}\le \frac{1+(1+o(1))ps^2\nu}{1-(1+o(1))ps^2\nu}\le 1+3ps^2\nu\,,
		\end{align*}
		and 
		\begin{align*}
			c_2&=\frac{\mu_1A_1-A_2}{\mu_2(\mu_1-\mu_2)}=\frac{(1+ps^2\nu)(1+p^2s^2\nu+O(p^3s^2\nu))-(1+2ps^2\nu+p^2s^4\nu^2)}{(ps^2\nu+O(p^2s^2\nu))(1-ps^2\nu+O(p^2s^2\nu))}\\&=\frac{(1+ps^2\mu+O(p^2s^2\nu))-(1+2ps^2\nu+p^2s^4\nu^2)}{(1+o(1))ps^2\nu}=\frac{-(1+o(1))ps^2\nu}{(1+o(1))ps^2\nu}\le 0\,,
		\end{align*}
		as desired. 
	\end{proof}

 Now we prove Lemma~\ref{lem-tail-prob-conditioning}. 
 \begin{proof}[Proof of Lemma~\ref{lem-tail-prob-conditioning}]
    \eqref{eq-tail-Es}, \eqref{eq-tail-Ek}, \eqref{eq-tail-E>L} can be proved similarly as in \cite[Lemma 2.4]{DD23b} given \eqref{eq-exp-X}. It remains to prove \eqref{eq-tail-Ec}. Since $E_{\operatorname{c}}=E_{\operatorname{c}}^{\operatorname{f}}+E_{\operatorname{c}}^{\operatorname{c}}$, it suffices to prove that for any $0\le x\le O(n)$, 
    \begin{equation}\label{eq-tail-Ecf}
    \mathcal P_{\pi^*,\mathtt{E}_{k-1}}[E_{\operatorname{c}}^{\operatorname{f}}\ge x]\le \exp(o(n\log n)-\alpha x\log n)\,,
    \end{equation}
    and
    \begin{equation}\label{eq-tail-Ecc}
    \mathcal P_{\pi^*,\mathtt{E}_{k-1}}[E_{\operatorname{c}}^{\operatorname{c}}\ge x]\le \exp(o(n\log n)-\alpha x\log n)\,.
    \end{equation}
    Denote $T_k^{\operatorname{f}}$ (resp. $T_k^{\operatorname{c}}$) as the number of free $k$-chain (resp. confined $k$-chains) in $\mathcal O_{\check\pi}(\mathtt U_{k-1},U)$. We note that the contributions to $E_{\operatorname{c}}^{\operatorname{f}}$ (resp. $E_{\operatorname{c}}^{\operatorname{c}}$) from distinct free $k$-chains (resp. confined $k$-chains) are independent and distribute like $Y_k$ (resp. are stochastically dominated by $Z_k$). It follows from Chebyshev's inequality that for any $\theta>0$, 
    \[
    \mathcal P_{\pi^*,\mathtt{E}_{k-1}}[E_{\operatorname{c}}^{\operatorname{f}}\ge x]\le e^{-\theta x}\mathbb{E}e^{\theta E_{\operatorname{c}}^{\operatorname{f}}}=e^{-\theta x}\prod_{k}(\mathbb{E}e^{\theta Y_k})^{T_k^{\operatorname{f}}}\,.
    \]
    We pick $\theta=(\alpha+o(1))\log n$ such that $\mu_1=1+(n\log n)^{-1}$ (this ensures $e^\theta\ll n$). Using \eqref{eq-exp-Y}, we see the left hand side of \eqref{eq-tail-Ecf} is upper-bounded by
    \begin{align*}
    e^{-\theta x}\prod_{k}(\mu_1^{k}+n^{-1-\alpha+o(1)}\mu_2^k)^{T_k^{\operatorname{f}}}\le e^{-(\alpha+o(1))x\log n}\prod_k\mu_1^{kT_k^{\operatorname{f}}}(1+n^{-1-\alpha+o(1)}\mu_2^k)^{T_k^{\operatorname{f}}}\,.
    \end{align*}
    Since $\mu_1=1+(n\log n)^{-1}$, $1+n^{-1-\alpha+o(1)}\mu_2^k\le 1+n^{-1-\alpha+o(1)}\mu_2=1+n^{-2+o(1)}$, and $\sum_{k}kT_k^{\operatorname{f}}\le n^2$, the above is $\exp(o(n\log n)-\alpha x\log n)$ for $x=O(n)$. This proves \eqref{eq-tail-Ecf}. 
    
    Similarly, by Chebyshev's inequality we have for any $\theta>0$, 
    \[
    \mathcal P_{\pi^*,\mathtt{E}_{k-1}}[E_{\operatorname{c}}^{\operatorname{c}}\ge x]\le e^{-\theta x}\mathbb{E}e^{\theta E_{\operatorname{c}}^{\operatorname{f}}}\le e^{-\theta x}\prod_{k}(\mathbb{E}e^{\theta Z_k})^{T_k^{\operatorname{c}}}\,.
    \]
   Again, pick $\theta=(\alpha+o(1))\log n$ such that $\mu_1=1+(n\log n)^{-1}$. Using \eqref{eq-exp-Z}, the left hand side of \eqref{eq-tail-Ecc} is bounded by
    \[
e^{-\theta x}\prod_{k}\big((1+3ps^2\nu)\mu_1^{k}\big)^{T_k^{\operatorname{c}}}\le e^{-(\alpha+o(1))x\log n}\prod_k (1+n^{-1+\alpha+o(1)})^{T_k^{\operatorname{c}}}(1+(n\log n)^{-1})^{kT_k^{\operatorname{c}}}\,.
    \]
    Since each confined cycle contains at least one edge in $\mathtt E_{k-1}$ while there are only $O(n)$ edges in $\mathtt E_{k-1}$ (by the goodness of the realization $\mathcal F_{k-1}$), we conclude $\sum_k T_k=O(n)$. Also we have $\sum_k kT_k^{\operatorname{c}}=O(n^2)$, so the right hand side of the above expression is $\exp(o(n\log n)-\alpha x\log n)$ for $x=O(n)$.
    This verifies \eqref{eq-tail-Ecc}, thereby concluding the proof. 
 \end{proof}
	
	\subsection{Proof of Lemma~\ref{lem-enumeration}}
 \begin{proof}[Proof of Lemma~\ref{lem-enumeration}]
     Fix the realization of $\pi^*$ and $\tilde{\pi}$. Consider the set $\overline{\operatorname{S}}_{\mathtt U_{k-1},\tilde{\pi}}(n_1,\dots,n_L)$ of permutations $\pi\in \operatorname{S}_n$ that equal to $\tilde\pi$ on $\mathtt{U}_{k-1}$, and $\pi^*\circ \pi^{-1}$ has at least $n_t$ vertices in $[n]\setminus \mathtt{U}_{k-1}$ lying in cycles with length $t$. Note that each partial matching $\check\pi$ in $\operatorname{S}_{\mathtt U_{k-1},U,\tilde{\pi}}(n_1,\dots,n_L)$ extends to $(n-|U|)!$ many permutations in $\overline{\operatorname{S}}_{\mathtt U_{k-1},\tilde{\pi}}(n_1,\dots,n_L)$ (and these permutations are distinct for different $\check\pi$), so we have
     \[
     (n-|U|)!\cdot |{\operatorname{S}}_{\mathtt U_{k-1},U,\tilde{\pi}}(n_1,\dots,n_L)|\le |\overline{\operatorname{S}}_{\mathtt U_{k-1},\tilde{\pi}}(n_1,\dots,n_L)|\,.
     \]
     
     To bound the size of $\overline{\operatorname{S}}_{\mathtt U_{k-1},\tilde{\pi}}(n_1,\dots,n_L)$, we consider the following mapping from the set $\overline{\operatorname{S}}_{\mathtt U_{k-1},\tilde{\pi}}(n_1,\dots,n_L)$ to the set of permutations on $[n]\setminus \mathtt{U}_{k-1}$: for ${\pi}\in \overline{\operatorname{S}}_{\mathtt U_{k-1},\tilde{\pi}}(n_1,\dots,n_L)$, define $\sigma:[n]\setminus \mathtt U_{k-1}\to [n]\setminus \mathtt U_{k-1}$ by letting $\sigma(i)=(\pi^*\circ \pi^{-1})^{(k)}(i)$ for each $i\in [n]\setminus \mathtt U_{k-1}$, where $k=k(i)$ is the minimal positive integer such that $(\pi^*\circ \pi^{-1})^{(k)}(i)\in [n]\setminus \mathtt U_{k-1}$. It is easy to check $\pi\mapsto \sigma$ is injective (as $\pi\mid_{\mathtt U_{k-1}}=\tilde\pi$ is given). Furthermore, because they mapping only shortens cycles, each $\sigma$ in the image of such a mapping satisfies the following: for $1\le t\le L$, there are at least $n_1+\cdots+n_t$ vertices lying in cycles with length at most $t$. It follows from standard fact about random permutations (see, e.g. \cite[Appendix D]{WXY22}) that the number of permutations on $n-|\mathtt U_{k-1}|$ elements with this property is at most
     \[
\exp((n-|\mathtt U_{k-1}|-n_1-\frac{n_2}{2}-\cdots-\frac{n_L}{L})\log n+O(n))\,.
     \]
     In conclusion, we obtain that (since $(n-|U|)!=\exp((n-|U|)\log n+O(n))$)
     \[
|{\operatorname{S}}_{\mathtt U_{k-1},U,\tilde{\pi}}(n_1,\dots,n_L)|\le \exp((T-n_1-\frac{n_2}{2}-\cdots-\frac{n_L}{L})\log n+O(n))\,.\qedhere
     \]
 \end{proof}
	
	\subsection{Proof of Lemma~\ref{lem-algebra}}
	\begin{proof}[Proof of Lemma~\ref{lem-algebra}]
		For any $n_1,\dots,n_L$ and $\mathbf{x}=(x_t)_{-1\le t\le L+1}\in \Delta(n_1,\dots,n_L)$, we have that 
		\begin{align*}
			&\ x_{-1}+\alpha x_0+\sum_{t=1}^L \alpha_tx_t+\alpha_{L+1}x_{L+1}\\
			=&\ (1-\alpha_{L+1})x_{-1}+(\alpha-\alpha_{L+1})x_0+\alpha_{L+1}\sum_{t=-1}^Lx_t-\sum_{t=1}^L(\alpha_{t+1}-\alpha_t)\sum_{s=1}^tx_s\\
			\ge&\ 0+0+ \alpha_{L+1}\cdot l_k T-\sum_{t=1}^L(\alpha_{t+1}-\alpha_t)u_k\sum_{s=1}^t n_s\\
			=&\ \alpha_{L+1}\cdot (l_kT-u_k\sum_{t=1}^L n_t)+u_k\sum_{t=1}^L\alpha_tn_t\,.
		\end{align*}
		Therefore, the left-hand side of \eqref{eq-algebra} is bounded from above by
		\begin{align*}
			&\ T-\sum_{t=1}^L\frac{n_t}{t}-\alpha_{L+1}(l_kT-u_k\sum_{t=1}^L n_t)-u_k\sum_{t=1}^L\alpha_tn_t\\
			=&\  (1-\alpha_{L+1}l_k)T+\sum_{t=1}^L(\alpha_{L+1}u_k-\alpha_tu_k-t^{-1})n_t\,.
		\end{align*}
		Recalling $\alpha_t=\frac{t-1}{t},\forall 1\le t\le L$, this can also be written as
		\begin{align}\nonumber &\ (1-\alpha_{L+1}l_k)T+\sum_{t=1}^L(\alpha_{L+1}u_k-1)n_t-\sum_{t=1}^L(u_k-1)\alpha_tn_t\\
			\le&\ (1-\alpha_{L+1}l_k)T+\sum_{t=1}^L(\alpha_{L+1}u_k-1)n_t-\frac{u_k-1}{2}\sum_{t=2}^Ln_t\,.\label{eq-algebra-upper-bound}
		\end{align}
		Note that $u_k\ge \alpha^{-1}+\varepsilon$ and thus $u_k\alpha_{L+1}-1\ge \alpha\varepsilon/2$ by our choice of $L$ as in \eqref{eq-L}. We divide into two cases:\\
		\noindent Case 1: If $n_1+\cdots+n_L\le (1-\zeta/4)T$, then \eqref{eq-algebra-upper-bound} is bounded by
		\begin{align*}
			\big[(1-\alpha_{L+1}l_k)+(\alpha_{L+1}u_k-1)(1-\zeta/4)\big]\cdot T=&\  \big[\alpha_{L+1}(u_k-l_k)-(\alpha_{L+1}u_k-1)\cdot \zeta/4\big]\cdot T\\
			(\text{since }u_k-l_k\le \eta,\alpha_{L+1}\le 1)\quad\le&\ (\eta-\alpha\varepsilon\zeta/8)\cdot T\\
			(\text{since }T\ge N^{-1}\eta n)\quad\quad\qquad\quad\ \le &\ -N^{-1}\eta(\alpha\varepsilon\zeta/8-\eta)n
		\end{align*}
		\noindent Case 2: $n_1+\cdots+n_L\ge (1-\zeta/4)T$, then since $n_1\le (1-\zeta/2)T$, we have $n_2+\cdots+n_L\ge\zeta /4\cdot T$. As a result, \eqref{eq-algebra-upper-bound} is upper-bounded by
		\[
		(1-\alpha_{L+1}l_k)T+(\alpha_{L+1}u_k-1)T-\frac{u_k-1}{2}\cdot \frac{\zeta}{4}\cdot T\le (\eta-\varepsilon\zeta/8)T\le -N^{-1}\eta\cdot (\varepsilon\zeta/8-\eta)n\,.
		\] 
		In conclusion, we may take $\chi=N^{-1}\eta(\alpha\varepsilon\zeta/8-\eta)>0$ and \eqref{eq-algebra} holds true, as desired
	\end{proof}
	
	\section{Deferred proofs in Section~\ref{sec-negative}}
	In the remaining of the paper we provide the proof of the two main technical inputs for the negative result. Fix $\alpha\in (0,1], \varepsilon>0$, as well as two subsets $\mathtt U_1,\mathtt U_2$ of $[n]$. Write $\mathtt W_i=[n]\setminus \mathtt{U}_i,i=1,2$. We pick a large integer $C$ such that
	\begin{equation}\label{eq-C}
		\alpha(\alpha^{-1}-\varepsilon+C^{-1})<1\,.
	\end{equation}
	Recalling the definition of $\mathcal P_{\mathtt U_1,\mathtt U_2}^{\dagger,*}$ and $\mathcal P_{\mathtt U_1,\mathtt U_2}^{\ddagger,*}$, we start by defining the event $\mathcal A_1$.
	\begin{definition}\label{def-admissible}
		For a graph $\mathcal H=([n],\mathcal E)$, we say it is admissible if the following hold:\\
		\noindent (i) $\mathcal H$ has maximal degree no more than $\log n$;\\
		\noindent (ii) For $D_n\equiv (\log n)^{1/10C}$, there are at least $(1-\exp(-D_n))n$ many $i\in [n]$ such that no vertex in the $(2C+2)$-neighborhood of $i$ in $\mathcal H$ has degree greater than $D_n$.\\
		\noindent (iii) Any connected subgraph of $\mathcal H$ with size less than $\log\log n$ contains at most one cycle.\\
		\noindent (iv) For $k\ge 3$, the number of cycles in $\mathcal H$ with length $k$ is no more than $(\log n)^k$. 
	\end{definition} 
	Let $(\pi^*,G_1^{\mathtt U_1},G_2^{\mathtt U_2})\sim \mathcal P_{\mathtt U_1,\mathtt U_2}^{\dagger,*}$, and recall that $\mathcal H_{\pi^*}(\mathtt U_1;\mathtt W_1)$ be the $\pi^*$-intersection graph of $G_1^{\mathtt U_1}$ and $G_2^{\mathtt U_2}$. We denote $\mathcal A_1$ for the event that $\mathcal H_{\pi^*}(\mathtt U_1;\mathtt W_1)$ is admissible. As promised before, we show $\mathcal A_1$ is a typical event under $\mathcal P_{\mathtt U_1,\mathtt U_2}^{\dagger,*}$. 
	
	\begin{lemma}
		We have $\mathcal P_{\mathtt U_1,\mathtt U_2}^{\dagger,*}[\mathcal A_1]=1-o(1)$. 
	\end{lemma}
	\begin{proof}
		First, recall that $\mathcal P_{\mathtt U_1,\mathtt U_2}^{\dagger,*}[\cdot]=\mathcal P^*_{\mathtt U_1,\mathtt U_2}[\cdot\mid \eqref{eq-density-upper-bound}\text{ holds}]$. Since both $\mathcal A_1$ and $\{\eqref{eq-density-upper-bound}\text{ holds}\}$ are decreasing events, FKG inequality yields that $\mathcal P_{\mathtt U_1,\mathtt U_2}^{\dagger,*}[\mathcal A_1]\ge \mathcal P^*_{\mathtt U_1,\mathtt U_2}[\mathcal A_1]$. Therefore, it suffices to show each item fails with $o(1)$ probability under $\mathcal P_{\mathtt U_1,\mathtt U_2}^*$. For items (i) and (iii), this claim follows from a straightforward union bound, and for item (iv), it follows from Markov inequality (see \cite[Lemma 3.3]{DD23b} for details). For item (ii), we show inductively that for each $i\in [n]$ and $0\le t\le 2C+2$,
		\begin{equation}\label{eq-Dn}
			\mathcal P_{\mathtt U_1,\mathtt U_2}^*[\exists j\in [n]:\operatorname{d}_{\mathcal H_{\pi^*}(\mathtt U_1;\mathtt W_1)}(i,j)\le t,\operatorname{deg}(j)\ge D_n]\le (t+1)\exp(-2D_n)\,.
		\end{equation}
		For $t=0$, the event is equivalent to $\operatorname{deg}(i)>D_n$. Under $\mathcal P^*_{\mathtt U_1,\mathtt U_2}$, the distribution of $\operatorname{deg}(i)$ is stochastically dominated by $\mathbf{B}(n,\lambda/n)$. Then the Chernoff bound yields that $\mathcal P_{\mathtt U_1,\mathtt U_2}^*[\operatorname{deg}(i)>D_n]\le \exp(-D_n\log D_n+O(D_n))$, and thus \eqref{eq-Dn} holds for $t=0$.  For any $1\le t\le 2C+2$, assume \eqref{eq-Dn} holds for $t-1$. Note that conditioned on any vertex in the $(t-1)$-neighborhood of $i$ has degree no more than $D_n$, the degree of any vertex with distance $(t-1)$ to $i$ is still dominated by $\mathbf B(n,\frac\lambda n)$. Therefore, from the induction hypothesis and a union bound,
		\begin{align*}
			&\ \mathcal P_{\mathtt U_1,\mathtt U_2}^*[\exists j\in [n]:\operatorname{d}_{\mathcal H_{\pi^*}(\mathtt U_1;\mathtt W_1)}(i,j)\le t,\operatorname{deg}(j)\ge D_n]\\
			\le&\ t\exp(-2D_n)+D_n^t\exp(-D_n\log D_n+O(D_n))\le (t+1)\exp(-2D_n)\,.
		\end{align*}
		This completes the induction proof of \eqref{eq-Dn}, and $\mathcal P^*_{\mathtt U_1,\mathtt U_2}[(iv)\text{ fails}]=o(1)$ follows from the case $t=2C+2$ and the Markov inequality.
	\end{proof}
	
	Recall that we defined $\mathcal P_{\mathtt U_1,\mathtt U_2}^{\ddagger,*}[\cdot]=\mathcal P_{\mathtt U_1,\mathtt U_2}^{\dagger,*}[\cdot \mid \mathcal A_1]=\mathcal P^*_{\mathtt U_1,\mathtt U_2}[\cdot \mid \mathcal A]$, where $\mathcal A\equiv \mathcal A_1\cap \{\eqref{eq-density-upper-bound}\text{ holds}\}$. Then it follows from the previous lemma that $\operatorname{TV}(\mathcal P_{\mathtt U_1,\mathtt U_2}^{\dagger,*},\mathcal P_{\mathtt U_1,\mathtt U_2}^{\ddagger,*})=o(1)$.
	
	\subsection{Proof of Proposition~\ref{prop-dp/dq-bound}}\label{subsec-B.1}
	We now proceed to bound the likelihood ratio of $\mathcal P_{\mathtt U_1,\mathtt U_2}^{\ddagger,*}$ and $\mathcal Q_{\mathtt U_1,\mathtt U_2}$. By definition, it is straightforward to calculate that for any $(G_1^{\mathtt U_1},G_2^{\mathtt U_2})$,
	\begin{align*}
		\frac{\mathcal P_{\mathtt U_1,\mathtt U_2}^{\ddagger,*}[G_1^{\mathtt U_1},G_2^{\mathtt U_2}]}{\mathcal Q_{\mathtt U_1,\mathtt U_2}[G_1^{\mathtt U_1},G_2^{\mathtt U_2}]}=\frac{1}{\mathcal P^*_{\mathtt U_1,\mathtt U_2}[\mathcal A]}\cdot \frac{1}{|\mathtt U_1|!}\sum_{\pi:\pi\mid_{\mathtt U_1}=\pi^*\mid_{\mathtt U_1}}\frac{\mathcal P^*_{\mathtt U_1,\mathtt U_2}[\pi,G_1^{\mathtt U_1}, G_2^{\mathtt U_2}]\mathbf{1}\{(\pi,G_1^{\mathtt U_1},G_2^{\mathtt U_2})\in \mathcal A\}}{\mathcal Q_{\mathtt U_1,\mathtt U_2}[G_1^{\mathtt U_1},G_2^{\mathtt U_2}]}\,.
	\end{align*}
	We have a trivial lower bound
	\[
	\mathcal P^*_{\mathtt U_1,\mathtt U_2}[\mathcal A]\ge \mathcal P^*_{\mathtt U_1,\mathtt U_2}[E(G_1)=E(G_2)=\emptyset]\ge\left(1-\frac\lambda n\right)^{2\binom{n}{2}}=\exp(-O_\lambda(n))\,,
	\]
	where $O_\lambda(\cdot)$ means the constant in $O(\cdot)$ term only depends on $\lambda$ (and similar for below). Thus,
	\[
	\frac{\mathcal P_{\mathtt U_1,\mathtt U_2}^{\ddagger,*}[G_1^{\mathtt U_1},G_2^{\mathtt U_2}]}{\mathcal Q_{\mathtt U_1,\mathtt U_2}[G_1^{\mathtt U_1},G_2^{\mathtt U_2}]}\le \frac{\exp(O_\lambda(n))}{|\mathtt U_1|!}\sum_{\pi:\pi\mid_{\mathtt U_1}=\pi^*\mid_{\mathtt U_1}}\frac{\mathcal P^*_{\mathtt U_1,\mathtt U_2}[\pi,G_1^{\mathtt U_1}, G_2^{\mathtt U_2}]\mathbf{1}\{(\pi,G_1^{\mathtt U_1},G_2^{\mathtt U_2})\in \mathcal A\}}{\mathcal Q_{\mathtt U_1,\mathtt U_2}[G_1^{\mathtt U_1},G_2^{\mathtt U_2}]}\,.
	\]
	Consequently, to prove Proposition~\ref{prop-dp/dq-bound}, it suffices to show
	\[
	\frac{1}{|\mathtt U_1|!}\mathbb{E}_{(G_1^{\mathtt U_1},G_2^{\mathtt U_2})\sim \mathcal P_{\mathtt U_1,\mathtt U_2}^*}\sum_{\pi:\pi\mid_{\mathtt U_1}=\pi^*\mid_{\mathtt U_1}}\frac{\mathcal P^*_{\mathtt U_1,\mathtt U_2}[\pi,G_1^{\mathtt U_1}, G_2^{\mathtt U_2}]\mathbf{1}\{(\pi,G_1^{\mathtt U_1},G_2^{\mathtt U_2})\in \mathcal A\}}{\mathcal Q_{\mathtt U_1,\mathtt U_2}[G_1^{\mathtt U_1},G_2^{\mathtt U_2}]}=\exp(O_{\lambda,\varepsilon}(n))\,.
	\]	
    
	For $(\pi^*,G_1^{\mathtt U_1},G_2^{\mathtt U_2})\sim \mathcal P^*_{\mathtt U_1,\mathtt U_2}$ and $\pi\in\operatorname{S}_n$ with $\pi\mid_{\mathtt U_1}=\pi^*\mid_{\mathtt U_1}$, let $\mathcal J_\pi=\mathcal J_\pi(\pi^*,G_1^{\mathtt U_1},G_2^{\mathtt U_2})$ be the set of edge orbits of $\sigma\equiv\pi^*\circ\pi^{-1}$ that entirely contains within $\mathcal H_{\pi^*}$. With a slight abuse of notation, we denote $|\mathcal J_\pi|$ as the total number of edges of $\mathcal H_{\pi^*}$ lying in orbits in $\mathcal J_\pi$. Following identical arguments as in \cite[Lemma 3.2]{DD23a}, we arrive at the following lemma.
	
	\begin{lemma}
		The following inequality holds
		\begin{align}
			\nonumber&\ \mathbb{E}_{(G_1^{\mathtt U_1},G_2^{\mathtt U_2})\sim \mathcal P_{\mathtt U_1,\mathtt U_2}^{\ddagger,*}}\sum_{\pi:\pi\mid_{\mathtt U_1}=\pi^*\mid_{\mathtt U_1}}\frac{\mathcal P^*_{\mathtt U_1,\mathtt U_2}[\pi,G_1^{\mathtt U_1}, G_2^{\mathtt U_2}]\mathbf{1}\{(\pi,G_1^{\mathtt U_1},G_2^{\mathtt U_2})\in \mathcal A\}}{\mathcal Q_{\mathtt U_1,\mathtt U_2}[G_1^{\mathtt U_1},G_2^{\mathtt U_2}]}\\
			\le&\ \mathbb{E}_{(\pi^*,G_1^{\mathtt U_1},G_2^{\mathtt U_2})\sim \mathcal P_{\mathtt U_1,\mathtt U_2}^{\ddagger,*}}\sum_{\pi:\pi\mid_{\mathtt U_1}=\pi^*\mid_{\mathtt U_1}}\frac{p^{-|\mathcal J_\pi|}}{\mathcal P_{\mathtt U_1,\mathtt U_2}^*[\mathcal A\mid \pi^*,\mathcal J_\pi]}\,.\label{eq-exp-moment}
		\end{align}
	\end{lemma}
	
We can handle the denominators of the summands in \eqref{eq-exp-moment} as follows. Note that conditioning on the realization of $\mathcal J_\pi$ is equivalent to conditioning on $H(\mathcal J_\pi)\subset \mathcal H_{\mathtt U_1,\mathtt U_2}$ together with another decreasing event (here $H(\mathcal J_\pi)$ denotes the graph on $[n]$ with edges in the orbits in $\mathcal J_\pi$; see \cite[Lemma 3.4]{DD23a} for details), and thus 
	\[
	\mathcal P_{\mathtt U_1,\mathtt U_2}^*[\mathcal A\mid \pi^*,\mathcal J_\pi]\ge \mathcal P_{\mathtt U_1,\mathtt U_2}^*[\mathcal H_{\pi^*}(\mathtt U_1;\mathtt W_1)=H(\mathcal J_\sigma)\mid \pi^*,\mathcal J_\pi]\ge (1-ps^2)^{\binom{n}{2}}\ge \exp(-O_\lambda(n))\,,
	\]
	where the last inequality follows from FKG inequality.
	Combining this lower bound with the above lemma, Proposition~\ref{prop-dp/dq-bound} follows upon showing that whenever $\mathcal H_{\pi^*}(\mathtt U_1;\mathtt W_1)$ is admissible and satisfies \eqref{eq-density-upper-bound} (which holds almost surely under $\mathcal P_{\mathtt U_1,\mathtt U_2}^{\ddagger,*}$), one has
	\begin{equation}\label{eq-goal-bound-1}
		\frac{1}{(n-|\mathtt U_1|!)}\sum_{\pi:\pi\mid_{\mathtt U_1}=\pi^*\mid_{\mathtt U_1}}p^{-|\mathcal J_\pi|}=\exp(O_\varepsilon(n))\,.
	\end{equation}
	
	To this end, we follow the approach in \cite{DD23a}. For each $r\in \mathbb{N}$, we define $\mathfrak C_r$ (resp. $\mathcal T_r$) as the set of isomorphic classes of connected non-tree graphs (resp, trees) with $r$ vertices. We represent each class with a representative graph.  For any $C\in \mathfrak C_r$, let $\operatorname{Aut}(C)$ be the number of automorphism of $C$, and $t(C)=t(C,\mathcal H_{\pi^*}(\mathtt U_1;\mathtt W_1))$ be the number of (unlabeled) embeddings of $C$ into $\mathcal H_{\pi^*}(\mathtt U_1;\mathtt W_1)$. As argued in the proof of Proposition 3.7 therein and by utilizing condition~\eqref{eq-density-upper-bound}, we conclude that the left hand side of \eqref{eq-goal-bound-1} is bounded from above by $1+o(1)$ multiplies
	\begin{equation}\label{eq-product}
		\prod_{C\in \bigcup\mathfrak C_r}\sum_{x=0}^n\Big(\frac{2e\operatorname{Aut}(C)t(C)}{np^{\alpha^{-1}-\varepsilon}}\Big)^x\times \prod_{T\in \bigcup\mathfrak T_r}\sum_{y=0}^n \Big(\frac{2ep\operatorname{Aut}(T)t(T)}{(np)^{|T|}}\Big)^y\,,
	\end{equation}
	where $x$ (resp. $y$) corresponds to the number of vertex-disjoint copies of $C$ (resp. $T$) in $H(\mathcal J_\pi)\subset \mathcal H_{\pi^*}(\mathtt U_1;\mathtt W_1)$, which is trivially upper-bounded by $n$. 
	It can be shown as in \cite[Lemma 3.5]{DD23a} that whenever $\mathcal H_{\pi^*}(\mathtt U_1;\mathtt W_1)$ is admissible and satisfies \eqref{eq-density-upper-bound}, it holds
	\[
	\sum_{C\in \bigcup \mathfrak C_r}\operatorname{Aut}(C)t(C)\le r^3(2^{\alpha^{-1}+1}\log n)^r\,,
	\]
	and
	\[
	\sum_{T\in \mathfrak T_r}\operatorname{Aut}(T)t(T)\le n(4\log n)^{2(r-1)}\,.
	\]
	In light of these bounds, recalling $p=n^{-\alpha+o(1)}$, we can conclude the first part of \eqref{eq-product} is $1+o(1)$, and so does the second part provided $\alpha<1$; see \cite[Proposition 3.7]{DD23a} for details. Therefore, the left hand side of \eqref{eq-goal-bound-1} is $1+o(1)$ whenever $\alpha<1$, as desired. 
	
	For the remaining case $\alpha=1$, it suffices to show the second term in \eqref{eq-product} is $\exp(O_\varepsilon(n))$. Note that condition \eqref{eq-density-upper-bound} implies that any subgraph of $\mathcal H_{\pi^*}(\mathtt U_1;\mathtt W_1)$ has density no more than $1-\varepsilon$. This means $\mathcal H_{\pi^*}(\mathtt U_1;\mathtt W_1)$ is a forest consists of trees with no more than $\varepsilon^{-1}$ many vertices. Clearly, this implies $t(T)\neq 0$ for only for $T\in \bigcup_{r=1}^{\lfloor1/\varepsilon\rfloor} \mathfrak T_r$. Additionally, since for every $T\in \bigcup_{r=1}^{\lfloor 1/\varepsilon\rfloor} \mathfrak C_r$, $\operatorname{Aut}(T)=O_\varepsilon(1),t(T)=O_\varepsilon(n)$, we have (as $np\ge nps^2=\lambda\ge 1$)
	\[
	\sup_{T \in \bigcup_{r=1}^{\lfloor 1/\varepsilon\rfloor} \mathfrak C_r}\frac{2ep\operatorname{Aut}(T)t(T)}{(np)^{|T|}}=O_\varepsilon\left(\frac{np}{(np)^{|T|}}\right)=O_\varepsilon(1)\,.
	\]
	This means the second term in \eqref{eq-product} is indeed $\exp(O_\varepsilon(n))$, as desired.  The proof of Proposition~\ref{prop-dp/dq-bound} is now completed.
	
	\subsection{Proof of Proposition~\ref{prop-posterior-bound-under-Q}}
	
	Finally we prove Proposition~\ref{prop-posterior-bound-under-Q}. For $\sigma\in \operatorname{B}(\mathtt U_1,\mathtt U_2)$ and $(G_1^{\mathtt U_1},G_2^{\mathtt U_2})\sim \mathcal Q_{\mathtt U_1,\mathtt U_2}$, abbreviate $$\mathcal R(\sigma,G_1^{\mathtt U_1},G_2^{\mathtt U_2})\equiv\frac{\mathcal P_{\mathtt U_1,\mathtt U_2}^{\ddagger,*}[\pi^*\mid_{\mathtt U_1}=\sigma,G_1^{\mathtt U_1},G_2^{\mathtt U_2}]}{{\mathcal Q}_{\mathtt U_1,\mathtt U_2}[G_1^{\mathtt U_1},G_2^{\mathtt  U_2}]}\,.$$ 
	Our goal is to upper-bound the following quantity:
	\[
	M_1\equiv \sup_{\sigma_0\in \operatorname{B}(\mathtt U_1,\mathtt U_2)}\sum_{\substack{\sigma\in \operatorname{B}(\mathtt U_1,\mathtt U_2)\\\operatorname{ov}(\sigma,\sigma_0)\ge \varepsilon n}}\mathcal R(\sigma,G_1^{\mathtt U_1},G_2^{\mathtt U_2})\,.
	\]
	Henceforth we may assume that $|\mathtt U_1|\ge \varepsilon n$, otherwise $M_1=0$.
	For $V_1,V_2\subset [n]$ with $|V_1|\le |V_2|$, we use $\operatorname{I}(V_1,V_2)$ to denote the set of injections from $V_1$ to $V_2$. We turn to consider 
	\[
	M_2\equiv \sup_{\substack{B\subset \mathtt U_1,|B|\ge \varepsilon n\\\sigma_0\in \operatorname{I}(B,\mathtt U_2)}}\sum_{\substack{\sigma\in \operatorname{B}(\mathtt U_1,\mathtt U_2)\\\sigma\mid_B=\sigma_0}}\mathcal R(\sigma,G_1^{\mathtt U_1},G_2^{\mathtt U_2})\,.
	\]
	Clearly $M_1\le \binom{|\mathtt U_1|}{\varepsilon  n}M_2\le 2^nM_2$.
    
    We need yet another layer of reduction, and to this end we introduce the good sets as below. The good sets enjoy some desirable combinatorial properties that will become crucial in later proof. Recall the large integer $C$ defined as in \eqref{eq-C}.
	
	\begin{definition}
		Let $\mathcal H$ be a graph on $[n]$, and let $\operatorname{d}_{\mathcal H}$ denote the graph distance on $\mathcal H$. We say a set $A\subset [n]$ is a \emph{good set} of $\mathcal H$, if the following two properties hold:\\
		\noindent(i) \emph{Well-separated in $\mathcal H$}: for any two vertices $u,v\in A$, $\operatorname{d}_{\mathcal H}>2C+2$.\\
		\noindent(ii) \emph{Faraway from short cycles}: For any cycle $\mathcal C$ in $\mathcal H$ of length $\le C$, $\operatorname{d}_\mathcal H(A,\mathcal C)>C$.
	\end{definition}
	
	The following lemma shows that in an admissible graph,  any set of linear size has a large good subset.
	\begin{lemma}
		Fix an arbitrary $\varepsilon>0$. Assume that $n$ is large enough and $\mathcal H=([n],\mathcal E)$ is admissible. For any subset $B\subset[n]$ with size at least $\varepsilon n$, there exists $A\subset B$ which is a good set of $\mathcal H_{\pi^*}(\mathtt U_1;\mathtt W_1)$, and satisfies $|A|=K_n\equiv \lceil 2n/(\alpha\varepsilon\sqrt{\log n})\rceil$. 
	\end{lemma}
	
	\begin{proof}
First, we delete any vertex $ v \in B $ such that either $ v $ is within distance $ C $ of a cycle in $ \mathcal{H} $ with size at most $ C $, or it is within distance $ 2C + 2 $ of a vertex with degree larger than $ D_n $. By admissibility, it is straightforward to see that the number of deleted vertices is at most
\[ 
\sum_{k=3}^Ck(\log n)^k\cdot (\log n)^{C} + n\exp(-D_n) \le \varepsilon n / 2 \,.
\] 
For the remaining at least $ \varepsilon n / 2 $ vertices in $ B $, they are all far from short cycles, and their $ (2C + 2) $-neighborhoods have size bounded above by $ D_n^{2C + 2} $. Given this, a simple iterative pick-and-delete algorithm yields a good set $ A' \subset B $ with:
\[
|A'| \ge \frac{\varepsilon n}{4D_n^{2C+2}} \ge \frac{2n}{\alpha \varepsilon \sqrt{\log n}}\,,
\]
and we can choose $ A $ as an arbitrary $ K_n $-subset of $ A' $.
	\end{proof}
	
	Define
	\[
	M_3\equiv \sup_{\substack{A\subset \mathtt U_1,|A|=K_n\\\sigma_0\in\operatorname{I}(A,\mathtt U_2) }}\sum_{\substack{\sigma\in \operatorname{B}(\mathtt U_1,\mathtt U_2)\\\sigma\mid_A=\sigma_0}}\mathcal R(\sigma,G_1^{\mathtt U_1},G_2^{\mathtt U_2})\mathbf{1}\{A\text{ is a good set of }\mathcal H_{\pi}(\mathtt U_1;\mathtt W_1)|!\}\,,
	\]
	where $\pi=\pi(\pi^*,\sigma)\in \operatorname{S}_n$ is the extension of $\sigma$ and $\pi^*\mid_{\mathtt W_1}$, and $\mathcal H_{\pi}(\mathtt U_1;\mathtt W_1)$ is the $\pi$-intersection graph of $G_1^{\mathtt U_1},G_2^{\mathtt U_2}$. 
	Since $\mathcal H_{\pi^*}(\mathtt U_1;\mathtt W_1)$ is almost surely admissible under $\mathcal P_{\mathtt U_1,\mathtt U_2}^{\ddagger,*}$, the above lemma yields $M_2\le \binom{\varepsilon n}{K_n}M_3\le 2^nM_3$. Therefore, Proposition~\ref{prop-posterior-bound-under-Q} follows upon showing that
	\begin{equation}\label{eq-M3}
		\mathcal Q_{\mathtt U_1,\mathtt U_2}[M_3\ge \exp(-4n)]\le \exp(-n\sqrt{\log n})\,.
	\end{equation}
	
	Using the simple fact that the maximum of a finite set $S\subset \mathbb R$ is no more than $\big(\sum_{s\in S}s^2\big)^{1/2}$, \eqref{eq-M3} follows from Markov's inequality provided we can show that
	\[
	\mathbb{E}_{(G_1^{\mathtt U_1},G_2^{\mathtt U_2})\sim \mathcal Q_{}}\sum_{\substack{A\subset \mathtt U_1,|A|=K_n\\\sigma_0\in \operatorname{I}(A,\mathtt U_2)}}\Bigg(\sum_{\substack{\sigma\in \operatorname{B}(\mathtt U_1,\mathtt U_2)\\\sigma\mid_A=\sigma_0}}\mathcal R(\sigma,G_1^{\mathtt U_1},G_2^{\mathtt U_2})\mathbf{1}\{A\text{ is a good set of }\mathcal H_{\pi}(\mathtt U_1;\mathtt W_1)|!\}\Bigg)^2
	\]
	is bounded from above by $\exp(-n\sqrt{\log n}-8n)$. There are $\binom{n}{K_n}\cdot |\mathtt U_1|(|\mathtt U_1|-1)\cdots(|\mathtt U_1|-K_n+1)\le \exp(n+K_n{\log n)}$ many choices of $(A,\sigma)$. Since the second moment of the latter sum under $\mathcal Q_{\mathtt U_1,\mathtt U_2}$ is invariant of such a choice, it suffices to show for each fixed pair $(A,\sigma)$, 
	\begin{equation}
		\begin{aligned}
			\mathbb{E}_{(G_1^{\mathtt U_1,\mathtt U_1})\sim \mathcal Q_{\mathtt U_1,\mathtt U_2}}&\Bigg(\sum_{\substack{\sigma\in \operatorname{B}(\mathtt U_1,\mathtt U_2)\\\sigma\mid_A=\sigma_0}}\mathcal R(\sigma,G_1^{\mathtt U_1},G_2^{\mathtt U_2})\mathbf{1}\{A\text{ is a good set of }\mathcal H(\mathtt U_1,\mathtt U_2)\}\Bigg)^2\\
			&\qquad\qquad\le \exp(-n\sqrt{\log n}-K_n{\log n}-9n)\,.
		\end{aligned}
	\end{equation}
	The rest of the paper is dedicated to prove this estimate. 
	
	Recall that for $\pi=\pi(\sigma)\in \operatorname{S}_n$,
	\[
	\mathcal P^{\ddagger,*}_{\mathtt U_1,\mathtt U_2}[\pi,G_1^{\mathtt U_1},G_2^{\mathtt U_2}]\le \frac{\exp(O_\lambda(n))}{|\mathtt U_1|!}\cdot\mathcal P^*_{\mathtt U_1,\mathtt U_2}[\pi,G_1^{\mathtt U_1,\mathtt U_2}]\cdot\mathbf{1}\{(\pi,G_1^{\mathtt U_1},G_2^{\mathtt U_2})\in \mathcal A\}\,,
	\]
	then it follows from a straightforward calculation that
	\[
	\mathcal R(\sigma,G_1^{\mathtt U_1},G_2^{\mathtt U_2})\le \frac{\exp(O_\lambda(n))}{|\mathtt U_1|!}\cdot P^{|\mathcal E_\pi|}Q^{|E(G_1^{\mathtt U_1})|+|E(G_2^{\mathtt U_2})|}R^{\binom{n}{2}-\binom{|\mathtt W_1|}{2}}\cdot \mathbf{1}\{(\pi,G_1^{\mathtt U_1},G_2^{\mathtt U_2})\in \mathcal A\}\,,
	\]
	where $\mathcal E_\pi$ is the edge set of $\mathcal H_{\pi}(\mathtt U_1;\mathtt W_1)$, and 
	\begin{equation}\label{eq-PQR}
		P\equiv \frac{1-2ps+ps^2}{p(1-s)^2},\quad Q\equiv \frac{(1-s)(1-ps)}{1-2ps+ps^2},\quad R\equiv \frac{1-2ps+ps^2}{(1-ps)^2}\,.
	\end{equation}
	Therefore, the left hand side of \eqref{eq-M3} is bounded by $\frac{\exp(O_\lambda(n))}{(|\mathtt U_1|!)^2}$ multiplies
	\[
	\mathbb{E}_{ \mathcal Q_{\mathtt U_1,\mathtt U_2}}\Bigg(\sum_{\substack{\sigma\in \operatorname{B}(\mathtt U_1,\mathtt U_2)\\\sigma\mid_A=\sigma_0}}P^{|\mathcal E_\pi|}Q^{|E(G_1^{\mathtt U_1})|+|E(G_2^{\mathtt U_2}|)|}R^{\binom{n}{2}-\binom{|\mathtt W_1|}{2}}\mathbf{1}\{(\pi,G_1^{\mathtt U_1},G_2^{\mathtt U_2})\in \mathcal A,A\text{ is good}\}\Bigg)^2\,.
	\]
	
	We denote $(\mathcal H_\pi)_A$ for the induced subgraphs of $\mathcal H_{\pi}(\mathtt U_1;\mathtt W_1)$ on $\mathtt W_1\cup A$, and let $(\mathcal H_\pi)^A$ be the subgraph of $\mathcal H_{\pi}(\mathtt U_1;\mathtt W_1)$ obtained from deleting all the edges in $\mathcal H_A$. Beware that $(\mathcal H_\pi)_A,(\mathcal H_\pi)^A$ both depend on the choices of $G_1^{\mathtt U_1},G_2^{\mathtt U_2}$, though we suppress the dependence form the notation. Similarly, we write $(G_1)_A,(G_2)_{\sigma_0(A)}$ for the induced subgraphs of $G_1,G_2$ on $A,\sigma_0(A)$, respectively, and let $(G_1)^A=G_1^{\mathtt U_1}\setminus (G_1)_A,(G_2)^{\sigma_0(A)}=G_2^{\mathtt U_2}\setminus (G_2)_A$ (here the set-minus means deleting the edges). Define the events
	\begin{align*}
		&\mathcal G_A^0=\{|\mathcal E_{\sigma_0}(A)|\le (\alpha^{-1}-\varepsilon)K_n\}\,,\\
		&\mathcal G_A^1=\{(\mathcal H_\pi)^A\text{ is admissible}\}\,,\\
		&\mathcal G_A^2=\{A\text{ is a good set of }(\mathcal H_\pi)^A\}.
	\end{align*}
	Clearly, $\mathcal G_A^0$ is independent of $\mathcal G_A^1\cap \mathcal G_A^2$, and
	
	\[\{(\pi,G_1^{\mathtt U_1},G_2^{\mathtt U_2})\in \mathcal A, A\text{ is good of }\mathcal H_{\pi}(\mathtt U_1;\mathtt W_1)\}\subset \mathcal G_A^0\cap \mathcal G_A^1\cap \mathcal G_A^2\,,\]
	and the above expression is upper-bounded by
	\begin{align}
		\nonumber	&\ \mathbb{E}_{\mathcal Q_{\mathtt U_1,\mathtt U_2}}P^{2|\mathcal E_{\sigma_0}(A)|}Q^{|E((G_1)_A)|+|E((G_2)_{\sigma_0(A)})|}R^{\binom{|\mathtt W_1|+K_n}{2}-\binom{|\mathtt W_1|}{2}}\mathbf{1}\{\mathcal G_A^0\}\\
		\nonumber	\times&\ \Bigg(\sum_{\substack{\sigma\in \operatorname{B}(\mathtt U_1,\mathtt U_2)\\\sigma\mid_A=\sigma_0}}P^{|(\mathcal E_\pi)^A|}Q^{|E((G_1)^A)|+|E((G_2)^A)|}R^{\binom{n}{2}-\binom{|\mathtt W_1|+K_n}{2}}\mathbf{1}\{\mathcal G_A^1\cap \mathcal G_A^2\}\Bigg)^2\\
		\label{eq-term-1}	=&\ \mathbb{E}_{\mathcal Q_{\mathtt U_1,\mathtt U_2}}P^{2|\mathcal E_{\sigma_0}(A)|}Q^{|E((G_1)_A)|+|E((G_2)_{\sigma_0(A)})|}R^{\binom{|\mathtt W_1|+K_n}{2}-\binom{|\mathtt W_1|}{2}}\mathbf{1}\{\mathcal G_A^0\}\\\times&\ 
		\label{eq-term-2}	\mathbb{E}_{\mathcal Q_{\mathtt U_1,\mathtt U_2}}\Bigg(\sum_{\substack{\sigma\in \operatorname{B}(\mathtt U_1,\mathtt U_2)\\\sigma\mid_A=\sigma_0}}P^{|(\mathcal E_\pi)^A|}Q^{|E((G_1)^A)|+|E((G_2)^A)|}R^{\binom{n}{2}-\binom{|\mathtt W_1|+K_n}{2}}\mathbf{1}\{\mathcal G_A^1\cap \mathcal G_A^2\}\Bigg)^2\,,
	\end{align}
	where the equality follows from independence. 
	
	Due to the presence of the indicator $\mathbf{1}\{\mathcal G_A^0\}$, it is straightforward to check that \eqref{eq-term-1} is upper-bounded by $\exp(O_\lambda(n))\cdot \exp((1-\alpha\varepsilon+o(1))K_n\log n)$; see \cite[Proposition A.2]{DD23b} for an analogous calculation. We claim that \eqref{eq-term-2} is upper-bounded by $\exp(O_{\lambda,\varepsilon}(n))\cdot \big((|\mathtt U_1|-K_n)!\big)^2$. Provided this is true, we conclude the left hand side of \eqref{eq-M3} is upper bounded by
	\[
	\frac{\exp(O_\lambda(n))}{(|\mathtt U_1|!)^2}\times \exp(O_\lambda(n))\exp((1-\alpha\varepsilon+o(1))K_n\log n)\times \exp(O_{\lambda,\varepsilon}(n))\big((|\mathtt U_1|-K_n)!\big)^2\,,
	\]
	which equals to $\exp(O_{\lambda,\varepsilon}(n))\cdot \exp(-(1+\alpha\varepsilon+o(1)) K_n\log n)$. By the choice of $K_n$, this is much smaller than $\exp(-n\sqrt{\log n}-K_n\log n-9n)$ for large $n$, and thus \eqref{eq-M3} holds. 
	
	Finally we verify the claimed upper bound on \eqref{eq-term-2}. Conceptually speaking, the magic here is the expression can be interpreted as the second moment of another kind of likelihood ratio $\dif \overline{\mathcal P}_{\mathtt U_1,\mathtt U_2}/\dif \mathcal Q_{\mathtt U_1,\mathtt U_2}$, where $\overline{\mathcal P}_{\mathtt U_1,\mathtt U_2}$ is just a slight variant of $\mathcal P_{\mathtt U_1,\mathtt U_2}^{\ddagger}$ (depending on $A$ and $\sigma_0$). In light of this, it is not surprising that the we can follow the steps in Section~\ref{subsec-B.1} to obtain an effective upper bound. 
	
	Indeed, the proof proceeds essentially the same as that in \cite[Appendix A]{DD23b}. Following the arguments therein, together with some technical refinements in Section~\ref{subsec-B.1}, we obtain that \eqref{eq-term-2} is bounded from above by
	\[
	\big((|\mathtt U_1|-K_n)!\big)^2\cdot\exp(O_\lambda(n))\times \mathbb{E}_{(\pi^*,G_1^{\mathtt U_1,\mathtt U_2})\sim \overline{\mathcal P}_{\mathtt U_1,\mathtt U_2}}\frac{1}{(|\mathtt U_1|-K_n)!}\sum_{\substack{\sigma\in \operatorname{B}(\mathtt U_1,\mathtt U_2)\\\sigma\mid_A=\sigma_0}}p^{-|\mathcal J_\pi|}. 
	\]
	Here we are abusing notation as we have not defined $\overline{\mathcal P}_{\mathtt U_1,\mathtt U_2}$ formally. This can be done in a formal way (akin to \cite[Definition A.3]{DD23b}), but here the only property we will use for $\overline{\mathcal P}_{\mathtt U_1,\mathtt U_2}$ is that almost surely with $(\pi^*,G_1^{\mathtt U_1},G_2^{\mathtt U_2})\sim \overline{\mathcal P}_{\mathtt U_1,\mathtt U_2}$, $\mathcal H_{\pi^*}(\mathtt U_1;\mathtt W_1)$ is admissible satisfying condition \eqref{eq-union-bound-over-x}, and it has $A$ as a good set. In light of this, we will bound the term in expectation deterministically under such a condition. 
	
	For any $r\ge 2$, define $\mathfrak H_r$ as the set of connected subgraphs of $\mathcal H_{\pi^*}(\mathtt U_1;\mathtt W_1)$ with $r$ vertices in $\mathtt{U}_1\setminus A$ and at least one vertex in $A$. For each $H\in \mathfrak H_r$, fix a vertex $v_H\in V(H)\cap A$, and denote $\operatorname{Aut}(v_H,H)$ as the number of automorphisms of $H$ preserving $v_H$. Also recall $\mathfrak C_r,\mathfrak T_r$ as defined in Section~\ref{subsec-B.1}. Continue following the proof in \cite[Appendix A]{DD23a}, we see the averaged sum is upper-bounded by
	\[
	\prod_{C\in \bigcup\mathfrak C_r}\sum_{x=0}^n\Big(\frac{2e\operatorname{Aut}(C)t(C)}{np^{\alpha^{-1}-\varepsilon}}\Big)^x\prod_{T\in \bigcup\mathfrak T_r}\sum_{y=0}^n \Big(\frac{2ep\operatorname{Aut}(T)t(T)}{(np)^{|T|}}\Big)^y\prod_{H\in \cup\mathfrak H_r}\Big(1+\frac{\operatorname{Aut}(v_H,H)}{n^{|V(H)\setminus A|}(p/2e)^{|E(H)|}}\Big)\,.
	\]
	We have shown in the last subsection that the first two terms are $\exp(O_\varepsilon(n))$ provided $\mathcal H_{\pi^*}(\mathtt U_1;\mathtt W_1)$ is admissible and satisfies \eqref{eq-density-upper-bound}. It remains to deal with the last term. First, admissibility yields that for any $r\ge 2$,
	\[
	\sum_{H\in \mathfrak H_r}\operatorname{Aut}(v_H,H)\le 2K_n(2^{\alpha^{-1}}\log n)^{4r}\,,
	\]
	see \cite[Lemma A.11]{DD23b} for details. Therefore, the last term in the above product is upper-bounded by (using that $1+x\le \exp(x)$)
	\[
	\exp\left(\sum_{r=1}^{|\mathtt U_1|-K_n}\frac{2K_n(2^{\alpha^{-1}}\log n)^{4r}}{n^r}\cdot\max_{H\in \mathfrak H_r}(2ep^{-1})^{|E(H)|}\right)\,.
	\]
	Using the goodness of $A$, we now verify that 
	\[
	\sum_{r=1}^{|\mathtt U_1|-K_n}\frac{2K_n(2^{\alpha^{-1}}\log n)^{4r}}{n^r}\cdot\max_{H\in \mathfrak H_r}(2ep^{-1})^{|E(H)|}=o(n)\,.
	\]
	We divide the sum into two parts.\\
	\noindent \emph{Part 1}: $1\le r\le C$. In this case, we have any $H\in \mathfrak H_r$ is a tree, since otherwise there is a  vertex in $A$ has distance no more than $C$ to a short cycle in $\mathcal H_{\pi^*}(\mathtt U_1;\mathtt W_1)$. It follows that
	\[
	\sum_{r=1}^C\frac{2K_n(2^{\alpha^{-1}}\log n)^{4r}}{n^r}\cdot\max_{H\in \mathfrak H_r}(2ep^{-1})^{|E(H)|}\le \frac{2K_n}{n}\cdot\sum_{r=1}^C\left(\frac{(2^{1+4\alpha^{-1}}e\log n)^4}{np}\right)^{r-1}\,,
	\]
	which is at most $O((\log n)^C)=o(n)$.\\
	\noindent \emph{Part 2}: $C<r\le |\mathtt U_1|-K_n$. Since $A$ is a good set and hence an independent set of $\mathcal H_{\pi^*}(\mathtt U_1;\mathtt W_1)$, we have for any $H\in \mathfrak H_r$,  $|E(H)|=|E(H\setminus A)|+|E(A;H)|$. By condition \eqref{eq-union-bound-over-x}, $|E(H\setminus A)|\le (\alpha^{-1}-\varepsilon)r$. By goodness of $A$, each vertex in $V(H)\setminus A$ connect to at most one vertex in $A$ (since otherwise two vertices in $A$ have graph distance no more than $2$ on $\mathcal H_{\pi^*}(\mathtt U_1;\mathtt W_1)$). Furthermore, any two vertices in $V(H)\setminus A$ that have neighbors in $A$ are at least $(2C+1)$-faraway from each other on $H$. This implies the $C$-neighborhoods in $H$ of these vertices are disjoint. Since $H$ is connected and has more than $C$ vertices, the size of any $C$-neighborhood in $H$ is at least $C$. Hence, the number of such vertices is no more than $r/C$. Altogether we obtain $|E(H)|\le (\alpha^{-1}-\varepsilon+C^{-1})r$, and so
	\[
	\sum_{r=C+1}^{|\mathtt U_1|-K_n}\frac{2K_n(2^{\alpha^{-1}}\log n)^{4r}}{n^r}\cdot\max_{H\in \mathfrak H_r}(2ep^{-1})^{|E(H)|}\le 2K_n\cdot\sum_{r=C+1}^{|\mathtt U_1|-K_n}\left( \frac{2^{1+4\alpha^{-1}}e(\log n)^4}{(np^{\alpha^{-1}-\varepsilon+C^{-1}})}\right)^r\,.
	\]
	Recalling $p=n^{-\alpha+o(1)}$, by our choice of $C$ in \eqref{eq-C} we have for large enough $n$, the above expression is upper-bounded by $2K_n=o(n)$. 
	
	We have verified the sum is $o(n)$. Therefore, the third term is $\exp(o(n))$ and thus the claim follows. The proof of Proposition~\ref{prop-posterior-bound-under-Q} is completed.

	\bibliographystyle{alpha}
	\small
	\bibliography{bib}

\end{document}